\documentclass[reqno,oneside,12]{amsart}

\usepackage[a4paper, top=3.5cm, bottom=3.5cm, left = 3.5cm, right =3.5cm]{geometry}

\usepackage[utf8]{inputenc}
\usepackage[T1]{fontenc}
\usepackage[svgnames,hyperref]{xcolor}
\usepackage{amssymb}
\usepackage{lmodern}
\usepackage{bbm} 

\usepackage{tikz}
\usetikzlibrary{decorations.markings, arrows.meta, matrix, positioning, calc, angles,quotes}
\usepackage{tikz-cd}
\tikzset{
  typeone/.style={circle,draw=black,thick,inner sep=0pt,minimum size=2.3mm},
  typetwo/.style={circle,draw=black,fill=black,thin,inner sep=0pt,minimum size=2.3mm},
  typethree/.style={rectangle,draw=black,fill=black,thin,inner sep=0pt,minimum size=2.3mm}
}
\tikzset{mid arrow/.style={white!65!black, decoration={
  markings,
  mark=at position .5 with {\arrow{>}}},postaction={decorate}}}
\definecolor{lightblue}{rgb}{0.8,0.85,1}
\newcommand{\trip}[3]{{[#1,\!#2,\!#3]}}
\newcommand{\bary}[3]{barycentric cs:111={3*(#1)},110={2*(#2)},100={#3}}
\newcommand{\tbary}[3]{barycentric cs:100={#1},010={#2},001={#3}}

\usepackage[english, french]{babel}
\newcommand{\newabstract}[1]{%
  \par\bigskip
  \csname otherlanguage*\endcsname{#1}%
  \csname captions#1\endcsname
  \item[\hskip\labelsep\scshape\abstractname.]
}

\usepackage[shortlabels]{enumitem}
\setlist[enumerate]{label=\rm{(\roman*)}}
\setlist[itemize]{label=\raisebox{0.25ex}{\tiny$\bullet$}}

\theoremstyle{plain}
\newtheorem{maintheorem}{Théorème} 
\newtheorem{maincorollary}[maintheorem]{Corollaire} 
\newtheorem{theorem}{Théorème}[section]
\newtheorem{corollary}[theorem]{Corollaire}
\newtheorem{proposition}[theorem]{Proposition}
\newtheorem{lemma}[theorem]{Lemme}

\theoremstyle{definition}
 
\newtheorem{definition}[theorem]{Définition}
\newtheorem{question}[theorem]{Question}

\newtheorem{example}[theorem]{Exemple}
\newtheorem{remark}[theorem]{Remarque}

\makeatletter
\let\c@equation\c@theorem
\makeatother

\newcommand{\R}{\mathbb{R}}

\newcommand{\F}{\mathbb{F}}

\newcommand{\A}{{\mathbbm{k}}}

\newcommand{\W}{\Pi}
\newcommand{\PW}{\nabla}
\newcommand{\K}{\mathbbm{k}}

\newcommand{\Al}{\mathcal{A}}
\newcommand{\Bl}{\mathcal{B}}
\newcommand{\Cl}{\mathcal{C}}
\newcommand{\Dl}{D}
\newcommand{\Fl}{\mathcal{F}}
\newcommand{\Gl}{\mathcal{G}}
\newcommand{\Hl}{\mathcal{H}}
\newcommand{\Kl}{\mathcal{K}}

\newcommand{\Vl}{\mathcal{V}}
\newcommand{\Xl}{\mathbf{X}}

\DeclareMathOperator{\GL}{GL}
\DeclareMathOperator{\Diff}{Diff}
\DeclareMathOperator{\SL}{SL}

\DeclareMathOperator{\Aut}{Aut}

\DeclareMathOperator{\Stab}{Stab}

\DeclareMathOperator{\CAT}{CAT}
\DeclareMathOperator{\Tame}{Tame}

\DeclareMathOperator{\car}{car}
\DeclareMathOperator{\supp}{Supp}
\DeclareMathOperator{\Fix}{Fix}
\DeclareMathOperator{\mult}{mult}

\newcommand{\TA}{\Tame(\A^n)}
\newcommand{\Comp}{\mathcal{C}}
\newcommand{\CComp}{\mathbf{Y}}
\newcommand{\eps}{\varepsilon}
\renewcommand{\phi}{\varphi}
\renewcommand{\setminus}{\smallsetminus}
\newcommand{\id}{\text{\rm id}}
\newcommand{\Ap}{\mathbf E}
\newcommand{\DD }{\mathbf D}
\newcommand{\V}{\mathbf V}
\renewcommand{\le}{\leqslant}
\renewcommand{\ge}{\geqslant}
\renewcommand{\leq}{\leqslant}
\renewcommand{\geq}{\geqslant}

\newcommand{\Td}{\Aut(\A^2)}
\newcommand{\Tt}{\Tame(\A^3)}
\newcommand{\pr}{\mathrm{pr}}
\newcommand{\wgt}{\rho_+}

\newcommand{\mycolor}{Navy}
\usepackage[pdfauthor={S. Lamy \& P. Przytycki}, colorlinks, linktocpage, citecolor = \mycolor,
linkcolor = \mycolor, urlcolor = \mycolor]{hyperref}
\usepackage[all]{hypcap} 

\addto\captionsfrench{}

\title{Presqu'un immeuble pour le groupe des automorphismes modérés}
\date{\today}
\author{St\'ephane Lamy \& Piotr Przytycki}
\thanks{S.L. partiellement supporté par le projet ``BirPol'' ANR-11-JS01-004-01 et par le CRM, UMI CNRS Montréal.}
\address{Institut de Mathématiques de Toulouse UMR 5219, Université de Toulouse, UPS
F-31062 Toulouse Cedex 9, France}
\email{slamy@math.univ-toulouse.fr}

\address{Dept. of Math. \& Stats., McGill University, Montreal, Quebec, Canada H3A 0B9}
\email{piotr.przytycki@mcgill.ca}
\thanks{P.P. partiellement supporté par NSERC, FRQNT, et National Science Centre, Poland, UMO-2015/\-18/\-M/\-ST1/\-00050.}


\begin{document}

\begin{abstract}
Inspirés par l'immeuble de Bruhat--Tits du groupe $\SL_n(\F)$, pour $\F$ un corps valué, nous construisons un espace métrique complet $\Xl$ sur lequel agit le groupe $\TA$ des automorphismes modérés de l'espace affine.
Les points de $\Xl$ sont certaines valuations monomiales, et $\Xl$ admet une structure naturelle de CW-complexe euclidien de dimension $n-1$.
Quand $n = 3$, et pour $\K$ de caractéristique zéro, nous prouvons que $\Xl$ a courbure négative ou nulle et est simplement connexe, et par conséquent est un espace $\CAT(0)$.
En application nous obtenons la linéarisabilité des sous-groupes finis de $\Tame(\A^3)$.

\newabstract{english}
Inspired by the Bruhat--Tits building of $\SL_n(\F)$, for $\F$ a valuation field, we construct a complete metric space $\Xl$ with an action of the tame automorphism group of the affine space $\TA$.
The points in $\Xl$ are certain monomial valuations, and $\Xl$ admits a natural structure of Euclidean CW-complex of dimension $n-1$.
When $n = 3$, and for $\K$ of characteristic zero, we prove that $\Xl$ has non-positive curvature and is simply connected, hence is a $\CAT(0)$ space.
As an application we obtain the linearizability of finite subgroups in $\Tame(\A^3)$.
\end{abstract}

\maketitle

\setcounter{tocdepth}{3}

\section{Introduction}

Soit $\K$ un corps, et $n \ge 2$.
On note $\Aut(\A^n)$ le groupe des automorphismes polynomiaux de l'espace affine de dimension $n$ sur $\K$.
Un système de coordonnées $(x_1, \dots, x_n)$ étant fixé, un élément $g \in \Aut(\A^n)$ s'écrit
\[
g \colon (x_1, \dots, x_n) \mapsto (g_1, \dots, g_n)
\]
où les $g_i$ sont des polynômes en les $x_i$, et l'inverse de $g$ admet une écriture de la même forme.

Le groupe $\Tame(\A^n)$ des \emph{automorphismes modérés} de l'espace affine est le sous-groupe de $\Aut(\A^n)$ engendré par le groupe linéaire $\GL_n(\K)$ et par les automorphismes élémentaires, ou «transvections polynomiales», de la forme
\[
(x_1, \dots, x_n) \mapsto (x_1 + P(x_2, \dots, x_n), x_2, \dots, x_n).
\]
Nous étudions la structure du groupe modéré, en particulier en dimension $n = 3$ et sur un corps de base $\K$ de caractéristique nulle.
Après avoir obtenu l'existence de nombreux sous-groupes normaux dans \cite{LP}, nous nous intéressons dans le présent travail à la question de la classification de ses sous-groupes finis à conjugaison près.
Notre motivation à long terme est d'établir un panorama des propriétés que l'on peut attendre du groupe entier des automorphismes polynomiaux, voire même du groupe de Cremona, qui est le groupe des transformations birationnelles (et non plus seulement bipolynomiales) de l'espace affine.
Nous renvoyons aux introductions de \cite{BFL} et \cite{LP} pour plus de détails sur l'articulation entre ces différents groupes, et un historique des développements récents.
Nous nous concentrons dans cette introduction sur les aspects spécifiques à ce travail.

Un sous-groupe $G$ de $\Aut(\A^n)$ est dit \emph{linéarisable} s'il existe $\phi \in \Aut(\A^n)$ tel que $\phi G \phi^{-1} \subset \GL_n(\K)$.
Il est connu que, pour un corps $\K$ de caractéristique nulle, tout sous-groupe fini $G$ de $\Aut(\A^2)$ est linéarisable.
Rappelons l'argument, qui se résume simplement, et qui servira de modèle pour passer en dimension 3.
On considère l'action de $G$ sur l'arbre de Bass--Serre associé à la structure de produit amalgamé pour $\Tame(\K^2) = \Aut(\A^2)$ (c'est le théorème de Jung--van der Kulk, voir par exemple \cite{LamyJung}).
Celle-ci admet toujours au moins un point fixe, ainsi $G$ est conjugué à un sous-groupe de l'un des facteurs du produit amalgamé.
Finalement, les sous-groupes finis de ces deux facteurs sont linéarisables par un critère de moyennisation général, voir lemme~\ref{lem:abstract linearization}.
En dimension plus grande, l'arbre de Bass--Serre se généralise en un complexe simplicial $\Comp_n$ de dimension $n-1$ sur lequel le groupe modéré agit avec pour domaine fondamental un simplexe.
C'est ce complexe $\Comp_3$ que nous avons utilisé dans notre précédent travail \cite{LP} pour obtenir l'hyperbolicité acylindrique, qui implique en particulier l'existence de sous-groupes normaux dans $\Tame(\A^3)$ \cite{DGO}.
C'est également avec une variante de cette construction qu'a été établi dans \cite{BFL} un résultat analogue de linéarisation des sous-groupes finis pour le groupe $\Tame(V)$ des automorphismes modérés d'une quadrique affine $V \subset \K^4$.
Dans ce dernier contexte, le complexe  sur lequel agit naturellement $\Tame(V)$ est un complexe carré $\CAT(0)$, cette propriété de courbure négative ou nulle assurant l'existence d'un point fixe global pour toute action de groupe fini.
Pour étendre ces résultats au groupe $\Tame(\A^3)$, le problème que nous avons dû contourner est que l'étude du link des sommets montre que les triangles de $\Comp_3$ ne peuvent pas être munis d'une structure euclidienne équivariante qui en fasse un espace $\CAT(0)$.
Soulignons également que bien que $\Comp_3$ soit un complexe contractile de dimension 2 \cite[Theorem A]{LP}, à notre connaissance cela reste une question ouverte de savoir si  cela implique l'existence d'un point fixe pour toute action d'un groupe fini (voir \cite[p.205]{OS}).

Dans le présent article, nous introduisons un nouvel espace $\Xl_n$ sur lequel agit le groupe modéré $\TA$ (section~\ref{sec:valuation}).
L'espace $\Xl_n$,
dont les points sont certaines valuations monomiales modulo homothétie,
est lui aussi un CW-complexe dont les cellules sont des régions euclidiennes (mais en général pas des polyèdres) de dimension au plus $n-1$, et sa construction est inspirée de l'immeuble de Bruhat--Tits de $\SL_n(\F)$, pour $\F$ un corps valué.
Une propriété agréable de l'espace $\Xl_n$ est que les stabilisateurs pour l'action de $\TA$ sont conjugués à des produits semi-directs $M \rtimes L$ avec $M$ un sous-groupe d'automorphismes triangulaires stables par moyenne et $L$ un sous-groupe du groupe linéaire (proposition~\ref{pro:stabilisateur}), ce qui par le critère général de linéarisation déjà mentionné permet d'obtenir que tout sous-groupe fini d'un tel stabilisateur est linéarisable.
Reste alors à montrer que tout groupe fini agit sur l'espace $\Xl_n$ avec un point fixe global.
Comme dans les exemples mentionnés plus haut, ceci découle automatiquement une fois que l'on a pu munir $\Xl_n$ d'une distance qui en fait un espace $\CAT(0)$ complet, en considérant le «circumcenter» d'une orbite quelconque.
La construction d'une telle distance est le résultat principal de cet article.

Tout d'abord en dimension $n \ge 2$ quelconque et sur un corps $\K$ arbitraire nous construisons  $\Xl_n$ comme le quotient d'une union de copies de l'espace euclidien $\R^{n-1}$, que nous appelons appartements.
Nous montrons que la pseudo-distance quotient induite est une distance, qui fait de $\Xl_n$ un espace de longueur complet (proposition~\ref{pro:metrique} et lemme~\ref{lem:complet}).
Ensuite, restreignant l'étude au cas $n = 3$ et au cas d'un corps $\K$ de caractéristique nulle, nous montrons que $\Xl_3$ est simplement connexe (proposition~\ref{pro:X3 1-connexe}).
Ici la restriction sur la caractéristique de $\K$ vient de la théorie des réductions de Shestakov--Umirbaev et Kuroda,
dont une conséquence fameuse est l'inclusion stricte $\Tt \subsetneq \Aut(\K^3)$,
mais que nous utilisons ici seulement au travers de la description de $\Tt$ comme un produit amalgamé de trois facteurs le long de leurs intersections respectives.
Enfin, en nous appuyant sur l'étude des intersections entre appartements (section~\ref{sec:pointfixes}) nous montrons que $\Xl_3$ est localement à courbure négative ou nulle (proposition~\ref{pro:courbure negative}).
Finalement, par le théorème de Cartan--Hadamard on conclut (voir section~\ref{sec:princ}):

\begin{maintheorem} \label{thm:main}
Sur un corps $\K$ de caractéristique nulle, l'espace $\Xl_3$ est un espace métrique complet $\mathrm{\CAT(0)}$.
\end{maintheorem}

En corollaire, suivant la stratégie exposée plus haut nous obtenons :

\begin{maincorollary} \label{cor:main}
Sur un corps $\K$ de caractéristique nulle, tout sous-groupe fini de $\Tt$ est linéarisable.
\end{maincorollary}

En dimension $n = 2$ la construction précédente produit un arbre $\Xl_2$ qui est non isométrique (et même non isomorphe de manière équivariante) à l'arbre de Bass--Serre de $\Aut(\A^2)$, ce qui était une première indication que ce complexe contient potentiellement des informations nouvelles sur le groupe modéré.
Concernant la possibilité de généraliser nos résultats en dimension plus grande, rappelons tout d'abord qu'il existe des sous-groupes finis non linéarisables dans $\Aut(\A^4)$.
Par exemple, suivant \cite{FMJ}, si $\K$ contient 3 racines cubiques de l'unité, $1$, $\omega$ et $\omega^{-1}$, alors l'action sur $\A^4$ du groupe symétrique $S_3 = \langle \sigma, \tau \mid \sigma^3 = \tau^2 = (\sigma \tau)^2 = 1\rangle$ définie comme suit est non linéarisable:
\begin{align*}
\sigma(a,b,x,y) &= (\omega a, \omega^{-1} b, x,y), \\
\tau(a,b,x,y) &= (b,a, -b^3x + (1+ab+a^2b^2)y, (1-ab)x + a^3y ).
\end{align*}
On peut soupçonner que $\tau$ est un automorphisme non modéré, mais ceci relève de la question ouverte de savoir si l'inclusion $\TA \subset \Aut(\K^n)$ est stricte en dimension $n \ge 4$.
En revanche, comme $\tau$ après composition par l'involution $(a,b,x,y) \mapsto (b,a,y,x)$ s'identifie à
un élément de $\SL_2(\K[a,b])$, il est connu grâce à \cite{Suslin} que cet exemple devient modéré après extension à $\K^5$, en étendant trivialement l'action sur la  variable supplémentaire.
En effet, Suslin prouve que pour tout $r \ge 3$, le groupe $\SL_r(\K[x_1, \dots, x_n])$, que l'on peut penser comme un sous-groupe de $\Aut(\K^{r + n})$, est engendré par les matrices élémentaires, qui sont des automorphismes modérés particuliers.
Mentionnons que cette propriété d'être «stablement modéré», c'est-à-dire modéré après extension triviale à quelques variables supplémentaires, n'est pas valable seulement pour les automorphismes dépendant linéairement d'une partie des variables: par \cite[Theorem 4.10]{BEW}, tout automorphisme polynomial en les deux variables $x,y$ et à coefficients dans l'anneau $\K[a,b]$ devient modéré après extension à $\A^6$.
Par ailleurs, toujours suivant \cite{FMJ}, l'exemple précédent reste non-linéarisable dans $\Aut(\A^{4+m})$, si l'on prolonge trivialement l'action de $S_3$ à un nombre $m$ quelconque de variables supplémentaires.
En combinant avec le résultat de Suslin, pour tout $n \ge 5$ on obtient donc un sous-groupe fini non linéarisable de $\Tame(\K^n)$;
autrement dit le corollaire~\ref{cor:main}, et donc également le théorème~\ref{thm:main}, ne sont plus valables en dimension $n \ge 5$.
Il serait cependant intéressant d'étudier si l'un des deux ingrédients de la propriété $\CAT(0)$ persiste, à savoir la simple connexité ou la courbure négative ou nulle.
Enfin, le cas de la dimension $n = 4$ reste ouvert, mais semble difficile en l'absence d'une théorie des réductions.

L'espace $\Xl_n$ est construit en considérant l'orbite sous l'action de $\TA$ de l'ensemble des valuations monomiales associées au système de coordonnées $x_1, \dots, x_n$.
On pourrait tout aussi bien considérer l'action du groupe $\Aut(\K^n)$, et obtenir un espace plus grand sur lequel agit le groupe entier des automorphismes polynomiaux.
Le point est que ce nouvel espace n'est plus connexe, et la composante connexe contenant les valuations monomiales initiales est précisément notre espace $\Xl_n$.
Pour éviter de multiplier les notations nous avons fait le choix de nous restreindre dès le départ à l'action du groupe modéré, cependant le lecteur intéressé pourra vérifier que la plupart des énoncés généraux (en particulier dans les sections~\ref{sec:valuation} et~\ref{sec:stabilizers}) resteraient valables pour le groupe $\Aut(\K^n)$, avec une preuve inchangée.

Nous voyons l'espace $\Xl_n$ muni de l'action de $\TA$ comme un analogue de l'immeuble affine de Bruhat--Tits associé a $\SL_n(\F)$, pour $\F$ un corps valué.
Nous avons en tête la construction via les normes ultramétriques (voir \cite{BT, Parreau2000}), où un appartement est associé à chaque base de $\F^n$ en faisant varier la pondération des normes ultramétriques associées.
Dans notre situation on peut voir chaque $f \in \TA$ comme donnant une base de $\K$-algèbre de $\K[x_1,\dots,x_n] = \K[f_1,\dots,f_n]$, et lui associer une collection $\Ap_f$ de valuations monomiales naturellement paramétrée par un espace euclidien $\R^{n-1}$, que nous appelons l'appartement associé à $f$.
Une autre similitude est que chaque tel appartement contient une valuation particulière (celle dont tous les poids sont égaux), dont un voisinage dans $\Xl_n$ est isométrique au cône sur l'immeuble sphérique associé à $\GL_n(\K)$.
Cependant, comme nous le montrons dans l'appendice~\ref{sec:appart}, pour $n\geq 3$ l'espace $\Xl_n$ n'est
pas la réalisation de Davis \cite[Definition~12.65]{AB} d'un immeuble,
car la propriété cruciale «par deux points passe un appartement» \cite[Definition~4.1(B1)]{AB} est mise en défaut, même localement.
Mentionnons également que tout élément dans le stabilisateur d'un appartement $\Ap_f$ agit sur celui-ci comme une matrice de permutation, et jamais comme une translation comme on aurait pu l'anticiper.
En particulier un domaine fondamental pour l'action de $\TA$ sur $\Xl_n$ est une chambre de Weyl entière, et donc n'est pas compact (corollaire~\ref{cor:intrin}\ref{cor:intrin2}).
Il n'est pas clair pour nous si certaines notions affaiblies d'immeuble (par exemple les
«masures» de Gaussent et Rousseau \cite{GR}) pourraient englober notre construction.

Par définition $\Xl_n$ est inclus dans l'espace de toutes les valuations sur l'anneau $\K[x_1, \dots, x_n]$, étudiés en particulier dans de précédents travaux par Boucksom, Favre et Jonsson \cite{FJ, BFJ}.
Ces auteurs étudient principalement l'espace des valuations centrées à l'origine, mais le cas des valuations centrées à l'infini est similaire, et est traité en détail dans le cas $n = 2$ dans \cite[Appendix A]{FJ}.
En particulier notre arbre $\Xl_2$ est un sous-arbre de l'arbre $\Vl_1$ introduit dans \cite[\S A.3]{FJ}, mais avec des choix de normalisation différents.
Le cas de la dimension $n$ quelconque est traité dans \cite{BFJ}, et une structure affine simpliciale est mise en évidence qui en restriction à l'espace $\Xl_n$ correspond à nos appartements.
Il serait intéressant d'explorer les relations entre les propriétés de l'espace de toutes les valuations et celles du sous-espace $\Xl_n$ que nous considérons dans cet article.

Une autre question naturelle serait de comprendre les axes des isométries loxodromiques sur $\Xl_3$.
Une telle étude pourrait déboucher sur une alternative de Tits pour $\Tame(\A^3)$, ou sur une compréhension des propriétés dynamiques des éléments de $\Tame(\A^3)$, comme la croissance des degrés sous itération.

\section*{Remerciements}

L'inspiration initiale de ce projet vient d'un mini-cours sur les immeubles euclidiens donné par Anne Parreau en août 2015 à Nancy, dans le cadre de la conférence Non-Positive Curvature \& Infinite Dimension, à laquelle participait le premier auteur.
Merci à Anne Parreau pour ces lumineux exposés, à Junyi Xie pour de nombreuses discussions sur les valuations, et à Pierre-Marie Poloni qui nous a indiqué la référence \cite{BEW}.

\section{Espace de valuations}
\label{sec:valuation}

\subsection{Valuations}

Une \emph{valuation} sur l'anneau $\K[x_1, \dots, x_n]$ est une fonction
\[
\nu\colon \K[x_1, \dots, x_n] \to \R \cup \{+\infty\}
\]
telle que, pour tous polynômes $P_1, P_2$ :
\begin{itemize}
\item $\nu(P_1 + P_2) \ge \min \{ \nu(P_1), \nu(P_2) \}$;
\item $\nu(P_1P_2) = \nu(P_1) + \nu(P_2)$;
\item $\nu(P) = 0$ pour tout polynôme $P$ constant et non nul;
\item $\nu(P) = +\infty$ si et seulement si $P = 0$.
\end{itemize}

On note $\Vl_n$ l'ensemble des classes de telles valuations, modulo homothétie par un réel positif.
Si $P \in \K[x_1, \dots, x_n]$ est un polynôme et $g = (g_1, \dots, g_n) \in
\Aut(\A^n)$ est un automorphisme, on note $g^*P$ le polynôme $P(g_1,\dots, g_n)$.
Le groupe $\Aut(\A^n)$ agit (à gauche) sur l'ensemble des valuations via la formule :
\[
(g \cdot \nu)(P) := \nu( g^*(P) ).
\]
Cette action commute avec les homothéties, on obtient ainsi une action de $\Aut(\A^n)$ sur $\Vl_n$.

Nous décrivons maintenant certaines \emph{valuations monomiales} formant un sous-ensemble de~$\Vl_n$.
Nous notons $\W$ le quadrant positif de $\R^n$, c'est-à-dire
\[\W = \{\alpha = (\alpha_1, \dots, \alpha_n) \in \R^n;\, \alpha_i > 0 \text{ pour tout } i\}.
\]
Nous dirons que $\W$ est l'\emph{espace des poids}.
On définit également le sous-espace des poids bien ordonnés
\[
\W^+ = \{\alpha = (\alpha_1, \dots, \alpha_n);\,
\alpha_1 \ge \alpha_2 \ge \dots \ge \alpha_n > 0\} \subset \W.
\]
Pour tout $\alpha = (\alpha_1, \dots, \alpha_n) \in 
\W$, nous notons $[\alpha_1, \dots, \alpha_n]$, ou simplement $[\alpha]$, la classe de $\alpha$ modulo homothétie par un réel positif, et nous notons $\PW$ la projectivisation de~$\W$, qui est un simplexe ouvert de dimension $n-1$.
De même $\PW^+\subset \PW$ désignera le simplexe semi-ouvert qui est la projectivisation de $\W^+$.
On équipe $\PW$ et $\PW^+$ de la topologie induite par $\mathbb{P}^{n}(\R)$.
En particulier, nous dirons qu'un poids $\alpha \in \PW^+$ vérifiant $\alpha_1 > \alpha_2 > \dots > \alpha_n$ est à l'intérieur de $\PW^+$, et au contraire que $\alpha$ est dans la frontière de $\PW^+$ s'il existe deux indices $i > j$ tel que $\alpha_i = \alpha_j$.
Il sera souvent utile de prolonger ces définitions pour inclure des poids avec certains coefficients nuls, qu'on appellera le
\emph{bord à l'infini} de $\PW$.

Étant donné un poids $\alpha = (\alpha_1, \dots, \alpha_n) \in \W$, on définit une valuation monomiale $\nu_{\id,\alpha}$ comme suit.
Pour tout
\[
P = \sum_{I = (i_1,\dots, i_n)} c_{I}
x_1^{i_1} \dots x_n^{i_n} \in \K[x_1, \dots, x_n],
\]
on appelle support de $P$, noté $\supp P$, l'ensemble des multi-indices $I$ tels que
$c_{I} \neq 0$, et l'on pose:
\[
\nu_{\id,\alpha} \left( P \right)
= \min_{I \in \supp P} \left( - \sum_{k=1}^n \alpha_k i_k \right).
\]
Autrement dit $-\nu_{\id,\alpha}(P)$ est le degré pondéré de $P$, où chaque variable $x_i$ est affectée du poids $\alpha_i$.
Nous avons adopté le point de vue des valuations dans un souci de compatibilité avec des travaux existants, en particulier \cite{FJ}.

Observons que pour tout poids $\alpha\in \W$ et tout $t > 0$, on a $t  \nu_{\id,\alpha} = \nu_{\id, t \alpha}$.
Ainsi la classe d'homothétie de $\nu_{\id,\alpha}$ ne dépend que de $[\alpha]\in \PW$, et on la note $\nu_{\id,[\alpha]}\in \Vl_n$.
Nous notons $\Ap_\id$ (resp.\ $\Ap_\id^+$) l'ensemble de toutes les classes $\nu_{\id,[\alpha]}$ avec $[\alpha] \in \PW$ (resp.\ $[\alpha] \in \PW^+$).
Nous définissons alors l'espace $\Xl_n \subset \Vl_n$ (noté simplement $\Xl$ si la dimension est claire par contexte) comme l'orbite de $\Ap_\id$ sous l'action de $\TA$.
Si $f \in \TA$, et $\alpha = (\alpha_1, \dots, \alpha_n) \in \W$, nous noterons
\[\nu_{f, \alpha} := f \cdot \nu_{\id,\alpha}.\]
Explicitement, en notant $f^{-1} = (g_1, \dots, g_n)$, la valuation $\nu_{f, \alpha}$ satisfait la propriété : $\nu_{f, \alpha} (g_i) = \alpha_i$ pour tout $1 \le i \le n$.
En effet $\nu_{f, \alpha} (g_i) = f\cdot \nu_{\id, \alpha}(g_i) = \nu_{\id, \alpha}(g_i \circ f) = \nu_{\id, \alpha}(x_i)$.
À nouveau, la classe d'homothétie de $\nu_{f, \alpha}$ ne dépend que de $[\alpha]\in \PW$, et on la note $\nu_{f, [\alpha]}$.

On identifie le groupe symétrique $S_n$ à un sous-groupe de $\GL_n(\K)$, et donc aussi de $\TA$, via l'action par permutation sur les vecteurs de la base canonique de $\K^n$.
En termes de coordonnées cela équivaut à poser, pour tout $\sigma \in S_n$ :
\[
\sigma = (x_{\sigma^{-1}(1)}, \dots, x_{\sigma^{-1}(n)}) \in \TA.
\]
Ainsi, pour tous $f \in \TA$ et $P \in \K[x_1, \dots, x_n]$, on a
\[
\sigma^*P = P(x_{\sigma^{-1}(1)}, \dots, x_{\sigma^{-1}(n)}) \in \K[x_1, \dots, x_n].
\]
De même, le groupe symétrique agit sur $\W$ par
\[
\sigma (\alpha) = (\alpha_{\sigma^{-1}(1)}, \dots, \alpha_{\sigma^{-1}(n)}).
\]

\begin{lemma}
\label{lem:permute}
Pour tous $f \in \TA$, $\alpha \in \W$, $\sigma \in S_n$, on a
\[
\nu_{f \sigma,\alpha} = \nu_{f, \sigma(\alpha)}.
\]
\end{lemma}

\begin{proof}
La vérification est directe :
\begin{multline*}
\nu_{f \sigma,\alpha}(P)
= \nu_{\id, \alpha}((f\sigma)^*(P))
=\nu_{\id, \alpha}(\sigma^*(f^*P))
=\nu_{\id, \alpha}(f^*P(x_{\sigma^{-1}(1)}, \dots, x_{\sigma^{-1}(n)}))\\
= \nu_{\id, \sigma(\alpha)} (f^*P(x_1, \dots, x_n)) = \nu_{f,\sigma(\alpha)} (P).
\end{multline*}
On pourra se convaincre de la quatrième égalité en considérant l'exemple $f=\id, \sigma = (123)$.
Il s'agit de vérifier
\[
\nu_{\id, (\alpha_1,\alpha_2,\alpha_3)}(P(x_3, x_1, x _2)) = \nu_{\id, (\alpha_3,\alpha_1,\alpha_2)} (P(x_1, x_2,x_3)).
\]
Dans les deux cas, la première entrée du polynôme $P$ a poids $\alpha_3$, la deuxième a poids $\alpha_1$, et la troisième a poids $\alpha_2$.
\end{proof}

\subsection{Appartements}

Pour tout $f \in \TA$, nous appelons \emph{appartement} associé à $f$ l'ensemble
\[
\Ap_f := f(\Ap_\id) = \{\nu_{f,[\alpha]} ;\, [\alpha] \in \PW \} \subset \Xl.
\]
En particulier $\Ap_{\id}$ sera appelé l'\emph{appartement standard}.
De façon similaire nous appelons \emph{chambre} associée à $f$ l'ensemble
\[
\Ap_f^+ := f(\Ap_\id^+) = \{\nu_{f,[\alpha]} ;\, [\alpha] \in \PW^+ \} \subset \Ap_f,
\]
et nous appelons $\Ap_\id^+$ la \emph{chambre standard}.
Le lemme~\ref{lem:permute} entraîne que l'appartement $\Ap_f$ associé à $f$ est la réunion des $n!$ chambres $\Ap^+_{f\sigma}$, où $\sigma \in S_n$ est une permutation.

Pour tout poids $\alpha \in \W$, il existe un unique $\alpha^+ \in \W^+$ obtenu en réordonnant les $\alpha_i$.
Cette application descend aux projectivisés en une application
\[
[\alpha] \in \PW \mapsto [\alpha]^+ := [\alpha^+] \in \PW^+.
\]
Nous noterons $\alpha^+ = (\alpha_1^+, \dots, \alpha_n^+)$ les coordonnées de $\alpha^+$, en particulier
\[
\alpha_1^+ = \max_{1 \le i \le n} \alpha_i
\;\text{ et }\;
\alpha_n^+ = \min_{1 \le i \le n} \alpha_i
\]

\begin{lemma} \label{lem:formes_lineaire}
Soient $\ell_i, \dots, \ell_n$ des formes linéaires indépendantes sur $\K^n$, et $\nu = \nu_{\id, \alpha}$ pour un poids $\alpha \in \W$.
Alors
\[
-\nu(\ell_i) - \dots - \nu(\ell_n) \ge \alpha^+_i + \dots + \alpha^+_n.
\]
\end{lemma}

\begin{proof}
Soit $\sigma \in S_n$ tel que $\alpha_j^+ = \alpha_{\sigma(j)}$ pour tout $1\leq j\leq n$.
Pour une forme linéaire $\ell = \sum_{j=1}^n a_{\sigma(j)} x_{\sigma(j)}$, on a $-\nu(\ell) = \max \{\alpha_j^+ \mid a_{\sigma(j)} \neq 0 \}$.
Donc si $-\nu(\ell) = \alpha_k^+$, $\ell$ ne dépend que des variables $x_{\sigma(k)}, \dots , x_{\sigma(n)}$.
Maintenant si $\ell_i, \dots, \ell_n$ sont indépendantes, elles ne peuvent pas toutes ne dépendre que des variables $x_{\sigma(i+1)}, \dots, x_{\sigma(n)}$, et donc il existe un indice $i \le j \le n$ tel que $-\nu(\ell_j) \ge \alpha_i^+$.
On conclut par récurrence sur le nombre de formes linéaires.
\end{proof}

Observons que l'inégalité peut être stricte dans le lemme~\ref{lem:formes_lineaire}.
Par exemple en dimension~$2$, si $\ell_1 = x_1$, $\ell_2 = x_1 + x_2$ et $\alpha = (2,1)$, on a
\[
-\nu(\ell_1) - \nu(\ell_2) = 2 + 2 > 2 + 1 = \alpha_1^+ + \alpha_2^+.
\]

\begin{proposition} \label{pro:alpha intrinseque}
Soit $f,g \in \TA$, et $\alpha, \beta \in \W$.
Si $\nu_{f,\alpha} = \nu_{g,\beta}$, alors $\alpha^+ = \beta^+$.
\end{proposition}

\begin{proof}
Étant donnée une valuation $\nu = \nu_{f,\alpha} \in \Xl$, on définit par récurrence
descendante une suite $\gamma_i$ en posant, pour $i = n, n-1, \dots, 1$:
\[
\gamma_i := \inf \left( -\nu(h_i h_{i+1} \dots h_n)\right) - \sum_{j = i+1}^n \gamma_{j},
\]
où l'infimum est pris sur les $(h_i, \dots, h_n)$ qui sont $n-i+1$ composantes distinctes d'un $h=(h_1,\ldots, h_n)\in\TA$.
Nous allons voir que $(\gamma_1, \dots, \gamma_n) = (\alpha^+_1, \dots, \alpha^+_n)$, ce qui montrera que $\alpha^+ = \gamma$ est intrinsèquement défini.

Observons d'abord que l'on peut supposer $f=\id$, car $\nu_{\id, \alpha} = f^{-1}\cdot\nu_{f,\alpha}$ et la définition de~$\gamma$ est par construction invariante sous l'action de $\TA$.
Supposons que $\gamma_j = \alpha^+_j$ pour $j = i+1, \dots, n$ (ce qui correspond à une condition vide dans le cas $i = n$), et montrons que c'est aussi le cas pour $j = i$.
Par définition
\[
\gamma_i = \inf \left( -\nu(h_i h_{i+1} \dots h_n)\right) - \alpha^+_{i+1} - \dots - \alpha^+_n.
\]
Pour tout $h = (h_1, \dots, h_n) \in \TA$
les parties linéaires des polynômes  $h_i, \dots, h_n$ sont des formes indépendantes, car sinon $h$ ne serait pas inversible.
Par le lemme~\ref{lem:formes_lineaire}, on a
\[
-\nu(h_i h_{i+1} \dots h_n) = -\nu(h_i) - \dots - \nu(h_n) \ge \alpha^+_i + \dots + \alpha^+_n
\]
et l'égalité est réalisée pour $h \in S_n$ une permutation telle que $h(\alpha) = \alpha^+$.
On conclut que $\gamma_i = \alpha^+_i$.
\end{proof}

Nous dégageons trois conséquences immédiates :

\begin{corollary}~ \label{cor:intrin}
\begin{enumerate}[wide]
\item \label{cor:intrin1}
Il existe une application bien définie $\wgt\colon \Xl \to
\PW^+$, qui envoie $\nu_{f,[\alpha]}$ sur $[\alpha]^+$, et qui est donc une bijection en restriction à chaque chambre $\Ap^+_f$.
\item \label{cor:intrin2}
La chambre standard $\Ap_{\id}^+$ est un domaine fondamental pour l'action de $\TA$ sur~$\Xl$, c'est-à-dire toute orbite rencontre $\Ap_{\id}^+$ en exactement un point.

\item \label{cor:intrin3}
Soit $\alpha = (\alpha_1, \dots, \alpha_n) \in \W^+$ et $f = (f_1, \dots, f_n) \in \TA$.
Si $f\cdot \nu_{\id,[\alpha]} = \nu_{\id,[\alpha]}$, alors $f \cdot \nu_{\id,\alpha} = \nu_{\id,\alpha}$.
\end{enumerate}
\end{corollary}

\begin{remark}
Pour chaque appartement $\Ap_f$ et chaque choix de $\sigma \in S_n$,
on a une application
\begin{align*}
\Ap_f &\mapsto \PW \\
\nu_{f\sigma,[\alpha]} &\to [\alpha]
\end{align*}
Cependant il n'est pas possible de faire des choix cohérents qui permettraient d'étendre cette application à $\Xl$.

Par exemple, en dimension $n = 2$, il n'est déjà pas possible d'étendre cette application à l'union des trois appartements $\Ap_\id \cup \Ap_f\cup \Ap_g$  pour $f=(x_2, x_1+x_2)$ et $g=(x_1+x_2, x_1)$.
Intuitivement, cela vient du fait que ces appartements forment un «tripode».
Formellement, chacun de ces appartements contient deux parmi les trois valuations suivantes :
\begin{align*}
\nu_{\id,[1,2]}=\nu_{f,[2,1]},&& \nu_{f,[1,2]}=\nu_{g,[2,1]}, &&\nu_{g,[1,2]}=\nu_{\id,[2,1]},
\end{align*}
et chacune devrait être envoyée sur $[1,2]$ ou $[2,1]$.
Cependant, il n'existe pas d'application depuis un ensemble de cardinalité $3$ vers un ensemble de cardinalité $2$ qui soit injective sur chaque paire.

On proposera dans le lemme~\ref{lem:rho} un remède partiel à cet état de fait.
\end{remark}

Nous noterons $\Fix(f)$ le sous-ensemble de $\Xl$ fixé par $f$.
Puisque nous ne considérons jamais les points fixes de $f$ comme automorphisme de l'espace affine $\K^n$, cette notation ne devrait pas porter à confusion.

\begin{corollary} \label{cor:intersection chambres}
Soit $f \in \TA$ et $\Fix(f)\subset \Xl$ son ensemble de points fixes.
Alors
\[
\Ap^+_f \cap \Ap^+_\id = \Fix(f) \cap \Ap^+_\id.
\]
\end{corollary}

\begin{proof}
Si $\nu_{\id,[\alpha]} \in \Fix(f) \cap \Ap^+_\id$, on a
\[
\nu_{f,[\alpha]} = f(\nu_{\id,[\alpha]}) = \nu_{\id,[\alpha]} \in \Ap^+_f \cap \Ap^+_\id.
\]
Réciproquement tout point $\nu_{f,[\alpha]} = \nu_{\id,[\beta]} \in \Ap_f^+ \cap \Ap_\id^+$ vérifie
$[\alpha] = [\alpha]^+ = [\beta]^+ = [\beta]$, où la deuxième égalité est la proposition
\ref{pro:alpha intrinseque}.
On a donc  $f(\nu_{\id,[\alpha]}) = \nu_{f,[\alpha]} = \nu_{\id,[\alpha]}$, autrement dit $\nu_{\id,[\alpha]} \in \Fix(f) \cap \Ap^+_\id$ comme attendu.
\end{proof}

D'après le corollaire~\ref{cor:intersection chambres}, afin de décrire les intersections entre chambres il faut comprendre les lieux $\Fix(f)\cap \Ap^+_\id$, ce que nous allons faire dans les deux sections qui suivent.

\section{Stabilisateurs}
\label{sec:stabilizers}

Soit $\alpha = (\alpha_1, \dots, \alpha_n) \in \W^+$.
Nous allons déterminer le stabilisateur $\Stab(\nu_{\id,[\alpha]})$ de la classe $\nu_{\id,[\alpha]}$ pour l'action de $\TA$ sur $\Xl$.
Le groupe des \emph{automorphismes triangulaires}, qui est un sous-groupe de $\TA$, jouera ici un rôle important.
Rappelons qu'un automorphisme est dit triangulaire s'il est de la forme
\[
(x_1, \dots, x_n) \mapsto (a_1 x_1 + P_1(x_2, \ldots, x_n), \ldots, a_i x_i  + P_i(x_{i+1}, \ldots, x_n), \ldots, a_n x_n + c).
\]

Écrivons
\[
\alpha = (\underbrace{\gamma_1, \dots, \gamma_1}_{m_1 \text{ fois}}, \dots,
\underbrace{\gamma_r, \dots, \gamma_r}_{m_r \text{ fois}})
\]
avec $\gamma_1 > \dots > \gamma_r > 0$, on a donc $\sum_{i = 1}^r m_i = n$.
Nous définissons maintenant deux sous-groupes $L_\alpha$ et $M_\alpha$ de $\TA$, dont on vérifie immédiatement que ce sont des sous-groupes de $\Stab(\nu_{\id,[\alpha]})$.

Tout d'abord
\[
L_\alpha \simeq \GL_{m_1}(\K) \times \dots \times \GL_{m_r}(\K)
\]
est le sous-groupe de $\GL_n(\K)$ des matrices diagonales par blocs de taille $m_i$.
En particulier $L_\alpha$ contient le sous-groupe des matrices diagonales.

Pour tout indice $i = 1, \dots, n$, notons $1 \le b(i) \le r$ le numéro du bloc auquel appartient l'indice $i$, c'est-à-dire tel que $\alpha_i = \gamma_{b(i)}$.
Nous définissons $M_\alpha$ comme le sous-groupe du groupe des
automorphismes triangulaires de la forme
\[
(x_1, \dots, x_n) \mapsto (x_1 + P_1, \dots, x_i + P_i, \dots, x_n +c),
\]
où chaque $P_i$ satisfait $\alpha_i \ge - \nu_{\alpha} (P_i)$
et ne dépend que des variables $x_j$ avec $b(j) > b(i)$.
En particulier $M_\alpha$ est un groupe de dimension finie contenant le sous-groupe des translations.

À noter que les définitions de $L_\alpha$ et $M_\alpha$ ne dépendent que de la classe d'homothétie de $\alpha$.
Observer aussi que le groupe $\langle L_\alpha, M_\alpha \rangle$ contient toujours le sous-groupe des matrices triangulaires supérieures.

\begin{example}~ \label{exple:stab}
\begin{enumerate}[wide]
\item \label{exple:stab1}
Si $\alpha = (1,\dots,1)$, alors $L_\alpha = \GL_n(\K)$ et $M_\alpha$
est le groupe des translations.

\item \label{exple:stab2}
Si $\alpha = (\alpha_1, \dots, \alpha_n)$ avec $\alpha_1 >
\dots >
\alpha_n$, alors $L_\alpha$ est le groupe des matrices diagonales, et
$M_\alpha$ est le groupe des automorphismes triangulaires
\[
(x_1, \dots, x_n) \mapsto (x_1 + P_1, \dots, x_i + P_i, \dots)
\]
où les $P_i$ satisfont $\alpha_i \ge -\nu_{\alpha} (P_i)$.
\end{enumerate}

Remarquons qu'en dimension $n = 2$, les deux exemples précédents couvrent tous les cas possibles.
En dimension $n = 3$, il y a deux autres possibilités (ici nous utilisons la notation $\deg P$ pour le degré ordinaire d'un polynôme $P$, c'est-à-dire chaque variable est affectée du poids 1) :
\begin{enumerate}[wide, resume]
\item \label{exple:stab3}
Si $\alpha = (\alpha_1, 1, 1)$ avec $\alpha_1 > 1$, alors
\begin{align*}
L_\alpha &= \{(a x_1, bx_2 + cx_3, b'x_2 + c'x_3) \} \subset \GL_3(\K), \\
M_\alpha &= \{(x_1 + P(x_2,x_3), x_2 + d, x_3 + d'); \, \alpha_1 \ge \deg P \}.
\end{align*}

\item \label{exple:stab4}
Si $\alpha = (\alpha_1, \alpha_1, 1)$ avec $\alpha_1 > 1$, alors
\begin{align*}
L_\alpha &= \{(a x_1 + bx_2, a'x_1 + b'x_2, cx_3) \} \subset \GL_3(\K), \\
M_\alpha &= \{(x_1 + P(x_3), x_2 + Q(x_3), x_3 + d); \, \alpha_1 \ge \deg
P,\deg Q \}.
\end{align*}

\end{enumerate}
\end{example}

\begin{proposition} \label{pro:stabilisateur}
Soit $\nu_{\id,[\alpha]} \in \Ap^+_\id$.
Alors le stabilisateur de $\nu_{\id,[\alpha]}$ pour l'action de $\TA$ sur $\Xl$ est le produit semi-direct de $M_\alpha$ et $L_\alpha$ :
\[
\Stab(\nu_{\id,[\alpha]}) = M_\alpha \rtimes L_\alpha.
\]
\end{proposition}

\begin{proof}
Considérons $f = (f_1, \dots, f_n) \in \Stab(\nu_{\id,[\alpha]})$.
Grâce au corollaire~\ref{cor:intrin}\ref{cor:intrin3}, nous avons $f \in  \Stab(\nu_{\id,\alpha})$.
En composant par une translation, qui est un élément de $M_\alpha$, nous pouvons supposer que les $f_i$ n'ont pas de terme constant.
Pour chaque indice $i$, nous écrivons $f_i = \ell_i + P_i$, avec $\ell_i$ linéaire et $P_i$ dont tous les monômes sont de degré au moins 2.
La condition
\[
\alpha_i  = -\nu_{\id,\alpha}(x_i) = -(f\cdot \nu_{\id,\alpha})(x_i) = -\nu_{\id,\alpha}(\ell_i + P_i) \ge -\nu_{\id,\alpha}(P_i)
\]
implique que $P_i$ ne dépend que des variables $x_j$ vérifiant $\alpha_i > \alpha_j$.
Donc en composant par un élément de $M_\alpha$ nous pouvons nous ramener au cas où tous les $P_i$ sont nuls.
On obtient ainsi un élément de $\GL_n(\K) \cap \Stab(\nu_{\id,\alpha})$, qui doit être triangulaire par blocs de taille $m_i \times m_j$.
En particulier, nous pouvons écrire une telle matrice comme la composée d'un élément de $L_\alpha$ et de $\GL_n(\K) \cap M_\alpha$.

On conclut que
$\Stab(\nu_{\id,[\alpha]}) = \langle M_\alpha, L_\alpha \rangle$.
Par construction $M_\alpha \cap L_\alpha = \{\id\}$,
et le fait que $M_\alpha$ soit normalisé par $L_\alpha$ est un calcul immédiat.
\end{proof}

\begin{corollary} \label{cor:stab chambre}
Le stabilisateur dans $\TA$ de la chambre standard $\Ap^+_\id$ est le produit semi-direct du groupe des translations et du groupe des matrices triangulaires supérieures.
En particulier, $f\in \TA$ fixe chaque point de $\Ap_\id$ si et seulement si $f$ est de la forme
\[
f = (c_1 x_1 + t_1, c_2 x_2 + t_2, \dots).
\]
\end{corollary}

\begin{proof}
Soit $f \in \Stab(\Ap^+_\id)$, et $\alpha \in \W^+$.
Par le corollaire~\ref{cor:intrin}\ref{cor:intrin3},
on a $\nu_{f,\alpha} = f(\nu_{\id,\alpha}) = \nu_{\id,\beta}$ pour un certain $\beta \in \W^+$, et par la proposition~\ref{pro:alpha intrinseque}, $\alpha = \alpha^+ = \beta^+ = \beta$.
Autrement dit $f$ fixe point par point les éléments de $\Ap^+_\id$.
En particulier $f$ fixe la valuation monomiale de poids $(1, \dots, 1)$, ce
qui correspond à l'exemple~\ref{exple:stab}\ref{exple:stab1}.
En composant par une translation, on peut donc supposer $f \in \GL_n(\K)$.
Si on considère ensuite $\alpha$ tel que $ \alpha_1 > \dots > \alpha_n > 0$, alors $f \in \Stab(\nu_{\id,[\alpha]})$ et relève de l'exemple~\ref{exple:stab}\ref{exple:stab2}.
Comme on vient de voir que $f$ est linéaire, on obtient que $f$ est triangulaire supérieure.
Ainsi $\Stab(\Ap_\id^+)$ est le groupe affine triangulaire, et donc égal au produit semi-direct annoncé.

La deuxième assertion s'obtient en conjuguant par des éléments du groupe symétrique.
\end{proof}

On peut utiliser le corollaire~\ref{cor:stab chambre} pour démontrer la fidélité de l'action de $\TA$ sur~$\Xl$.
On reporte la preuve à l'appendice (proposition~\ref{pro:fidele}) car cet énoncé n'est pas nécessaire pour la preuve du théorème~\ref{thm:main}.

Appelons \emph{face} de $\PW^+$ le sous-ensemble de $\PW^+$ défini par un certain nombre d'égalités de la forme $\alpha_k=\alpha_{k+1}$.
Une \emph{face} de $\Xl$ est un ensemble de valuations $\nu_{f,[\alpha]}$, pour $f\in\TA$ fixé et $[\alpha]$ dans une face fixée de $\PW^+$.

Soit $\Delta$ le poset des classes à gauche de $\GL_n(\K)$ par rapport aux sous-groupes paraboliques standards.
Rappelons suivant \cite[Definition~6.32]{AB} que $\Delta$ est l'immeuble sphérique $\Delta(G,B)$ pour $G=\GL_n(\K)$ et $B$ le sous-groupe des matrices triangulaires supérieures.

\begin{lemma}
\label{lem:GL}
Le poset des faces de $\Xl$ contenant $\nu_{\id,[1,\ldots, 1]}$ est isomorphe à $\Delta$.
\end{lemma}

\begin{proof}
Par le corollaire~\ref{cor:stab chambre} le sous-groupe $\GL_n(\K)\subset  \Stab(\nu_{\id,[1,\ldots, 1]})$ agit transitivement sur les faces de $\Xl$ contenant $\nu_{\id,[1,\ldots, 1]}$.
De plus, par la proposition~\ref{pro:stabilisateur}, les stabilisateurs dans $\GL_n(\K)$ des faces de $\Xl$ contenues dans $\Ap^+_\id$ sont les sous-groupes paraboliques standards de $\GL_n(\K)$.
Ainsi les faces de $\Xl$ contenant $\nu_{\id,[1,\ldots, 1]}$ correspondent aux classes à gauche de $\GL_n(\K)$ par rapport aux sous-groupes paraboliques standards.
\end{proof}

Notons $\Tame_0(\A^n) \subset \TA$ le sous-groupe des automorphismes modérés fixant l'origine.
Tout automorphisme $f \in \TA$ s'écrit sous la forme $f = f_0 \circ t$, où $f_0 \in \Tame_0(\A^n)$ et $t$ est la translation qui envoie $f^{-1}(0)$ sur $0$.
Par le corollaire~\ref{cor:stab chambre}, les translations fixent point par point l'appartement standard $\Ap_\id$, donc on a $\Ap_f = \Ap_{f_0}$.
Ainsi $\Xl$ est couvert par les $\Ap_{f_0}$ avec $f_0 \in \Tame_0(\A^n)$.

Soit $\Diff\colon\Tame_0(\A^n)\to \GL_n(\K)$ la différentielle à l'origine, ou autrement dit l'homomorphisme qui oublie tous les termes de degré~$> 1$ en $x_1,\ldots, x_n$.
Par ailleurs, pour chaque $a\in \GL_n(\K)$ la décomposition de Bruhat \cite[IV.2]{Bourbaki} donne un unique $\sigma_a\in S_n$ tel que $a\in B_n\sigma_aB_n$, où $B_n\subset \GL_n(\K)$ est le sous-groupe des matrices triangulaires supérieures.
Pour chaque $f\in \Tame_0(\A^n)$, notons $\sigma_f:=\sigma_{\Diff(f)}$.

\begin{lemma}~ \label{lem:rho}
\begin{enumerate}[wide]
\item L'application $\rho\colon \Xl\to \PW$ qui a tout $\nu=\nu_{f,[\alpha]} \in \Xl$ avec $f \in \Tame_0(\A^n)$ et $\alpha\in \W^+$ associe $\rho(\nu)=[\sigma_f(\alpha)]\in \PW$ est bien définie.
\item
L'application $\wgt$ du corollaire~\ref{cor:intrin}\ref{cor:intrin1} factorise par $\rho$, au sens où le diagramme suivant commute $($on suppose toujours $f \in \Tame_0(\A^n)$ et $\alpha\in \W^+)$ :
\[\begin{tikzcd}
\nu_{f,[\alpha]} \in \Xl \ar[rr,"\wgt"] \ar[dr,"\rho",swap]&& \PW^+ \ni [\alpha]^+ = [\sigma_f(\alpha)]^+ \\
& {[\sigma_f(\alpha)]} \in \PW \ar[ur]
\end{tikzcd}\]
\item \label{lem:rho_goodprojection}
Soit $f\in \Tame_0(\A^n)$ tel que $\Diff(f)\in B_n$.
Alors pour tout $\alpha\in \W$ on a
$\rho(\nu_{f,[\alpha]})=[\alpha]$.
\end{enumerate}
\end{lemma}

\begin{proof}
\begin{enumerate}[wide]
\item
Soient $f,g \in \Tame_0(\A^n)$ et $\alpha, \beta \in \W^+$ tel que $\nu_{f,[\alpha]}=\nu_{g,[\beta]}$.
Par la proposition~\ref{pro:alpha intrinseque} on a $[\alpha] = [\beta]$.
Il s'agit maintenant de montrer que $\sigma_g=\sigma_f \circ \sigma$ pour une permutation $\sigma\in S_n$ vérifiant $\sigma(\alpha)=\alpha$.
Notons donc $\Sigma\subset S_n$ le sous-groupe des permutations fixant~$\alpha$.
Observons que puisque $\alpha \in \W^+$, le groupe $\Sigma$ est engendré par un sous-ensemble des générateurs standards $\sigma_i=(i,i+1)\in S_n$.

Par la proposition~\ref{pro:stabilisateur}, $f^{-1}g\in M_\alpha\rtimes L_\alpha$.
Notons que $\Diff(M_\alpha \cap \Tame_0(\A^n))\subset B_n$.
De plus, en considérant une décomposition de Bruhat pour chaque bloc on obtient $L_\alpha\subset B_n\Sigma B_n$.
 Ainsi $\Diff(g)\in \Diff(f)B_n\Sigma B_n$, donc $\Diff(g)\in B_n\sigma_fB_n\sigma_{i_1}\cdots\sigma_{i_k}B_n$, où les $\sigma_{i_j} \in \Sigma$ sont des générateurs standards.
Par un des axiomes des systèmes de Tits (l'axiome (T3) dans \cite[IV.2]{Bourbaki}), il existe un sous-produit $\sigma$ du produit $\sigma_{i_1}\cdots\sigma_{i_k}$ tel que $\Diff(g)\in B_n \sigma_f\sigma B_n$. Ainsi $\sigma_g=\sigma_f\sigma$, comme attendu.

\item Ce point est immédiat, puisque par définition pour tout poids $\alpha$ et toute permutation~$\sigma$ on a $[\alpha]^+ = [\sigma(\alpha)]^+$.

\item
L'hypothèse $\Diff(f)\in B_n$ implique $\sigma_f=\id$.
Étant donné $\alpha \in \W$, écrivons $\alpha=\sigma(\alpha^+)$ avec $\alpha^+\in \W^+,\sigma\in S_n$.
On a $\Diff(f\sigma)= \Diff(f)\sigma\in  B_n\sigma$, et par conséquent $\sigma_{f \sigma}=\sigma$. Ainsi par le lemme~\ref{lem:permute},
\[
\rho(\nu_{f,[\alpha]})=\rho(\nu_{f,[\sigma(\alpha^+)]})=\rho (\nu_{f\sigma ,[\alpha^+]})=[\sigma (\alpha^+)] = [\alpha],
\]
comme attendu. \qedhere
\end{enumerate}
\end{proof}

\begin{corollary} \label{cor:intersection appartements}
Soit $f \in \TA$ tel que $\Ap_f$ partage un point avec l'intérieur de $\Ap^+_\id$.
Alors il existe $\sigma\in S_n$ tel que $\Ap_f \cap \Ap_\id = \Fix(f\sigma) \cap \Ap_\id$
\end{corollary}

\begin{proof}
Quitte à composer $f$ à droite par une translation, on peut supposer $f \in \Tame_0(\A^n)$.
Pour chaque $\sigma\in S_n$ on a $\Ap_f=\Ap_{f\sigma}$, d'où l'inclusion
$\Fix(f\sigma) \cap \Ap_\id\subset \Ap_{f\sigma} \cap \Ap_\id=\Ap_f \cap \Ap_\id$.
Pour la réciproque, par hypothèse il existe $[\alpha']$ à l'intérieur de $\PW^+$ avec $\nu_{\id,[\alpha']}\in \Ap_f$.
Par la proposition~\ref{pro:alpha intrinseque} et le lemme~\ref{lem:permute},
\[
\nu_{\id,[\alpha']}=\nu_{f,[\sigma(\alpha')]}=\nu_{f\sigma,[\alpha']}
\]
pour un certain $\sigma\in S_n$.
Par la proposition~\ref{pro:stabilisateur} et l'exemple~\ref{exple:stab}\ref{exple:stab2}, $\Diff(f\sigma)\in B_n$, on peut donc appliquer le lemme~\ref{lem:rho}\ref{lem:rho_goodprojection} à $f\sigma$.
Ainsi pour chaque $\nu =\nu_{f\sigma,[\alpha]} \in \Ap_f$ on a $\rho(\nu)=[\alpha]$.
Si de plus $\nu\in \Ap_\id$, de nouveau par le lemme~\ref{lem:rho}\ref{lem:rho_goodprojection} appliqué à $\id$ on a $\nu=\nu_{\id,[\alpha]}$.
Ainsi $\nu\in \Fix(f\sigma)$ comme attendu.
\end{proof}

\section{Points fixes}
\label{sec:pointfixes}
\subsection{Équations admissibles}

Nous dirons qu'un poids $\alpha \in \W$ satisfait une \emph{équation admissible} s'il
existe des entiers $m_j \ge 0$ non tous nuls et un indice $i$ tel que
\begin{equation} \label{eq:alpha}
\alpha_i = \sum_{j \neq i} m_j \alpha_j.
\end{equation}

Nous appelons \emph{hyperplan admissible} associé à une telle équation l'ensemble
des $\alpha\in \W$ vérifiant l'équation,
et \emph{demi-espace admissible} l'ensemble des $\alpha$ vérifiant
l'inégalité
\[
\alpha_i \ge \sum_{j \neq i} m_j \alpha_j.
\]
En particulier, si l'équation admissible n'est pas de forme $\alpha_i=\alpha_k$, alors le demi-espace admissible est le demi-espace ne contenant pas
$(1, \dots, 1)$.
On utilisera également la terminologie d'\emph{hyperplans} et de \emph{demi-espaces admissibles} pour leurs projectivisations dans $\PW$.

Par exemple, quand $n=2$ les hyperplans admissibles de $\PW$ sont les points $[p,1]$ et $[1,p]$, où $p \ge 1$ est entier, et les demi-espaces admissibles sont les intervalles semi-ouverts de $[p,1]$ vers le point au bord à l'infini $[1,0]$, ou de $[1,p]$ vers $[0,1]$.
Quand $n=3$ les hyperplans admissibles dans~$\PW$ sont des segments de droites que nous appellerons \emph{droites admissibles} (voir figure~\ref{fig:nervure}).

Parmi les hyperplans admissibles, dans $\W$ aussi bien que dans $\PW$, nous appelons \emph{hyperplan principal} tout hyperplan d'équation de la forme $\alpha_i = m_k \alpha_k$, pour des indices $i \neq k$.
Autrement dit dans l'équation (\ref{eq:alpha}) nous demandons que tous les $m_j$ sauf un soient nul.
Géométriquement les hyperplans principaux dans $\PW$ sont exactement les hyperplans admissibles passant par $n-2$ sommets du simplexe.

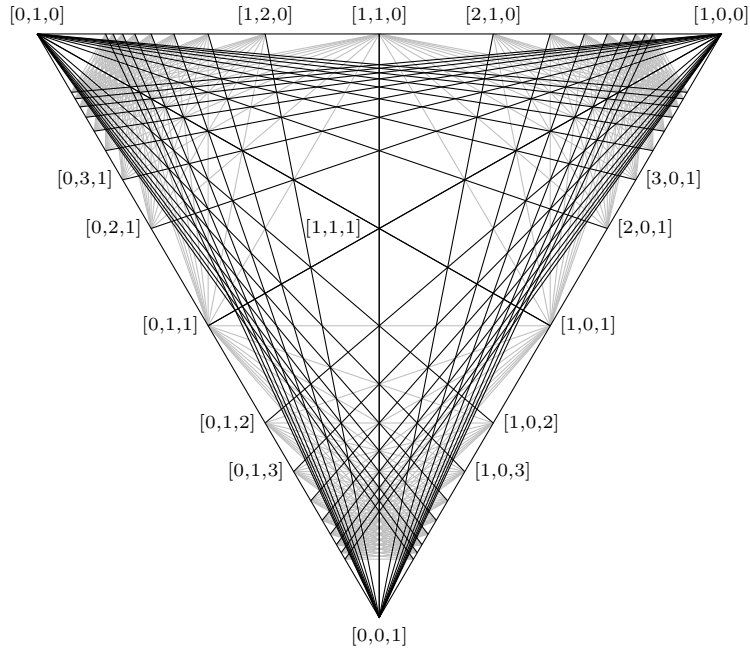
\begin{figure}[ht]
\[
\begin{tikzpicture}[scale = 9,font=\scriptsize]
\coordinate [label=above:$\trip{0}{1}{0}$] (010) at (0,0);
\coordinate [label=above:$\trip{1}{0}{0}$] (100) at (1,0);
\coordinate [label=below:$\trip{0}{0}{1}$] (001) at (.5, -.86);
\coordinate [label=above:$\trip{1}{1}{0}$](110) at ($ (010)!.5!(100) $) {};
\coordinate [label=right:$\trip{1}{0}{1}$] (101) at ($ (001)!.5!(100) $) {};
\coordinate [label=left:$\trip{0}{1}{1}$] (011) at ($ (001)!.5!(010) $) {};
\coordinate [label={left,xshift=-3:$\trip{1}{1}{1}$}] (111) at (intersection of 011--100 and 101--010);
\coordinate [label=above:$\trip{2}{1}{0}$] (210) at ($ (100)!1/(3)!(010) $) {};
\coordinate [label=above:$\trip{1}{2}{0}$] (120) at ($ (010)!1/(3)!(100) $) {};
\coordinate (310) at ($ (100)!1/(4)!(010) $) {};
\coordinate (130) at ($ (010)!1/(4)!(100) $) {};
\foreach \x in {2,...,3}
  {
  \coordinate [label=right:$\trip{\x}{0}{1}$] (\x01) at ($ (100)!1/(\x+1)!(001) $) {};
  \coordinate [label=right:$\trip{1}{0}{\x}$] (10\x) at ($ (001)!1/(\x+1)!(100) $) {};
  \coordinate [label=left:$\trip{0}{\x}{1}$] (0\x1) at ($ (010)!1/(\x+1)!(001) $) {};
  \coordinate [label=left:$\trip{0}{1}{\x}$] (01\x) at ($ (001)!1/(\x+1)!(010) $) {};
  }
\foreach \x in {4,...,9}
  {
  \coordinate (\x10) at ($ (100)!1/(\x+1)!(010) $) {};
  \coordinate (1\x0) at ($ (010)!1/(\x+1)!(100) $) {};
  \coordinate (\x01) at ($ (100)!1/(\x+1)!(001) $) {};
  \coordinate (10\x) at ($ (001)!1/(\x+1)!(100) $) {};
  \coordinate (0\x1) at ($ (010)!1/(\x+1)!(001) $) {};
  \coordinate (01\x) at ($ (001)!1/(\x+1)!(010) $) {};
  }
\foreach \x in {1,...,9}
  \foreach \y in {1,...,9}
    {
    \draw [white!75!black] (0\x1)--(1\y0);
    \draw [white!75!black] (01\x)--(10\y);
    \draw [white!75!black] (\x10)--(\y01);
    }
\foreach \x in {1,...,9}
  {
  \draw (001)--(\x10);
  \draw (001)--(1\x0);
  \draw (010)--(\x01);
  \draw (010)--(10\x);
  \draw (100)--(01\x);
  \draw (100)--(0\x1);
  }
\draw (010)--(100)--(001)--(010);
\end{tikzpicture}
\]
\caption{Quelques droites admissibles, principales ou non, dans le cas $n = 3$.} \label{fig:nervure}
\end{figure}

\begin{lemma} \label{lem:finite}
Pour tout compact $K \subset \PW$, l'ensemble des hyperplans admissibles
rencontrant $K$ est fini.
\end{lemma}

\begin{proof}
Pour tout entier $p \ge 1$, posons
\[
K_p = \left\lbrace [\alpha];\, p \ge \frac{\alpha_1^+}{\alpha_n^+} \right\rbrace.
\]
Comme les $K_p$ forment une exhaustion de $\PW$ par des compacts, il est suffisant de
prouver le lemme pour $K = K_p$.
Or ceci découle de l'observation que si $\sum_{j \neq i} m_j > p$ dans
(\ref{eq:alpha}),
alors pour tout $[\alpha] \in K_p$
\[
\sum_{j \neq i} m_j \alpha_j > p \alpha_n^+ \ge \alpha_1^+ \ge \alpha_i,
\]
et donc l'hyperplan admissible correspondant ne rencontre pas $K_p$.
\end{proof}

Pour $\alpha\in \W$ on définit la \emph{multiplicité} $\mult(\alpha)=\mult([\alpha])$ comme le nombre d'équations admissibles satisfaites par $\alpha$.
Comme conséquence immédiate du lemme~\ref{lem:finite}, on a $\mult([\alpha])<\infty$ pour tout $[\alpha]\in \PW$.
On obtient également :

\begin{remark} \label{rem:nervure}
Soit $[\alpha] \in \PW$, et notons $\Hl_{[\alpha]} \subset \PW$ la réunion de tous les hyperplans admissibles ne passant pas par $[\alpha]$.
Notons $U$ la composante connexe contenant $[\alpha]$ de $\PW \setminus \Hl_{[\alpha]}$.
Alors par le lemme~\ref{lem:finite} $U$ est un voisinage de $[\alpha]$ dans $\PW$, qui par construction n'intersecte aucun autre hyperplan admissible que ceux passant par $[\alpha]$.
\end{remark}

\begin{remark} \label{rem:lieu fixe}
Pour tout $L \subset \PW$ un demi-espace admissible, il existe $g \in \TA$ tel que
$g(\nu_{\id,[\alpha]}) =  \nu_{\id,[\alpha]}$ si et seulement si $[\alpha] \in L$.
En effet, quitte à composer par une permutation on peut supposer que l'inégalité définissant $L$ est de la forme $\alpha_1 \ge \sum_{i \ge 2} m_i \alpha_i$, et alors $g = (x_1 + x_2^{m_2}\dots x_n^{m_n}, x_2, \dots, x_n)$ convient.
\end{remark}

Plus généralement, les demi-espaces admissibles permettent de caractériser le lieu fixé par un automorphisme.

\begin{proposition} \label{pro:lieu fixe}
Soit $f = (f_1, \dots, f_n) \in \TA$ et $\alpha = (\alpha_1, \dots, \alpha_n) \in \W$.
Alors $\nu_{\id,[\alpha]} \in \Fix(f) \cap \Ap_\id$ si et seulement si $\alpha$ satisfait chaque inégalité de la forme
\[
\alpha_{i} \ge \sum_{k = 1}^r m_{j_k} \alpha_{j_k},
\]
où $i = 1,\dots, n$ et $x_{j_1}^{m_{j_1}} \dots  x_{j_r}^{m_{j_r}}$ est un monôme (distinct de $x_i$) apparaissant dans le polynôme~$f_i$.
\end{proposition}

\begin{proof}
On s'intéresse aux classes d'homothétie $\nu_{\id,[\alpha]}$ fixées par $f$, ce qui par le corollaire~\ref{cor:intrin}\ref{cor:intrin3} revient à trouver les valuations $\nu_{\id,\alpha}$ fixées par $f$.

Supposons la valuation $\nu = \nu_{\id,\alpha}$ fixée par~$f$, ce qui revient à dire que pour chaque $i$ on a $-\nu(f_i) = \alpha_i$.
En particulier pour chaque monôme $x_{j_1}^{m_{j_1}} \dots  x_{j_r}^{m_{j_r}}$ distinct de $x_i$ apparaissant dans $f_i$ on doit avoir
\[
\alpha_{i} = -\nu(f_i) \ge
-\nu(x_{j_1}^{m_{j_1}} \dots  x_{j_r}^{m_{j_r}})
=\sum_{k = 1}^r m_{j_k} \alpha_{j_k}.
\]
Ceci implique que $i \not \in \{ j_1, \dots ,j_r\}$, et donne la liste attendue d'équations admissibles.

Réciproquement, supposons que $\alpha$ satisfasse toutes ces inégalités.
Notons $\ell_i$ la partie linéaire de la composante $f_i$, on a donc $(\ell_1, \dots, \ell_n) \in \GL_n(\K)$ et $-\nu(f_i)\geq -\nu(\ell_i)$ pour chaque $i$.
Par le lemme~\ref{lem:formes_lineaire}, pour toute valuation $\nu = \nu_{\id,\alpha}$ nous avons
\[
-\nu(\ell_1)-\cdots -\nu(\ell_n)\geq \alpha_1+\cdots +\alpha_n.
\]
Le fait que $\alpha_i$ satisfasse les inégalités données par les monômes des $f_i$ donne $\alpha_i \ge -\nu(f_i)$ pour tout $i$.
On obtient $\alpha_i \ge -\nu(\ell_i)$ pour tout $i$, et si l'une de ces inégalités était stricte, alors on obtiendrait une contradiction
\[
\alpha_1 +\dots +\alpha_n > -\nu(\ell_1) - \dots -\nu(\ell_n) \ge \alpha_1 +\dots +\alpha_n.
\]
On conclut que $\alpha_i = -\nu(f_i) = -\nu(\ell_i)$ pour tout $i$, et donc $f$ fixe $\nu_{\id,\alpha}$ comme attendu.
\end{proof}

Par la proposition~\ref{pro:lieu fixe}, pour chaque automorphisme linéaire ou élémentaire $f\in \TA$ on a $\Ap_f\cap\Ap_\id\neq \emptyset$. Puisque $\TA$ est engendré par les automorphismes linéaires et élémentaires, on obtient immédiatement :

\begin{corollary}
\label{cor:connexe}
Pour tous $f,g \in \TA$, il existe une suite $f_0, \dots, f_k \in \TA$ tel que $f = f_0$, $g = f_k$, et $\Ap_{f_{i-1}} \cap \Ap_{f_i} \neq \emptyset$ pour chaque $i = 1, \dots, k$.
\end{corollary}

\begin{example}
Si $g = (x_1 + x_2^3, x_2)$ et $h = (x_1, x_2 + x_1^2)$, alors les appartements $\Ap_h$, $\Ap_\id$, $\Ap_g$ et $\Ap_{gh}$ forment une suite d'appartements de $\Xl_2$ tel que chaque couple d'appartements consécutifs s'intersectent, voir figure~\ref{fig:chaine_appart}.
On établira plus loin dans la section~\ref{sec:arbre} que, comme suggéré par la figure, $\Xl_2$ admet une structure d'arbre.
\end{example}

\begin{figure}[ht]
\[
\begin{tikzpicture}[xscale = 2, yscale=1, font=\scriptsize]
\coordinate [label=above:$\nu_{gh, [1,1]}$] (gh 11) at (2,3);
\coordinate [label=above:$\nu_{gh, [2,1]}$] (gh 21) at (3,3);
\coordinate [label=above:$\nu_{gh, [3,1]}$] (gh 31) at (3.7,3);
\coordinate [label=above:$\Ap_{gh}$] (gh 41) at (5,3);
\coordinate [label=above:$\Ap_{g}$](g 13) at (0,2);
\coordinate [label=below:${\nu_{g, [1,2]} = \nu_{gh, [1,2]}}$] (g 12) at (1,2);
\coordinate [label=above:$\nu_{g, [1,1]}$] (g 11) at (2,2);
\coordinate [label=above:$\nu_{g, [2,1]}$] (g 21) at (3,2);
\coordinate (id left) at (-1,1);
\coordinate (id 13) at (0,1);
\coordinate [label=above:${\nu_{\id, [1,2]} = \nu_{h, [1,2]}}$] (id 12) at (1,1);
\coordinate [label=above:$\nu_{\id, [1,1]}$] (id 11) at (2,1);
\coordinate [label=above:$\nu_{\id, [2,1]}$]  (id 21) at (3,1);
\coordinate [label=below:${\nu_{\id, [3,1]} = \nu_{g, [3,1]}}$] (id 31) at (3.7,1);
\coordinate [label=above:$\Ap_{\id}$] (id 41) at (5,1);
\coordinate [label=below:$\nu_{h, [1,1]}$] (h 11) at (2,0);
\coordinate [label=below:$\nu_{h, [2,1]}$] (h 21) at (3,0);
\coordinate [label=below:$\nu_{h, [3,1]}$] (h 31) at (3.7,0);
\coordinate [label=above:$\Ap_{h}$] (h 41) at (5,0);
\draw (g 13)--(g 12)--(gh 11)--(gh 41)
      (g 12)--(g 21)--(id 31)--(id 41)
      (id 12)--(id 31)
      (id 13)--(id 12)--(h 11)--(h 41);
\foreach \v in {gh 11, gh 21, gh 31, g 12, g 11, g 21, id 12, id 11, id 21, id 31, h 11, h 21, h 31} {
\node[circle,fill=black,minimum size=4pt,inner sep=0pt] at (\v) {};
}
\end{tikzpicture}
\]
\caption{Quatre appartements dans $\Xl_2$} \label{fig:chaine_appart}
\end{figure}
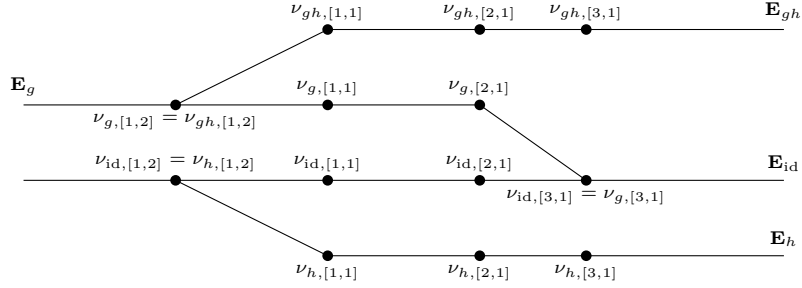

La proposition~\ref{pro:lieu fixe} a aussi les conséquences suivantes :

\begin{corollary}
\label{cor:stabjump}
Soient $\alpha',\alpha''\in \W$. Supposons que chaque demi-espace admissible contenant $\alpha'$ contient $\alpha''$. Alors le stabilisateur dans $\TA$ de $\nu_{\id,[\alpha']}$ est contenu dans le stabilisateur de $\nu_{\id,[\alpha'']}$.
\end{corollary}

\begin{proof}
Soit $f\in \TA$.
Par la proposition~\ref{pro:lieu fixe}, $\Fix(f) \cap \Ap_\id=\{\nu_{\id,[\alpha]}\}$ où $[\alpha]\in L_1\cap \ldots \cap L_k$, intersection de demi-espaces admissibles de $\PW$.
Si $\nu_{\id,[\alpha']}\in \Fix(f)$, alors $[\alpha']\in L_i$ pour chaque $i=1,\ldots, k$. Par l'hypothèse, $[\alpha'']\in L_i$ pour chaque $i=1,\ldots, k$. Donc $\nu_{\id,[\alpha'']}\in \Fix(f)$, comme attendu.
\end{proof}

\begin{corollary}
\label{cor:stab}
Soient $n=3$ et $\alpha\colon (t_0,\infty)\to \mathrm{int}(\W^+)$ une courbe le long de laquelle $\frac{\alpha_1}{\alpha_2}$ et $\frac{\alpha_2}{\alpha_3}$ sont non décroissantes.
Alors les stabilisateurs dans $\TA$ des points $\nu_{\id,[\alpha(t)]}$ forment une famille croissante de groupes.
\end{corollary}

\begin{proof} Soient $t\in (t_0,\infty)$ et $f\in \TA$ tels que $\nu_{\id,[\alpha(t)]}\in \Fix(f)$.
Si $L$ est un demi-espace admissible dans $\PW$ contenant $[\alpha(t)]$, vu que $\alpha_1(t) \ge \alpha_2(t) \ge \alpha_3(t)$, l'inégalité définissant $L$ est de la forme $\alpha_2 \ge m_3\alpha_3$ ou $\alpha_1 \ge m_2\alpha_2+m_3\alpha_3$.
Puisque $\frac{\alpha_1}{\alpha_2}$ et $\frac{\alpha_2}{\alpha_3}$ sont non décroissantes le long de $\alpha(t)$, dans tous les cas $[\alpha(t')]\in L$ pour $t'\geq t$.
Par le corollaire~\ref{cor:stabjump}, on obtient $\nu_{\id,[\alpha(t')]}\in \Fix(f)$.
\end{proof}

\subsection{Autour de \texorpdfstring{$[m,p,1]$}{[m,p,1]}} \label{sec:mp1}

On se place maintenant en dimension $n = 3$, et on va étudier plus précisément les intersections d'appartements autour d'une valuation $\nu_{\id,[\alpha]}$ de poids $\alpha=(m,p,1)$ avec $m \ge p \ge 1$ deux entiers.

\begin{remark}
\label{rem:mp1}
Soient $p \ge 1$ un entier et $\alpha \in \W^+$ avec $\alpha_2 = p \alpha_3$ et $\mult (\alpha) \ge 2$.
Chaque autre équation admissible pour $\alpha$ est de la forme $\alpha_1=m_2\alpha_2+m_3\alpha_3$, ce qui entraîne $\alpha_1=m_2p\alpha_3+m_3\alpha_3=(m_2p+m_3)\alpha_3$.
Donc $\alpha=[m,p,1]$ pour $m=m_2p+m_3$.
\end{remark}

Observons qu'il existe des poids $\alpha \in \W^+$ avec $\mult (\alpha) \ge 2$ qui ne sont pas de cette forme, par exemple $(6,3,2)$ vérifiant $\alpha_1 = 2 \alpha_2$ et $\alpha_1 = 3 \alpha_3$, ou encore $(11,3,2)$ vérifiant $\alpha_1 = 3 \alpha_2 + \alpha_3$ et $\alpha_1 = \alpha_2 + 4 \alpha_3$.

\begin{remark}
\label{rem:mp1_equations}
Fixons $m \ge p \ge 1$ deux entiers, et notons $m = pq + r$ la division euclidienne de $m$ par $p$.
Le poids $\alpha=(m,p,1)$ satisfait exactement $q+2$ équations admissibles, qui sont $\alpha_2 = p\alpha_3$ et $\alpha_1 = a\alpha_2 + (m-pa)\alpha_3$, $a = 0, \dots, q$.
Ces équations correspondent respectivement aux directions de $[m,p,1]$ vers $[1,0,0]$ (ou $[0,p,1]$) et vers $[m-pa,0,1]$ (ou $[a,1,0]$).
\end{remark}

Soit $\alpha$ dans l'intérieur de $\W^+$ et $U \subset \PW^+$ le voisinage de $[\alpha]$ de la remarque~\ref{rem:nervure}.
Soient $f,g\in \TA$ avec $\nu:=\nu_{f,[\alpha]}=\nu_{g,[\alpha]}$.
Notons $U_{f,g}:=\{[\alpha']\in U;\ \nu_{f,[\alpha']}=\nu_{g,[\alpha']}\}\subset \PW^+$.
On dira que $\Ap_f$ et $\Ap_g$ sont \emph{localement équivalents} en $\nu$, noté $f \sim_{\nu} g$, si $U_{f,g}=U$.

Si $\nu=\nu_{\id,[\alpha]}$ avec $\alpha=(m,p,1)$, $m > p > 1$, rappelons que $M_\alpha\subset \Stab(\nu)$ est le sous-groupe des automorphismes
\[
(x_1 + P(x_2, x_3), x_2 + Q(x_3), x_3 + d)
\]
avec $-\nu(P) \leq m$, $-\nu(Q)=\deg Q \leq p$.
Définissons alors $N_\alpha\subset M_\alpha$ comme le sous-groupe des automorphismes avec $-\nu(P) < m$, $\deg Q <p$.

\begin{lemma}
\label{lem:normal}
Soit $\nu=\nu_{\id,[\alpha]}$ de poids $\alpha=(m,p,1)$ avec $m > p > 1$.
\begin{enumerate}
\item \label{lem:normal1}
Pour chaque $f\in \Stab(\nu)$ il existe $g\in M_\alpha$ tel que $\Ap_f^+ = \Ap_g^+$, et donc en particulier $f\sim_\nu g$.
\item \label{lem:normal1bis}
Soit $h \in M_\alpha$. Alors $\id \sim_\nu h$ si et seulement si $h \in N_{\alpha}$.
\item \label{lem:normal2}
$N_\alpha$ est normal dans $M_\alpha$. En particulier pour $f,g\in M_\alpha$ on a $f \sim_\nu g$ si et seulement si $f$ et $g$ deviennent égaux dans le quotient $M_\alpha /N_\alpha$.
\end{enumerate}
\end{lemma}

\begin{proof}
\begin{enumerate}[wide]
\item
On est dans le cadre de l'exemple~\ref{exple:stab}\ref{exple:stab2}, donc la proposition~\ref{pro:stabilisateur} donne
\[
f = (x_1 + P(x_2, x_3), x_2 + Q(x_3), x_3+d) \circ (a_1x_1, a_2x_2, a_3x_3) \in M_\alpha \rtimes L_\alpha,
\]
avec $-\nu(P) \le m$, $\deg Q \le p$.
Par le corollaire~\ref{cor:stab chambre} on a $\Ap^+_\id=\Ap^+_a$ où $a = (a_1x_1, a_2x_2, a_3x_3)$.
On conclut en prenant $ g= (x_1 + P(x_2, x_3), x_2 + Q(x_3), x_3+d)$.

\item
Ce point découle de la proposition~\ref{pro:lieu fixe}, en remarquant que $h \in N_\alpha$ si et seulement si le poids $\alpha$ est contenu dans l'intérieur de chaque demi-espace admissible associé à un monôme de l'une des composantes de $h$.

\item
Considérons
\begin{align*}
f &= \big(x_1 + P(x_2, x_3), x_2 + Q(x_3), x_3+d\big)\in M_\alpha,\\
g &= \big(x_1 + P'(x_2, x_3), x_2 + Q'(x_3), x_3+d'\big)\in N_\alpha,\\
f^{-1} &= \Big(x_1 - P\big(x_2-Q(x_3-d), x_3-d\big), x_2 - Q(x_3-d), x_3-d\Big).
\end{align*}
Posons
\begin{align*}
Q''(x_3)&=Q(x_3)-Q(x_3+d') ,\\
P''(x_2,x_3)&=P(x_2,x_3)-P\big(x_2+Q''(x_3)+Q'(x_3+d),x_3+d'\big).
\end{align*}
Alors
\begin{align*}
f^{-1}gf=f^{-1}&\Big(x_1+P(x_2,x_3)+P'\big(x_2+Q(x_3), x_3+d\big),x_2+Q(x_3)+Q'(x_3+d),x_3+d+d'\Big)\\
=&\Big(x_1+P''(x_2,x_3)+P'\big(x_2+Q(x_3), x_3+d\big),x_2+Q''(x_3)+Q'(x_3+d), x_3+d'\Big)
\end{align*}
Puisque $\deg Q'<p$ et $\deg Q''<\deg Q\leq p$ la deuxième composante de $f^{-1}gf$ est comme attendu. De même,
$-\nu (P')<m$ et $-\nu (P'')<-\nu(P)\leq m$, donc la première composante de $f^{-1}gf$ est comme attendu, et $f^{-1}gf\in N_\alpha$.\qedhere
\end{enumerate}
\end{proof}

Un \emph{secteur centré en $[\alpha]$} dans $\PW$ est l'intersection de deux demi-espaces admissibles dont les droites frontières passent par $[\alpha]$.
La frontière d'un secteur est ainsi l'union de deux demi-droites issues de $[\alpha]$, chacune dirigée vers un point au bord à l'infini de $\PW$.

\begin{lemma} \label{lem:secteur commun}
Soit $\nu =\nu_{\id,[\alpha]}$ de poids $\alpha=(m,p,1)$ avec $m > p > 1$ et $f,g\in M_\alpha$.
Si $V_{f,g} \neq V$, alors $V_{f,g}$ est l'intersection de $V$ avec un secteur $S_{f,g} \subset \PW$ centré en $[\alpha]$,
et $g \sim_\nu fh$ pour un automorphisme triangulaire $h$ que l'on peut choisir de la forme suivante, en fonction des deux demi-droites du secteur $S_{f,g}$:
\begin{enumerate}
\item \label{lem:secteur1}
Si l'une des demi-droites de $S_{f,g}$ est dirigée vers $[m-pa, 0 ,1]$, alors l'autre demi-droite est dirigée vers $[b, 1,0]$ avec $\lfloor \frac{m}{p} \rfloor \ge b \ge a \ge 0$, et
\[
h = (x_1 + \sum_{i = a}^{b} c_i  x_2^i x_3^{m - pi}, x_2, x_3), \quad
c_a \neq 0, c_b \neq 0.
\]

\item \label{lem:secteur2}
Si l'une des demi-droites de $S_{f,g}$ est dirigée vers $[0,p,1]$, alors $S_{f,g}$ est le demi-espace $ \alpha_2/\alpha_3 \ge p$ et
\[
h = (x_1, x_2 + cx_3^p, x_3), \quad c\neq0.
\]

\item \label{lem:secteur3}
Si l'une des demi-droites de $S_{f,g}$ est dirigée vers $[1,0,0]$ et $S_{f,g}$ n'est pas un demi-espace, alors l'autre demi-droite est dirigée vers $[b, 1,0]$ avec $\lfloor \frac{m}{p} \rfloor \ge b \ge 0$ et
\[
h = (x_1 + \sum_{i = 0}^{b} c_i  x_2^i x_3^{m - pi}, x_2 + cx_3^p, x_3),
\quad c_b \neq 0, c\neq0.
\]
\end{enumerate}
\end{lemma}

\begin{proof}
Par le corollaire~\ref{cor:intrin}\ref{cor:intrin2} on peut supposer $f = \id$.
Écrivons
\[g = (x_1 + P(x_2, x_3), x_2 + Q(x_3), x_3+d),\]
avec $-\nu(P) \le m$, $\deg Q \le p$.
Par le lemme~\ref{lem:normal}, quitte à composer par un élément de $N_\alpha$ on peut supposer $d=0$, $Q$ homogène de degré $p$, et $P$ homogène de degré $m$ avec les variables $x_2$, $x_3$ de poids respectifs $p$, $1$.
Ceci donne le $h$ attendu : par la proposition~\ref{pro:lieu fixe} les trois cas de l'énoncé correspondent respectivement à
\ref{lem:secteur1} $Q= 0$,
\ref{lem:secteur2} $P=0$,
et~\ref{lem:secteur3} $P$ et $Q$ tous deux non nuls.
Le cas $Q=P=0$ est exclu par l'hypothèse $V_{\id,g}\neq V$.
\end{proof}

\begin{figure}[ht]
\[\begin{array}{ccc}
\begin{tikzpicture}[scale=3.5,font=\scriptsize]
\coordinate [label=above:$\trip{0}{1}{0}$] (010) at (0,0);
\coordinate [label=above:$\trip{1}{0}{0}$] (100) at (1,0);
\coordinate [label=below:$\trip{0}{0}{1}$] (001) at (.5, -.86);
\coordinate [label=above:$\trip{1}{1}{0}$](110) at ($ (010)!.5!(100) $) {};
\coordinate [label=below right:$\trip{1}{0}{1}$] (101) at ($ (001)!.5!(100) $) {};
\coordinate [label=below left:$\trip{0}{1}{1}$] (011) at ($ (001)!.5!(010) $) {};
\coordinate [label=below right:$\trip{3}{0}{1}$] (301) at ($ (100)!1/4!(001) $) {};
\coordinate [label=left:$\trip{0}{2}{1}$] (021) at ($ (010)!1/3!(001) $) {};
\coordinate [label={[left, yshift=-8, xshift=5]:$[\alpha]$}] (alpha) at (intersection of 021--100 and 010--301);
\fill[lightblue] (301)--(alpha)--(010)--(100)--cycle;
\draw [white!50!black] (010)--(100)--(001)--(010);
\draw (021)--(100);
\draw (110)--(101);
\draw (010)--(301);
\draw [white!75!black] (011)--(100);
\draw [white!75!black] (101)--(010);
\draw [white!75!black] (110)--(001);
\end{tikzpicture}
&
\begin{tikzpicture}[scale=3.5,font={\fontsize{7}{7}\selectfont}]
\coordinate [label=above:$\trip{0}{1}{0}$] (010) at (0,0);
\coordinate [label=above:$\trip{1}{0}{0}$] (100) at (1,0);
\coordinate [label=below:$\trip{0}{0}{1}$] (001) at (.5, -.86);
\coordinate [label=above:$\trip{1}{1}{0}$](110) at ($ (010)!.5!(100) $) {};
\coordinate [label=below right:$\trip{1}{0}{1}$] (101) at ($ (001)!.5!(100) $) {};
\coordinate [label=below left:$\trip{0}{1}{1}$] (011) at ($ (001)!.5!(010) $) {};
\coordinate [label=below right:$\trip{3}{0}{1}$] (301) at ($ (100)!1/4!(001) $) {};
\coordinate [label=left:$\trip{0}{2}{1}$] (021) at ($ (010)!1/3!(001) $) {};
\coordinate [label={[left, yshift=-8, xshift=5]:$[\alpha]$}] (alpha) at (intersection of 021--100 and 010--301);
\fill[lightblue] (301)--(alpha)--(110)--(100)--cycle;
\draw [white!50!black] (010)--(100)--(001)--(010);
\draw (021)--(100);
\draw (110)--(101);
\draw (010)--(301);
\draw [white!75!black] (011)--(100);
\draw [white!75!black] (101)--(010);
\draw [white!75!black] (110)--(001);
\end{tikzpicture}
&
\begin{tikzpicture}[scale=3.5,font={\fontsize{7}{7}\selectfont}]
\coordinate [label=above:$\trip{0}{1}{0}$] (010) at (0,0);
\coordinate [label=above:$\trip{1}{0}{0}$] (100) at (1,0);
\coordinate [label=below:$\trip{0}{0}{1}$] (001) at (.5, -.86);
\coordinate [label=above:$\trip{1}{1}{0}$](110) at ($ (010)!.5!(100) $) {};
\coordinate [label=below right:$\trip{1}{0}{1}$] (101) at ($ (001)!.5!(100) $) {};
\coordinate [label=below left:$\trip{0}{1}{1}$] (011) at ($ (001)!.5!(010) $) {};
\coordinate [label=below right:$\trip{3}{0}{1}$] (301) at ($ (100)!1/4!(001) $) {};
\coordinate [label=left:$\trip{0}{2}{1}$] (021) at ($ (010)!1/3!(001) $) {};
\coordinate [label={[left, yshift=-8, xshift=5]:$[\alpha]$}] (alpha) at (intersection of 021--100 and 010--301);
\fill[lightblue] (101)--(alpha)--(110)--(100)--cycle;
\draw [white!50!black] (010)--(100)--(001)--(010);
\draw (021)--(100);
\draw (110)--(101);
\draw (010)--(301);
\draw [white!75!black] (011)--(100);
\draw [white!75!black] (101)--(010);
\draw [white!75!black] (110)--(001);
\end{tikzpicture}
\\
\ref{lem:secteur1}\, a=b=0
&
\ref{lem:secteur1}\, a=0,\, b=1
&
\ref{lem:secteur1}\, a=b=1
\\
&& \\
\begin{tikzpicture}[scale=3.5,font={\fontsize{7}{7}\selectfont}]
\coordinate [label=above:$\trip{0}{1}{0}$] (010) at (0,0);
\coordinate [label=above:$\trip{1}{0}{0}$] (100) at (1,0);
\coordinate [label=below:$\trip{0}{0}{1}$] (001) at (.5, -.86);
\coordinate [label=above:$\trip{1}{1}{0}$](110) at ($ (010)!.5!(100) $) {};
\coordinate [label=below right:$\trip{1}{0}{1}$] (101) at ($ (001)!.5!(100) $) {};
\coordinate [label=below left:$\trip{0}{1}{1}$] (011) at ($ (001)!.5!(010) $) {};
\coordinate [label=below right:$\trip{3}{0}{1}$] (301) at ($ (100)!1/4!(001) $) {};
\coordinate [label=left:$\trip{0}{2}{1}$] (021) at ($ (010)!1/3!(001) $) {};
\coordinate [label={[left, yshift=-8, xshift=5]:$[\alpha]$}] (alpha) at (intersection of 021--100 and 010--301);
\fill[lightblue] (100)--(alpha)--(021)--(010)--cycle;
\draw [white!50!black] (010)--(100)--(001)--(010);
\draw (021)--(100);
\draw (110)--(101);
\draw (010)--(301);
\draw [white!75!black] (011)--(100);
\draw [white!75!black] (101)--(010);
\draw [white!75!black] (110)--(001);
\end{tikzpicture}
&
\begin{tikzpicture}[scale=3.5,font={\fontsize{7}{7}\selectfont}]
\coordinate [label=above:$\trip{0}{1}{0}$] (010) at (0,0);
\coordinate [label=above:$\trip{1}{0}{0}$] (100) at (1,0);
\coordinate [label=below:$\trip{0}{0}{1}$] (001) at (.5, -.86);
\coordinate [label=above:$\trip{1}{1}{0}$](110) at ($ (010)!.5!(100) $) {};
\coordinate [label=below right:$\trip{1}{0}{1}$] (101) at ($ (001)!.5!(100) $) {};
\coordinate [label=below left:$\trip{0}{1}{1}$] (011) at ($ (001)!.5!(010) $) {};
\coordinate [label=below right:$\trip{3}{0}{1}$] (301) at ($ (100)!1/4!(001) $) {};
\coordinate [label=left:$\trip{0}{2}{1}$] (021) at ($ (010)!1/3!(001) $) {};
\coordinate [label={[left, yshift=-8, xshift=5]:$[\alpha]$}] (alpha) at (intersection of 021--100 and 010--301);
\fill[lightblue] (100)--(alpha)--(010)--cycle;
\draw [white!50!black] (010)--(100)--(001)--(010);
\draw (021)--(100);
\draw (110)--(101);
\draw (010)--(301);
\draw [white!75!black] (011)--(100);
\draw [white!75!black] (101)--(010);
\draw [white!75!black] (110)--(001);
\end{tikzpicture}
&
\begin{tikzpicture}[scale=3.5,font={\fontsize{7}{7}\selectfont}]
\coordinate [label=above:$\trip{0}{1}{0}$] (010) at (0,0);
\coordinate [label=above:$\trip{1}{0}{0}$] (100) at (1,0);
\coordinate [label=below:$\trip{0}{0}{1}$] (001) at (.5, -.86);
\coordinate [label=above:$\trip{1}{1}{0}$](110) at ($ (010)!.5!(100) $) {};
\coordinate [label=below right:$\trip{1}{0}{1}$] (101) at ($ (001)!.5!(100) $) {};
\coordinate [label=below left:$\trip{0}{1}{1}$] (011) at ($ (001)!.5!(010) $) {};
\coordinate [label=below right:$\trip{3}{0}{1}$] (301) at ($ (100)!1/4!(001) $) {};
\coordinate [label=left:$\trip{0}{2}{1}$] (021) at ($ (010)!1/3!(001) $) {};
\coordinate [label={[left, yshift=-8, xshift=5]:$[\alpha]$}] (alpha) at (intersection of 021--100 and 010--301);
\fill[lightblue] (100)--(alpha)--(110)--cycle;
\draw [white!50!black] (010)--(100)--(001)--(010);
\draw (021)--(100);
\draw (110)--(101);
\draw (010)--(301);
\draw [white!75!black] (011)--(100);
\draw [white!75!black] (101)--(010);
\draw [white!75!black] (110)--(001);
\end{tikzpicture}
\\
\ref{lem:secteur2}
&
\ref{lem:secteur3}\, b=0
&
\ref{lem:secteur3}\, b=1
\end{array}
\]
\caption{Les 3 droites admissibles de $\PW$ passant par $[\alpha] =[3,2,1]$, et les différents secteurs $S_{f,g}$ possibles.}
\label{fig:321}
\end{figure}
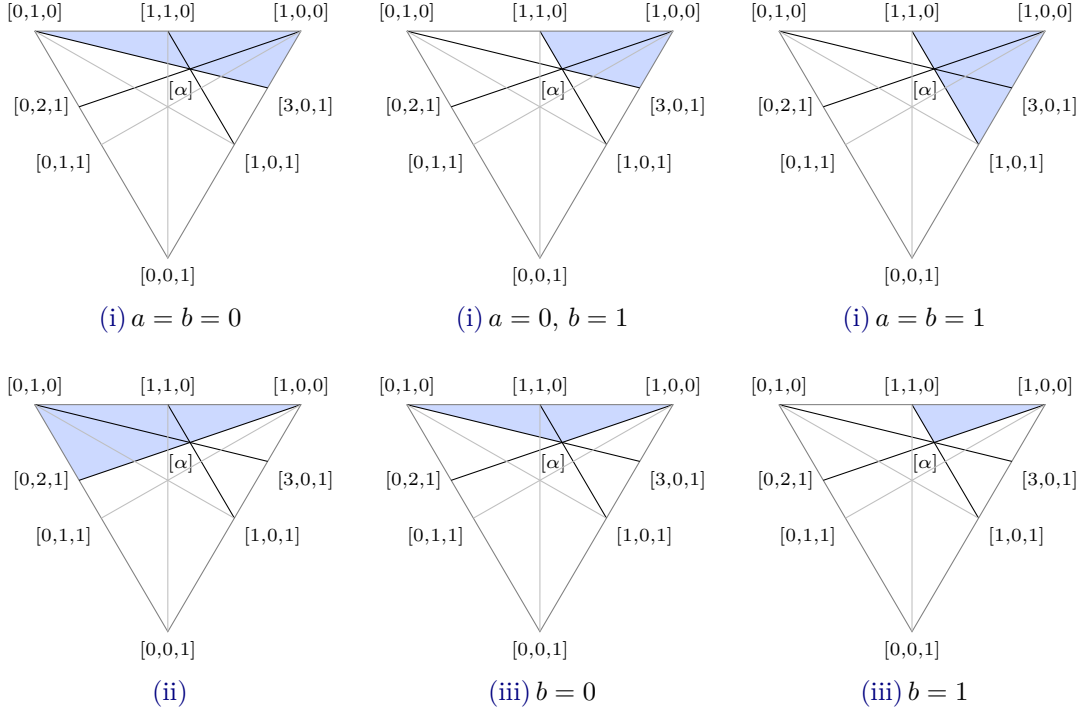

\begin{example}[figure~\ref{fig:321}]
Si $\alpha = (3,2,1)$, il y a exactement 3 droites admissibles passant par
$[\alpha]$, correspondant aux équations:
\begin{align*}
\alpha_2 &=  2\alpha_3; &
\alpha_1 &=  3\alpha_3; &
\alpha_1 &=  \alpha_2 + \alpha_3.
\end{align*}
Observer que l'équation $2\alpha_1 = 3\alpha_2$ est également satisfaite par $\alpha$, mais par définition n'est pas une équation admissible.
Pour tout choix des coefficients $c_i$, l'automorphisme suivant est un élément
de $\Stab(\nu_{\id,[\alpha]})$:
\[
(x_1 + c_0 x_3^3 + c_1 x_2x_3 , x_2 + c_2 x_3^2, x_3).
\]
En annulant certains coefficients parmi les $c_i$, on peut réaliser chacun des six secteurs $S_{f,g}$ centrés en $[\alpha]$ prédits par le lemme~\ref{lem:secteur commun}.
\end{example}

\section{Distance} \label{sec:metrique}

\subsection{Espace de longueur}
D'après le corollaire~\ref{cor:intrin}\ref{cor:intrin2}, $\Xl$ est la réunion des translatés de $\Ap^+_{\id}$ sous l'action de $\Tame(\A^n)$.
Nous pouvons donc identifier $\Xl$ avec $\Big(\bigsqcup \Ap^+_f\Big)/\sim$, réunion disjointe de copies de $\Ap_{\id}^+$ indicées par  $\Tame(\A^n)$, quotientée par une relation d’équivalence~$\sim$.
On commet ici l'abus d'écriture de noter $\Ap^+_f$ aussi bien la chambre dans $\Xl$ que sa copie qui nous sert dans la construction abstraite par union disjointe et quotient.

Afin de munir $\Xl$ d'une distance, concentrons-nous d'abord sur (chaque copie de) $\Ap_{\id}^+$ et plus précisément, via l'application $\wgt$ du corollaire~\ref{cor:intrin}\ref{cor:intrin1}, sur le simplexe des poids projectivisés $\PW^+$.
Nous identifions chaque $[\alpha]\in \PW^+$ avec son représentant dans $\W^+$ contenu dans l'hyperboloïde $\prod \alpha_i = 1$.
Passant aux logarithmes $\beta_i=\log \alpha_i$, nous obtenons
\[
\PW^+ = \bigl\lbrace [\alpha]; \alpha = (\exp \beta_1, \dots, \exp \beta_n), \beta_1\geq  \dots \geq \beta_n, \sum \beta_i = 0 \bigr\rbrace.
\]
Nous munissons alors $\PW^+$ de la distance $|\cdot,\cdot|$ induite par la distance euclidienne de $\R^n=\{(\beta_1,\ldots,\beta_n)\}$.
Cela rend $\PW^+$ isométrique à la chambre de Weyl définie par les inégalités $\beta_1\geq  \dots \geq \beta_n$ dans
\[\R^{n-1} =  \bigl\lbrace (\beta_1, \dots, \beta_n) \in \R^n;\; \sum \beta_i = 0 \bigr\rbrace .\]

Notons que la même application de passage aux logarithmes identifie le simplexe $\PW$ entier à $(\R^{n-1},|\cdot,\cdot|)$.
Ce procédé de définir une distance en passant aux logarithmes des poids est celui utilisé dans \cite{BT, Parreau2000} dans le contexte de l'immeuble de $\SL_n(\F)$.
Si on ne passait pas aux logarithmes, l'espace total obtenu ne serait plus à courbure négative ou nulle, voir plus loin l'exemple~\ref{exple:angles}.
Un autre choix \textit{a priori} naturel serait d'équiper le simplexe~$\PW$ de la distance de Hilbert. Cependant dès que le bord d'un domaine convexe contient deux segments de droite engendrant un plan, la distance de Hilbert n'est pas uniquement géodésique \cite[Thm~5.6.8]{Pa}, et en particulier n'est pas $\CAT(0)$.

\begin{lemma}
\label{lem:convex}
Chaque demi-espace admissible de $\PW$ est un domaine convexe sous l'identification avec $(\R^{n-1},|\cdot,\cdot|)$.
\end{lemma}

\begin{proof}
Considérons une équation admissible  $\alpha_1 = \sum_{j >1} m_j \alpha_j$.
Le domaine des solutions de l'inégalité $\alpha_1 \geq  \sum_{j >1} m_j \alpha_j$ dans $(\R^{n-1},|\cdot,\cdot|)$ étant fermé, pour obtenir sa convexité il suffit de démontrer que pour chaque $\beta,\beta'$ à la frontière du domaine, le point $\frac{\beta+\beta'}{2}$ est dans le domaine.
Pour de tels $(\beta_i)=(\log \alpha_i)$, $(\beta'_i)=(\log \alpha'_i)$ on a $\alpha_1 = \sum_{j >1} m_j \alpha_j$ et $\alpha'_1 = \sum_{j >1} m_j \alpha'_j$.
On doit donc vérifier

\begin{multline*}
\exp\bigg(\frac{\log (m_2\alpha_2+\cdots +m_n\alpha_n)+\log(m_2\alpha'_2+\cdots +m_n\alpha'_n)}{2}\bigg) \geq  \\
 m_2\exp\bigg(\frac{\log\alpha_2+\log\alpha_2'}{2}\bigg)+ \cdots+m_n\exp\bigg(\frac{\log\alpha_n+\log\alpha_n'}{2}\bigg),
\end{multline*}
que l'on peut encore écrire
\begin{align*}
\sqrt{(m_2\alpha_2+\cdots +m_n\alpha_n)(m_2\alpha'_2+\cdots +m_n\alpha'_n)}&\geq \sqrt{(m_2\alpha_2)(m_2\alpha'_2)}+\cdots+\sqrt{(m_n\alpha_n)(m_n\alpha'_n)}.
\end{align*}
En passant au carré, on obtient l'inégalité équivalente :

\begin{align*}
\frac12 \sum_{j,k}\big((m_j\alpha_j)(m_k\alpha'_k)+(m_k\alpha_k)(m_j\alpha'_j)\big)&\geq \sum_{j,k}\sqrt{(m_j\alpha_j)(m_j\alpha'_j)(m_k\alpha_k)(m_k\alpha'_k)}.
\end{align*}
Cette dernière découle directement de la classique inégalité des moyennes arithmético-géomé\-triques $\frac{x + y}{2} \ge \sqrt{xy}$.
\end{proof}

\begin{remark}
\label{rem:droites}
Les hyperplans principaux sont encore des hyperplans pour la distance $|\cdot,\cdot|$.
Pour $n=3$ les droites principales dans $\PW$ forment 3 familles des droites parallèles s'intersectant avec des angles $\pi/3$.
On représente sur la figure~\ref{fig:homeo} les droites principales dans $\PW$ avec à gauche la distance du simplexe, et à droite la distance $(\R^2,|\cdot,\cdot|)$.
Sont également représentées, en rouge et bleu, deux droites admissibles mais non principales de $\PW$: leurs images dans $\R^2$ deviennent des courbes, et le demi-espace admissible devient un convexe.
\end{remark}

\begin{figure}[ht]
\begin{align*}
\begin{tikzpicture}[scale = 6]
\coordinate (010) at (0,0);
\coordinate (100) at (1,0);
\coordinate (001) at (.5, -.86);
\coordinate (010') at (0.01,-0.007);
\coordinate (100') at (0.99,-0.007);
\coordinate (001') at (.5, -.84);
\coordinate (110) at ($ (010)!.5!(100) $) {};
\coordinate (101) at ($ (001)!.5!(100) $) {};
\coordinate (011) at ($ (001)!.5!(010) $) {};
\coordinate (111) at (intersection of 011--100 and 101--010);
\foreach \x in {2,...,59}
  {
  \coordinate (\x10) at ($ (100)!1/(\x+1)!(010) $) {};
  \coordinate (1\x0) at ($ (010)!1/(\x+1)!(100) $) {};
  \coordinate (\x01) at ($ (100)!1/(\x+1)!(001) $) {};
  \coordinate (10\x) at ($ (001)!1/(\x+1)!(100) $) {};
  \coordinate (0\x1) at ($ (010)!1/(\x+1)!(001) $) {};
  \coordinate (01\x) at ($ (001)!1/(\x+1)!(010) $) {};
  }
\foreach \x in {1,...,59}
  {
  \draw[very thin] (001) -- (\x10);
  \draw[very thin] (001) -- (1\x0);
  \draw[very thin] (010) -- (\x01);
  \draw[very thin] (010) -- (10\x);
  \draw[very thin] (100) -- (01\x);
  \draw[very thin] (100) -- (0\x1);
  }
\draw[thick] (010) -- (100) -- (001) -- cycle;
\draw[line width=3] (010') -- (100') -- (001') -- cycle;
\draw[red, thick] (110)--(101);
\draw[blue, thick] (110)--(201);
\end{tikzpicture}
&&
\begin{tikzpicture}[scale = 1.03, radius = 2.6]
\clip(0,0) circle;
\foreach \c in {1,...,20}
  {
  \coordinate (\c-1) at (-{ln(\c)},-3) {};
  \coordinate (\c-2) at (-{ln(\c)},+3) {};
  \coordinate (\c+1) at ({ln(\c)},-3) {};
  \coordinate (\c+2) at ({ln(\c)},+3) {};
  \coordinate (red\c) at ( {ln(\c+1)}, {(1 / sqrt(3)) * (ln(\c+1) - 2*ln(\c))} ) {};
  \coordinate (red+\c) at ( {ln(\c+1) - ln(\c)}, {(1 / sqrt(3)) * (ln(\c+1) + ln(\c))} ) {};
  \coordinate (blue\c) at ( {ln(2*\c+1)}, {(1 / sqrt(3)) * (ln(2*\c+1) - 2* ln(\c))} ) {};
  \coordinate (blue+\c) at ( {ln(\c+2)-ln(\c)}, {(1 / sqrt(3)) * (ln(\c+2) + ln(\c))} ) {};
  }
\foreach \c [evaluate=\c as \d using \c+1] in {1,...,19}
  {
  \draw[thin] (\c-1)--(\c-2);
  \draw[thin] (\c+1)--(\c+2);
  \draw[thin] plot [domain=-3:3,samples=1] (\x,{(1 / sqrt(3)) * (\x + 2*ln(\c))});
  \draw[thin] plot [domain=-3:3,samples=1] (\x,{(1 / sqrt(3)) * (\x - 2*ln(\c))});
  \draw[thin] plot [domain=-3:3,samples=1] (\x,{(-1 / sqrt(3)) * (\x + 2*ln(\c))});
  \draw[thin] plot [domain=-3:3,samples=1] (\x,{(-1 / sqrt(3)) * (\x - 2*ln(\c))});
  \draw [red,thick] (red\c)--(red\d);
  \draw [red,thick] (red+\c)--(red+\d);
  \draw [blue,thick] (blue\c)--(blue\d);
  \draw [blue,thick] (blue+\c)--(blue+\d);
  }
\end{tikzpicture}
\end{align*}
\caption{Homéomorphisme entre $\PW$ et $\R^2$.}
\label{fig:homeo}
\end{figure}
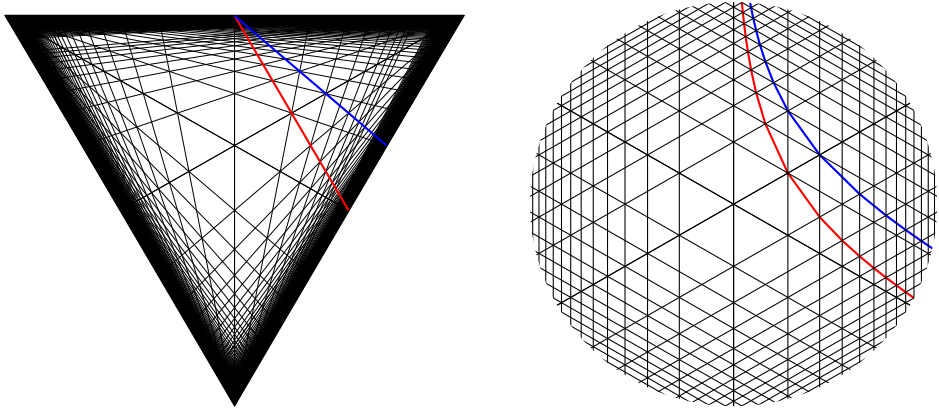

À travers l'identification $[\alpha]\mapsto \nu_{f,[\alpha]}$ de $\PW^+$ avec $\Ap_{f}^+$, on munit chaque chambre $\Ap_{f}^+$ d'une distance que l'on note encore $|\cdot,\cdot|$.
Suivant \cite[I.5.19]{BH}, nous définissons sur $X$ une \emph{pseudo-distance quotient}:

\begin{definition}
\label{def:chaîne}
Une \emph{chaîne} (de \emph{longueur} $k$)
est une suite
\[
(x'_0,x_1\sim x_1',x_2\sim x_2',\ldots,x_{k-1}\sim x'_{k-1},x_k)
\]
dans $\bigsqcup \Ap^+_f$ telle qu'il existe $f_0,\ldots, f_{k-1}\in\TA$ avec $x_i',x_{i+1}\in \Ap_{f_i}^+$.
Considérons maintenant
\[\overline x,\overline y\in \Xl=\Big(\bigsqcup \Ap^+_f\Big)/\sim.\]
La pseudo-distance quotient $d_\Xl(\overline x,\overline y)$ est définie comme l'infimum des sommes $\sum_{i=0}^{k-1}|x'_i,x_{i+1}|$, pour tous les entiers $k$ et toutes les chaînes de longueur $k$, avec $x_0'\in\overline x$ (c'est-à-dire $x'_0$ est un représentant de $\overline x$) et $x_k\in\overline y$.
\end{definition}

\begin{proposition} \label{pro:metrique}
La pseudo-distance quotient $d_\Xl$ est une distance, et $\Xl$ muni de cette distance est un espace de longueur.
\end{proposition}

\begin{proof}
Pour montrer que $d_\Xl$ est une distance, d'après \cite[I.5.28]{BH} il suffit de vérifier les deux propriétés ci-dessous.
De plus, $\Xl$ sera alors un espace de longueur par \cite[I.5.20]{BH}.
\begin{enumerate}
\item \label{pro:metrique1}
Pour tous $x,z\in \Ap^+_f, x',z'\in \Ap^+_{f'}$ tels que $x\sim x'$ et $z\sim z'$, on a $|x,z|=|x',z'|$.
\item \label{pro:metrique2}
Pour tout $\overline x\in \Xl$ il existe $\eps=\eps (\overline x) > 0$ tel que $\bigcup_{x \in \overline x}B(x,\eps)$ est une réunion de classes d’équivalence pour la relation $\sim$, où $B(x,\eps)$ dénote une boule ouverte dans chaque $(\Ap_f^+,|\cdot,\cdot|)$.
\end{enumerate}

La propriété~\ref{pro:metrique1} est une conséquence directe de la proposition~\ref{pro:alpha intrinseque}.

Pour établir la propriété~\ref{pro:metrique2}, si $x=\nu_{f,[\alpha]}\in \Ap^+_{f}$, choisissons d'abord $\eps >0$ suffisamment petit pour que sous l'identification de $\Ap^+_{f}$ avec $\PW^+\subset \PW$, le voisinage $U$ de $[\alpha]$ de la remarque~\ref{rem:nervure} contienne $B(x,\eps)$.
Ainsi $B(x,\eps)$ n'intersecte aucun autre hyperplan admissible que ceux passant par $[\alpha]$.
Ceux-ci sont lisses dans les coordonnées de $(\R^{n-1},|\cdot,\cdot|)$, et sont en nombre fini à nouveau par le lemme~\ref{lem:finite}, donc quitte à diminuer $\eps$, pour chaque $y\in B(x,\eps)$ il existe un chemin de $y$ vers $x$ dans $B(x,\eps)$ le long duquel la multiplicité est constante, sauf peut-être en l'extrémité $x$ où elle peut augmenter.
Notons que $\eps=\eps([\alpha])$ ne dépend pas de $f$ mais seulement de $[\alpha]$, grâce à la proposition~\ref{pro:alpha intrinseque}.
Ainsi $\eps(\overline x)$ est bien défini.

Considérons maintenant $y\in B(x,\eps)\subset \Ap^+_{f}$, et soit $y'\sim y$ avec $y'\in \Ap^+_{f'}$.
En utilisant l'action de $\TA$ (et en remplaçant $f$ par $f'^{-1}f$) on peut supposer $f'=\id$.
Par la proposition~\ref{pro:alpha intrinseque}, la relation $y'\sim y$ se traduit par $\overline{y}\in \Fix(f)$.
Par définition de $\eps$, il existe un chemin de $y$ vers $x$ dont tous les points sont contenus dans le même ensemble d'hyperplans admissibles (sauf peut-être $x$ qui est alors contenu dans un ensemble d'hyperplans encore plus grand).
Alors par le corollaire~\ref{cor:stabjump}, on a $\overline{x}\in \Fix(f)$, et donc le point $x'\in \Ap^+_{\id}$ de poids $[\alpha]$ satisfait $x'\sim x$.
Finalement $y'\in \bigcup_{x \in \overline x}B(x,\eps)$, et cet ensemble est donc bien saturé pour la relation d'équivalence $\sim$.
\end{proof}

Dans le lemme qui suit nous conservons les notations $B(z,\eps)$ et $\overline {B}(z, \eps)$ pour les boules respectivement ouvertes ou fermées dans une chambre $\Ap^+_f$, et nous utilisons les notations $B_\Xl(\overline z, \eps)$, $\overline {B}_\Xl(\overline z, \eps)$ pour les boules dans $\Xl$ par rapport à $d_\Xl$.

\begin{lemma}
\label{lem:boule}
Soit $\overline z\in \Xl$ et $\eps = \eps(\overline z)$ vérifiant la propriété~\ref{pro:metrique2} ci-dessus.
Soit $\V(\overline z,\eps)=\bigsqcup_{z \in \overline z}B(z,\eps)/\sim$ avec la pseudo-distance quotient induite de $\bigsqcup_{z \in \overline z}B(z,\eps)$.
De même, soit $\overline{\V}(\overline z,{\frac{\eps}{4}})=\bigsqcup_{z \in \overline z}\overline{B}(z,{\frac{\eps}{4}})/\sim$ avec la pseudo-distance quotient induite de $\bigsqcup_{z \in \overline z}\overline{B}(z,{\frac{\eps}{4}})$. Alors
\begin{enumerate}[$(a)$] 
\item \label{boule_a}
$\V(\overline z,\eps)\subset \Xl$ coïncide comme ensemble avec $B_{\Xl}(\overline z,\eps)$, et
\item \label{boule_b}
$\overline{\V}(\overline z,{\frac{\eps}{4}})$ est isométrique à la boule $\overline{B}_{\Xl}(\overline z,\frac{\eps}{4})$.
\end{enumerate}
\end{lemma}

\begin{proof}
D'après \cite[I.5.27(1,2)]{BH}, l'ensemble $\V(\overline z,\eps)\subset \Xl$ coïncide avec $B_{\Xl}(\overline z,\eps)$ et l'ensemble
$\overline{\V}(\overline z,{\frac{\eps}{4}})\subset \Xl$ coïncide avec $\overline{B}_{\Xl}(\overline z,{\frac{\eps}{4}})$.
De plus, d'après \cite[I.5.27(3)]{BH}, cette dernière identification est une isométrie pour les distances induites par les espaces ambiants $\V(\overline z,\eps)$ et $\Xl$.
Observer que nous prenons un rayon $\eps/4$, au lieu de $\eps/2$ dans \cite{BH}, parce que nous voulons travailler ici avec des boules fermées.

Il reste à remarquer que $\overline{\V}(\overline z,\frac{\eps}{4})$ est isométriquement plongé dans $\V(\overline z,\eps)$, car pour tous $\overline x, \overline y$ dans $\overline{\V}(\overline z,{\frac{\eps}{4}})$ et toute chaîne dans $\bigsqcup_{z \in \overline z}B(z,\eps)$ de la forme $(x'_0\in \overline x,x_1\sim x_1',\ldots,x_k\in \overline y)$ comme précédemment, nous pouvons la remplacer par la chaîne suivante. Les points $x_i',x_i$ appartiennent aux boules $B(z_i, \eps), B(z_{i-1},\eps)$ dans $\Ap^+_{f_i},\Ap^+_{f_{i-1}}$, pour $z_i,z_{i-1}$ dans $\Ap^+_{f_i},\Ap^+_{f_{i-1}}$ représentant $\overline {z}$.
Appelons $p(x_i')$ la projection radiale (par rapport au centre $z_i$) de $x_i'$ sur $\overline {B}(z_i, \frac{\eps}{4})$ dans $\Ap^+_{f_i}$ et définissons de même $p(x_i)\in \overline {B}(z_{i-1}, \frac{\eps}{4})$.
Par la définition d'une chaîne on a $x_i\sim x_i'$.
Par la proposition 2.3 cela veut dire, en supposant pour simplicité que $f_{i-1}=\id$, que $f_i$ fixe $\overline{x}_i$.
Puisque $f_i$ fixe aussi $\overline{z}$, par la convexité de $\Fix(f_i)$ dans $\Ap^+_\id$ qui vient du lemme~5.1 et de la proposition 4.5, on obtient que $f_i$ fixe $\overline{p(x_i)}$. Alors $p(x_i)\sim p(x_i')$ et les projections forment une chaîne avec $\sum_{i=0}^{k-1}|p(x'_i),p(x_{i+1})|\leq\sum_{i=0}^{k-1}|x'_i,x_{i+1}|$.
\end{proof}

\begin{lemma}
\label{lem:restriction}
\begin{enumerate}[wide]
\item L'application $\wgt\colon (\Xl, d_\Xl)\to (\PW^+,|\cdot,\cdot|)$ du corollaire~$\ref{cor:intrin}\ref{cor:intrin1}$
est une isométrie en restriction à chaque $\Ap^+_f$.
\item L'application $\rho\colon (\Xl, d_\Xl)\to (\PW,|\cdot,\cdot|)$ du lemme~$\ref{lem:rho}$ est une isométrie en restriction à chaque $\Ap^+_f$.
\item Les applications $\wgt\colon \Xl \to \PW^+$ et $\rho\colon \Xl \to \PW$ n'accroissent pas les distances.
\end{enumerate}
\end{lemma}

\begin{proof}
\begin{enumerate}[wide]
\item Soit $\overline{x},\overline {y}\in \Ap^+_f$. Soit $(x'_0,x_1\sim x_1',x_2\sim x_2',\ldots,x_{k-1}\sim x'_{k-1},x_k)$ une chaîne avec $x'_0\in \overline{x}, x_k\in \overline{y}$.
Alors $\sum_{i=0}^{k-1}|x'_i,x_{i+1}|=\sum_{i=0}^{k-1} | \wgt(\overline{x}_{i}), \wgt(\overline{x}_{i+1}) |\geq |\wgt(\overline{x}), \wgt(\overline{y})|$ avec l'égalité pour la chaîne triviale de longeur $k=2$.
Ainsi $d_\Xl(\overline{x}, \overline {y})= |\wgt(\overline{x}), \wgt(\overline{y})|$.
\item Par le corollaire~\ref{cor:stab chambre} on peut supposer $f\in \Tame_0(\A^n)$. Par définition pour $\alpha\in \W^+$ on a $\rho(\nu_{f,[\alpha]})=[\sigma_f(\alpha)]$.
Alors le point~(ii) découle du point~(i) et du fait que l'application $[\alpha]\to [\sigma_f(\alpha)]$ est une isométrie de $(\PW^+,|\cdot,\cdot|)$ sur son image dans $(\PW,|\cdot,\cdot|)$.
\item Grâce aux points précédents c'est une conséquence immédiate de la définition de la distance $d_\Xl$. \qedhere
\end{enumerate}
\end{proof}

\begin{remark}
\label{rem:plonge}
L'application $e\colon (\PW,|\cdot,\cdot|) \to (\Xl, d_\Xl)$ définie par $e([\alpha])=\nu_{\id,[\alpha]}$ n'accroît pas la distance.
Par le lemme~\ref{lem:restriction}(iii), $\rho$ n'accroît pas la distance non plus.
De plus, par le lemme~\ref{lem:rho}\ref{lem:rho_goodprojection},
$\rho\circ e$ est l'identité. Ainsi, $\Ap_\id$ identifié avec $\R^{n-1}$ est isométriquement plongé dans $\Xl$. En utilisant l'action de $\TA$ on obtient que chaque appartement $\Ap_f$ est un $\R^{n-1}$ isométriquement plongé, ce qui justifie la notation $\Ap$ pour «Euclidien».
\end{remark}

\begin{lemma}
\label{lem:complet}
L'espace $\Xl$ est complet.
\end{lemma}

\begin{proof}
Soit $(\nu_m)$ une suite de Cauchy dans $\Xl$.
Par le lemme~\ref{lem:restriction}(iii), $(\wgt(\nu_m))$ est aussi une suite de Cauchy, et admet une limite $[\alpha]\in \PW^+$ puisque $(\PW^+,|\cdot,\cdot|)$ est fermé dans $\R^{n-1}$.
Soit $\eps=\eps([\alpha])$ correspondant à la propriété~\ref{pro:metrique2} dans la preuve de la proposition~\ref{pro:metrique}.
Soit $M$ suffisamment grand tel que pour chaque $m\geq M$ on ait
\[
d_\Xl(\nu_M,\nu_{m}) < \frac{\eps}{2}
\text{ et }
|\wgt(\nu_{m}),[\alpha]| < \frac{\eps}{2}.
\]
Si $\nu_M\in \Ap^+_{f}$, soit $\overline{z}=\nu_{f,[\alpha]}$.
Ainsi pour $m\geq M$ on a $\nu_{m}\in B_{\Xl}(\overline{z},\eps)=\V(\overline{z},\eps)$ par le lemme~\ref{lem:boule}\ref{boule_a}.
Par définition de $\V(\overline{z},\eps)$, pour chaque $m\geq M$ il existe $z_{m}\in \overline{z}$ avec $\nu_{m}\in B(z_{m},\eps)$.
Donc $d_\Xl(\overline{z}, \nu_{m})\leq |[\alpha],\wgt(\nu_{m})|$ qui tend vers 0 quand $m$ tend vers l'infini.
Ainsi $\nu_m\to \overline z$.
\end{proof}

Finalement, remarquons que
puisque pour tous $f,g \in \TA$ l'action de $f$ induit une isométrie de $\Ap_g$ vers $\Ap_{fg}$,
l'action de $\TA$ sur $\Xl$ est par isométries.
De plus, par le corollaire~\ref{cor:connexe}, $\Xl$ est connexe.

\subsection{Un lemme angulaire}

\begin{lemma} \label{lem:angle}
Soit $\alpha=(m,p,1)\in\W$ et $0\leq k\leq m$.
Dans $\PW$ considérons les demi-droites (voir figure~\ref{fig:angles}) :
\begin{itemize}
\item $c_1$ de $[\alpha]$ vers $[0,0,1]$;
\item $c_2$ de $[\alpha]$ vers $[k,0,1]$;
\item $c_3$ de $[\alpha]$ vers $[m-k,0,1]$;
\item $c_4$ de $[\alpha]$ vers $[m,0,1]$.
\end{itemize}
Les $c_i$ deviennent des courbes lisses pour la distance $|\cdot,\cdot|$ de $\R^2$.
Pour $1 \le i < j \le 4$, notons $\theta_{ij}$ l'angle au point $[\alpha]$ entre les courbes $c_i$ et $c_j$, pour la distance $|\cdot,\cdot|$.
Alors $\theta_{12} = \theta_{34}$, ou de façon équivalente $\theta_{12} + \theta_{13} = \pi/3$.
\end{lemma}

\begin{figure}[ht]
\begin{tikzpicture}[scale = 12,font=\small,rotate = -60]
\coordinate (010) at (0,0);
\coordinate (100) at (1,0);
\coordinate [label=left:{$\trip{0}{0}{1}$}] (001) at (.5, -.86);
\coordinate (310) at ($ (100)!1/4!(010) $) {};
\coordinate (210) at ($ (100)!1/3!(010) $) {};
\coordinate (110) at ($ (100)!1/2!(010) $) {};
\coordinate [label=right:{$\trip{m}{0}{1}$}] (601) at ($ (100)!1/7!(001) $) {};
\coordinate [label=below:{$\trip{m-k}{0}{1}$}] (401) at ($ (100)!1/5!(001) $) {};
\coordinate [label={[xshift=-10,below]:{$\trip{k}{0}{1}$}}] (201) at ($ (100)!1/3!(001) $) {};
\coordinate [label=above:{$[\alpha] = \trip{m}{p}{1}$}] (621) at (intersection of 310--001 and 010--601);
\draw (601) to (001);
\draw (621) to node[auto,swap, pos=.85] {$c_1$} (001);
\draw (621) to node[auto,swap, pos=.85] {$c_2$} (201);
\draw (621) to node[auto, swap, pos=.68] {$c_3$} (401);
\draw (621) to node[auto,pos=.85] {$c_4$} (601);
\draw pic["$\theta_{12}$",draw=black,angle eccentricity=1.3,angle radius=1.2cm] {angle=001--621--201};
\draw pic["$\theta_{34}$",draw=black,angle eccentricity=1.2,angle radius=1.5cm] {angle=401--621--601};
\end{tikzpicture}
\caption{Lemme angulaire~\ref{lem:angle}. Le dessin est dans le simplexe $\PW$, mais noter que les angles doivent être considérés relativement à la distance $|\cdot,\cdot|$.} \label{fig:angles}
\end{figure}
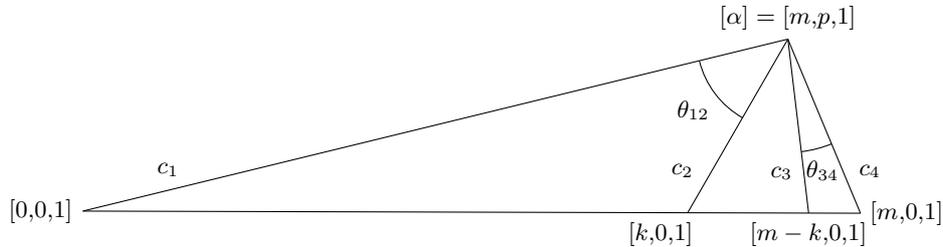

\begin{proof}
Notons que $[0,1,0]$ est colinéaire avec $[\alpha]$ et $[m,0,1]$, donc $c_1$ et $c_4$ sont principales et $\theta_{14}=\frac{\pi}{3} $ par la remarque~\ref{rem:droites}.
Le fait que les deux conclusions sont équivalentes vient alors de l'égalité
\[
\frac{\pi}{3} = \theta_{14}
=\theta_{12} + \theta_{13} + (\theta_{34} - \theta_{12}).
\]

Considérons l'involution de $\PW$ et son bord à l'infini donnée par :
\[
\tau\colon [\alpha_1, \alpha_2, \alpha_3] \mapsto [\alpha_1, p \alpha_3, \alpha_2 / p].
\]
Cette involution fixe la droite des poids $[t, p, 1]$, $t \ge 0$, qui contient $[\alpha]$.
De plus, $\tau [0,0,1] = [0,1,0]$ qui est colinéaire avec $[\alpha]$ et $[m,0,1]$, et
$\tau [k,0,1] = [k,p,0]$ qui est colinéaire avec $[\alpha]$ et $[m-k,0,1]$.
En particulier, $\tau$ échange les droites contenant $c_1$ et $c_4$, et également les droites contenant $c_2$ et $c_3$, ainsi $\tau$ envoie les deux courbes formant l'angle $\theta_{12}$ sur celles formant l'angle $\theta_{34}$.

Au niveau des $\beta_i = \log \alpha_i$, l'involution $\tau$ devient :
\[
(\beta_1, \beta_2, \beta_3) \mapsto (\beta_1, \beta_3 +
\log p, \beta_2 - \log p).
\]
Ainsi pour la distance $|\cdot,\cdot|$ l'involution $\tau$ est une symétrie axiale (d'axe $\beta_2 - \beta_3 = \log p$).
En particulier elle préserve les angles non orientés et on conclut $\theta_{12} = \theta_{34}$.
\end{proof}

\section{Simple connexité}
\label{sec:simple_connexite}
Dans cette section nous montrons que $\Xl_2$ est un arbre et que $\Xl_3$ est simplement connexe.
Dans ce but nous rappelons le formalisme des complexes de groupes simples, en suivant \cite[II.12]{BH}.

\subsection{Complexes de groupes}

Un \emph{complexe de groupes} $\Gl$ sur un poset $(\Sigma,<)$ est une collection de groupes $\{G_\sigma\}_{\sigma\in \Sigma}$, avec pour chaque couple $\sigma> \tau$ un homomorphisme injectif $\phi_{\tau\sigma}\colon G_\sigma\to G_\tau$.
On demande de plus que chaque triplet $\sigma> \tau> \rho$ donne lieu à une relation de compatibilité $\phi_{\rho\tau}\circ\phi_{\tau\sigma}=\phi_{\rho\sigma}$.
Un \emph{sous-complexe} de $\Gl$ est la donnée d'un sous-poset $\Sigma'$ de $\Sigma$, muni, pour tous $\sigma, \tau\in \Sigma'$, des mêmes $G_\sigma, \phi_{\tau\sigma}$ que $\Gl$.

Un \emph{morphisme} $\psi\colon\Gl\to G$ d'un complexe de groupes vers un groupe $G$ est une collection d'homomorphismes $\psi_\sigma\colon G_\sigma\to G$ satisfaisant $\psi_\sigma=\psi_\tau\circ \phi_{\tau\sigma}$ pour tout couple $\sigma> \tau$.
Dans la situation typique où $\Sigma$ est le poset des cellules d'un complexe polyédral $\Dl$, à un tel morphisme nous associons un \emph{développement} de $\Dl$ qui est la réunion disjointe $\Dl \times G$ de copies de $\Dl$ indicées par $G$ quotientée par la relation $(x,h)\sim (x,hg)$, où $x$ appartient à une cellule $\sigma$ et $g$ est contenu dans l'image de $\psi_\sigma$.
Cette notion de développement peut être vue comme l'inverse d'un passage au quotient :

\begin{theorem}[{\cite[II.12.20(1)]{BH}}]
\label{thm:quotientcomplex}
Soit $\Gl$ un complexe de groupes sur le poset des cellules d'un complexe polyédral $\Dl$.
Supposons que $\Dl$ est un domaine fondamental de l'action de $G$ sur un complexe polyédral $\Dl'$, que chaque $G_\sigma$ est le stabilisateur de $\sigma$, et que les $\phi_{\tau\sigma}\colon G_\sigma \to G_\tau, \psi_\sigma \colon G_\sigma\to G$ sont les inclusions évidentes.
Alors $\Dl'$ est isomorphe comme complexe polyédral, de manière $G$-équivariante,
au développement associé à $\psi$.
\end{theorem}

Dans la situation du théorème~\ref{thm:quotientcomplex} nous disons que $\Gl$ est \emph{développable}.

Supposons $\Dl$ simplement connexe. Alors le \emph{groupe fondamental} $F\Gl$ de $\Gl$ est le produit libre de tous les $G_\sigma$ quotienté par les relations $g\sim \phi_{\tau\sigma}(g)$.
Si $\Gl$ est développable, on définit son \emph{développement universel} comme le développement associé au morphisme naturel $\psi\colon \mathcal G \to F\Gl$.
Chaque morphisme $\psi\colon \mathcal G \to G$ induit une application $F\psi \colon F\Gl\to G$.
Si $F\psi$ est un isomorphisme, alors le développement associé à $\psi$ est isomorphe au développement universel.

\begin{theorem}[{\cite[II.12.20(4)]{BH}}]
\label{thm:dev}
Le développement universel est simplement connexe.
\end{theorem}


\subsection{\texorpdfstring{$\Xl_2$}{X2} est un arbre}
\label{sec:arbre}

Quand $n=2$, $\PW^+$ muni de la distance euclidienne $|\cdot,\cdot|$ est une demi-droite d'extrémité égale à $[1,1]$, et chaque hyperplan admissible dans $\PW^+$ est un point $[i,1]$ pour $i \ge 1$.
Le représentant de $[i,1]$ dans l'hyperbole $\alpha_1\alpha_2=1$ est $(\alpha_1, \alpha_2) = \big(\sqrt i, \frac{1}{\sqrt i}\big)$, et donc $\beta_1= \frac{\log i}{2},\beta_2=-\frac{\log i}{2}$.
On obtient que $[i+1,1]$ est à distance $\frac{\log(i+1)-\log i}{\sqrt{2}}$ de $[i,1]$.
Nous considérons alors $\PW^+$ comme un graphe métrique, avec un sommet $s_i=[i,1]$ pour chaque $i\geq 1$, et une arête $e_i$ de longueur $\frac{\log(i+1)-\log i}{\sqrt{2}}$ entre chaque $s_i$ et $s_{i+1}$.
Pour chaque cellule ouverte $\sigma$ de $\PW^+$, notons $G_\sigma\subset \Td$ le stabilisateur de $\nu_{\id,[\alpha]}$ pour $[\alpha]\in\sigma$: par la proposition~\ref{pro:lieu fixe}, cette définition est bien indépendante du choix de $[\alpha]$.

D'après la proposition~\ref{pro:stabilisateur}, $G_{s_1}$ est le groupe affine~$A_2$ et $G_{e_1}$ est le sous-groupe affine triangulaire dans~$A_2$.
En outre, pour $i\geq 2$, $G_{s_i}=G_{e_i}$ est le groupe des automorphismes de la forme $(ax_1 + P(x_2),bx_2 + c)$ avec $\deg P \le i$.
En particulier, pour tout $i \ge 2$ nous avons $G_{s_i} \subset G_{s_{i+1}}$.
Soit $\mathcal G$ le graphe de groupes sur le poset des cellules de $\PW^+$ muni de ces groupes.
La limite inductive $\lim_{i \ge 2} G_{s_i}$ est égale au groupe $E_2$ des automorphismes triangulaires.
Ainsi le groupe fondamental $F\mathcal G$ est le produit amalgamé de $A_2$ et $E_2$ le long du groupe affine triangulaire $A_2\cap E_2$.

D'après le corollaire~\ref{cor:intrin}\ref{cor:intrin2}, $\Ap_\id^+$ est un domaine fondamental pour l'action de $\Td$ sur $\Xl_2$. En identifiant $\Ap_\id^+$ avec $\PW^+$, d'après le théorème~\ref{thm:quotientcomplex}, le développement associé au morphisme naturel $\psi \colon \mathcal G\to\Td$ est isomorphe à $\Xl_2$.
Considérons le morphisme $F\psi\colon F\mathcal G\to \Td$ induit par $\psi$.
D'après le théorème de Jung--van der Kulk on a $\Td= A_2 \ast_{A_2\cap E_2}E_2$, ainsi $F\psi$ est un isomorphisme, et donc $\Xl_2$ est isomorphe au développement universel de~$\mathcal G$.
Celui-ci étant (un graphe)
simplement connexe d'après le théorème~\ref{thm:dev}, on conclut comme attendu que $\Xl_2$ est un arbre métrique.

\subsection{\texorpdfstring{$\Xl_3$}{X3} est
simplement connexe}

Nous utilisons maintenant les mêmes idées que dans le paragraphe précédent pour montrer :

\begin{proposition} \label{pro:X3 1-connexe}
Sur un corps $\K$ de caractéristique nulle, $\Xl_3$ est
simplement connexe.
\end{proposition}

Au lieu du théorème de Jung--van der Kulk nous utilisons les deux résultats suivants concernant la structure du groupe $\Tt$.
Notons $A=A_3$ le groupe affine,
\begin{align*}
B&=\{(ax_1 +P(x_2,x_3), bx_2+cx_3+d, b'x_2+c'x_3+d'); P\in \K[x_2,x_3], a\neq 0, bc'-b'c\neq 0\},\\
C&= \{(f_1,f_2,f_3) \in \Tt; f_3 = cx_3+d, c\neq 0\},\\
H_1&= \{(ax_1 + P(x_2,x_3), bx_2 + R(x_3), cx_3 + d); P\in \K[x_2,x_3], R\in \K[x_3]\ a,b,c\neq 0\},\\
K_2&=\{(ax_1 +bx_2+P(x_3), a'x_1+b'x_2+R(x_3), cx_3+d); P,R\in \K[x_3], ab'-a'b\neq 0,c\neq 0\}.
\end{align*}

\begin{theorem}[{\cite[Theorem 2]{Wright}, \cite[Corollary 5.8]{Lamy}}]
\label{thm:triangle}
Sur un corps $\K$ de caractéristique nulle, le groupe $\Tt$ est le produit amalgamé des groupes $A, B,C$ le long de leurs intersections respectives.
Autrement dit, $\Tt$ est le groupe fondamental du triangle de groupes où les groupes de sommets sont $A,B,C$ et les autres groupes sont leurs intersections adéquates.
\end{theorem}

\begin{proposition}[{\cite[Proposition~3.5]{LP}}] \label{pro:C amalgame}
Le groupe $C$ est le produit amalgamé de $H_1$ et $K_2$ le long de leur intersection.
\end{proposition}

\begin{proof}[Preuve de la proposition~\ref{pro:X3 1-connexe}]
Les droites admissibles munissent $\PW^+$ d'une structure de complexe polygonal (nous oublions provisoirement la distance pour ne garder que l'objet combinatoire).
Précisément, les sommets sont les points d'intersection entre droites admissibles différentes, les arêtes ouvertes sont les composantes connexes du complément des sommets dans les droites admissibles, et les cellules ouvertes de dimension $2$ sont les composantes connexes dans $\PW^+$ du complément des sommets et des arêtes (i.e.\ des droites admissibles).
Pour chaque cellule ouverte $\sigma$ de $\PW^+$, soit $G_\sigma\subset \Tt$ le stabilisateur de $\nu_{\id,[\alpha]}$ pour un poids $[\alpha]\in\sigma$.
Soit $\mathcal G$ le complexe de groupes sur le poset des cellules de $\PW^+$ muni de ces groupes.

Pour obtenir la simple connexité de $\Xl_3$, selon le même argument que dans le cas $n=2$, il suffit de démontrer que $F\mathcal G=\Tt$.
Bien que le théorème~\ref{thm:dev} concerne la simple connexité de la réalisation géométrique du poset de $\Xl$, la distance $d_\Xl$ lui est localement bilipschitz et cela entraînera donc bien la simple connexité de $(\Xl_3,d_\Xl)$.

Afin d'analyser le groupe fondamental $F\mathcal G$, nous considérons la partition suivante de $\mathcal G$ en sous-complexes (figure~\ref{fig:decoupage}, à gauche), où par définition les cellules sont munies des mêmes groupes que dans $\Gl$.
Soit $\mathcal A$ le sommet $[1,1,1]$, $\mathcal B$ le sous-complexe sur le poset des cellules dans la demi-droite $\frac{1}{2}\alpha_1\geq \alpha_2=\alpha_3$, et $\mathcal C$ le sous-complexe sur le poset des cellules dans la région $\alpha_2\geq 2\alpha_3$.
Soit $\mathcal{AB}$ l'arête entre $[1,1,1]$ et $[2,1,1]$, $\mathcal{AC}$ l'arête entre $[1,1,1]$ et $[2,2,1]$, et $\mathcal{ABC}$ le triangle $([1,1,1],[2,1,1],[2,2,1])$. Finalement, soit $\mathcal{BC}$ le sous-complexe sur le poset des cellules dans la région $\alpha_1\geq 2\alpha_3>\alpha_2>\alpha_3$.

De plus, nous partitionnons encore le sous-complexe $\mathcal C$ de la manière suivante.
Soit $\mathcal K$ le sous-complexe sur le poset des cellules dans la demi-droite $\alpha_1=\alpha_2\geq 2\alpha_3$,
$\mathcal H$ le sous-complexe sur le poset des cellules dans la région définie par $\alpha_2\geq 2\alpha_3$ et $\alpha_1\geq \alpha_2+\alpha_3$,
et $\mathcal HK$ le sous-complexe sur le poset des cellules dans la région $R$ déterminée par $\alpha_2+\alpha_3>\alpha_1>\alpha_2\geq 2\alpha_3$.
Observons qu'ici comme plus haut certaines faces des cellules de ces posets ne leurs appartiennent pas.
Décrivons plus précisément le poset $\Sigma$ des cellules de $\Hl\Kl$.
Les seules droites admissibles intersectant $R$ sont d'équation $\alpha_1=m\alpha_3$ ou $\alpha_2=m\alpha_3$, pour $m>2$, et elles ne s'intersectent pas dans $R$.
Alors chaque $2$-cellule de $\Sigma$, sauf une, a exactement deux arêtes dans $\Sigma$ (voir figure~\ref{fig:decoupage}, à droite).
La structure des posets de $\mathcal {BC}$ et de $\mathcal {H}$ est plus compliquée.

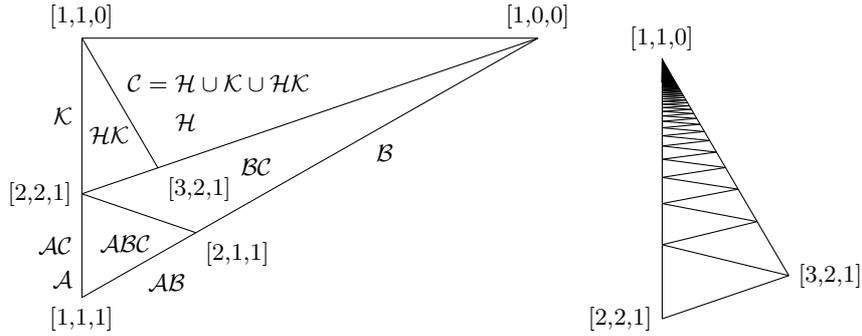
\begin{figure}[ht]
\[
\begin{tikzpicture}[scale=12, font=\small]
\coordinate (100) at (1,0);
\coordinate [label=below:$\trip{1}{1}{1}$, label=above left:$\Al$] (111) at (.5, -.86/3);
\coordinate [label=above:$\trip{1}{1}{0}$](110) at (.5,0) {};
\coordinate [label=left:$\trip{2}{2}{1}$] (221) at ($ (110)!.6!(111) $);
\coordinate [label=below right:$\trip{2}{1}{1}$] (211) at ($ (111)!.25!(100) $);
\coordinate [label=below right:$\trip{3}{2}{1}$] (321) at ($ (221)!1/6!(100) $);
\draw (110) to (100);
\draw (111) to node[left]{$\Al\Cl$} (221);
\draw (111) to node[below right]{$\Al\Bl$} (211);
\draw (211) to node[below right]{$\Bl$} node[above left, near start]{$\Bl\Cl$} (100);
\draw (221) to node[left]{$\Kl$} (110);
\draw (221) to node[below, xshift= -5, yshift = -5]{$\Al\Bl\Cl$} (211);
\draw (221) to (100);
\draw (321) to node[left, near start]{$\Hl\Kl$} node[right, near start, xshift=10, yshift=5]{$\Hl$} (110);
\node at (100) [above] {$\trip{1}{0}{0}$};
\node at (.65,-.05) {$\Cl = \Hl \cup \Kl \cup \Hl\Kl$};
\end{tikzpicture}
\begin{tikzpicture}[scale = 20,font=\small]
\coordinate (010) at (0,0);
\coordinate (100) at (1,0);
\coordinate (001) at (.5, -.86);
\coordinate [label=above:$\trip{1}{1}{0}$](110) at (\tbary{1}{1}{0});
\coordinate [label=left:$\trip{2}{2}{1}$] (a) at (\tbary{2}{2}{1});
\coordinate [label=right:$\trip{3}{2}{1}$] (b) at (\tbary{3}{2}{1});
\foreach \x [evaluate=\x as \y using int(\x+1)] in {2,...,29}
{
  \coordinate (\x\x1) at (\tbary{\x}{\x}{1});
  \coordinate  (\y\x1) at (\tbary{\y}{\x}{1});
  \draw (\x\x1)--(\y\x1);
}
\foreach \x [evaluate=\x as \y using int(\x-1)] in {3,...,29}
{
  \draw (\x\x1)--(\x\y1);
}
\draw (221)--(110)--(321);
\draw[fill=black] (28281)--(29281)--(110)--(28281);
\end{tikzpicture}
\]
\caption{Découpage de $\Gl$ en sous-complexes, et zoom sur les cellules dans $\Hl\Kl$.}
\label{fig:decoupage}
\end{figure}

Pour démontrer $F\mathcal G=\Tt$, grâce au théorème~\ref{thm:triangle} il suffit d'obtenir :
\begin{enumerate}
\item \label{tr1} $F\mathcal A=A$
\item \label{tr2} $F\mathcal {AB}=A\cap B$
\item \label{tr3} $F\mathcal {AC}=A\cap C$
\item \label{tr4} $F\mathcal {ABC}=A\cap B\cap C$
\item \label{tr5} $F\mathcal B=B$
\item \label{tr6} $F\mathcal {BC}=B\cap C$
\item \label{tr7} $F\mathcal C=C$
\item \label{tr8} Les groupes de $\mathcal {BC}$ forment un système inductif de groupes;
de plus pour chaque élément $g\in G_\sigma$ dans $\mathcal {BC}$ il existe un élément $g'\in G_{\sigma'}$ (resp. $g''\in G_{\sigma''}$) dans $\mathcal {BC}$ équivalent à $g$ dans la limite $F\mathcal {BC}$, tel que $\sigma'$ (resp. $\sigma''$) a une face dans le poset de $\Bl$ (resp.\ $\Cl$).
\end{enumerate}

La propriété~\ref{tr8} est nécessaire pour s'assurer que chaque élément de $B\cap C$ soit identifié à sa copie dans $B$ et $C$.
Les propriétés analogues pour les autres inclusions sont immédiates.

Les propriétés~\ref{tr1}-\ref{tr4} découlent immédiatement de la proposition~\ref{pro:stabilisateur}, car dans chacun de ces cas le groupe fondamental est le stabilisateur d'un seul point.
Plus précisément,~\ref{tr1} correspond à l'exemple~\ref{exple:stab}\ref{exple:stab1}, et~\ref{tr2} à l'exemple~\ref{exple:stab}\ref{exple:stab3}.
De plus, par la proposition~\ref{pro:C amalgame} $A\cap C=A\cap K_2$ donc $A\cap B\cap C=A\cap B\cap K_2$ ce qui permet de déduire~\ref{tr3} et~\ref{tr4} des exemples~\ref{exple:stab}\ref{exple:stab4} et~\ref{exple:stab}\ref{exple:stab2}.

Pour~\ref{tr5}, notons que, également d'après la proposition~\ref{pro:stabilisateur} et l'exemple~\ref{exple:stab}\ref{exple:stab4}, les groupes de $\mathcal B$ forment une suite croissante dont la réunion est $B$.

Pour~\ref{tr6}, considérons une courbe $\alpha\colon (t_0, \infty) \to \W^+$ satisfaisant les hypothèses du corollaire~\ref{cor:stab}, telle que sa projection $[\alpha(t)]\subset \PW^+$ soit contenue dans les cellules de $\mathcal {BC}$, et qui tende vers le point $[1,0,0]$ en étant asymptote à la droite $\alpha_2=2\alpha_3$.
Par exemple la courbe $\alpha\colon (3,\infty)\to \W^+$ avec $\alpha(t)=(t^2,2t,t+1)$ convient (voir figure~\ref{fig:courbes}).
Pour $\alpha\in \W^+$, soit $\sigma(\alpha)$ la cellule de $\PW^+$ contenant $[\alpha]$ dans son intérieur.
D'après le corollaire~\ref{cor:stab}, les $G_{\sigma(\alpha(t))}$ forment une suite croissante de groupes.
Par la proposition~\ref{pro:stabilisateur}, la réunion de cette suite est égale à $B\cap C$.
Pour tout point dans une cellule de $\mathcal {BC}$, il existe un segment de droite reliant ce point à un point de la courbe $[\alpha(t)]$ auquel on peut appliquer le corollaire~\ref{cor:stab}.
On obtient ainsi que chaque groupe de $\mathcal {BC}$ est contenu dans un $G_{\sigma(\alpha(t))}$.
En conséquence les groupes de $\mathcal {BC}$ forment un système inductif de groupes (ce qui donne aussi la première partie de~\ref{tr8}) dont la limite est $B\cap C$.

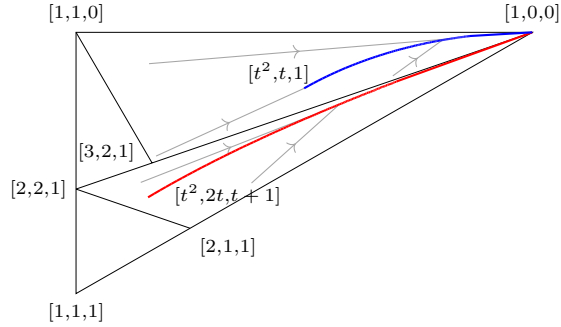
\begin{figure}[ht]
\[
\begin{tikzpicture}[scale=2,font=\scriptsize]
\coordinate [label=below:$\trip{1}{1}{1}$] (111) at (-90:{sqrt(3)});
\coordinate [label=above:$\trip{1}{1}{0}$] (110) at (0:0);
\coordinate [label=above:$\trip{1}{0}{0}$] (100) at (0:3);
\coordinate [label=below right:$\trip{2}{1}{1}$] (211) at (\bary{1}{0}{1});
\coordinate [label=left:$\trip{2}{2}{1}$] (221) at (\bary{1}{1}{0});
\coordinate [label={left,xshift=-3,yshift=4:$\trip{3}{2}{1}$}] (321) at (\bary{1}{1}{1});

\draw[very thin] (110)--(100)--(111)--(110)--cycle;
\draw[very thin] (221)--(100) (221)--(211) (321)--(110);
\foreach \x [evaluate=\x as \y using \x/10] in {30,...,140}
{
\coordinate (p\x) at (\bary{\y+1}{\y-1}{\y*\y-2*\y});
\coordinate (q\x) at (\bary{1}{\y-1}{\y*\y-\y+1});
}

\coordinate (ar1) at (\bary{.1}{.9}{.4});
\draw[mid arrow] (ar1)--(q105);
\coordinate (ar2) at (\bary{.9}{1}{1});
\draw[mid arrow] (ar2)--(q30);
\coordinate (ar3) at (\bary{.08}{0.1}{1});
\draw[mid arrow] (ar3)--(q95);
\coordinate (ar4) at (\bary{1.35}{1}{1});
\draw[mid arrow] (ar4)--(p80);
\coordinate (ar5) at (\bary{1}{.1}{2});
\draw[mid arrow] (ar5)--(p90);

\foreach \x [evaluate=\x as \y using \x+1] in {30,...,139}
{
\draw[red,thick] (p\x)--(p\y);
\draw[blue,thick] (q\x)--(q\y);
}

\draw[red,thick] (p140)--(100);
\draw[blue,thick] (q140)--(100);

\node at (p35) [below,xshift=19,yshift=2] {$\trip{t^2}{2t}{t+1}$};
\node at (q40) [left,xshift=-10,yshift=-2] {$\trip{t^2}{t}{1}$};
\end{tikzpicture}
\]
\caption{Courbes et segments satisfaisant les hypothèses du corollaire~\ref{cor:stab}.} \label{fig:courbes}
\end{figure}

Pour montrer~\ref{tr7} calculons $F\mathcal{C}$.
Nous utilisons la partition de $\mathcal{C}$ en $\mathcal{K},\mathcal{H}$ et $\mathcal{HK}$. D'abord, d'après la proposition~\ref{pro:stabilisateur} et l'exemple~\ref{exple:stab}\ref{exple:stab4}, les groupes de $\mathcal{K}$ forment une suite croissante dont la réunion est~$K_2$.
De même, les groupes de $\mathcal{HK}$ forment une suite croissante dont la réunion est $H_1\cap K_2$.
Enfin, dans les cellules de $\mathcal {H}$ considérons la projection $[\alpha(t)]$ de la courbe $\alpha\colon(2,\infty)\to\W^+$ avec $\alpha(t)=(t^2,t,1)$ (voir figure~\ref{fig:courbes}).
D'après le corollaire~\ref{cor:stab}, les $G_{\sigma(\alpha(t))}$ forment une suite croissante.
Par la proposition~\ref{pro:stabilisateur} la réunion de cette suite est égale à $H_1$. D'après le corollaire~\ref{cor:stab} appliqué à des segments de droites bien choisis, chaque groupe de $\mathcal {H}$ est contenu dans un $G_{\sigma(\alpha(t))}$.
Ainsi les groupes de $\mathcal {H}$ forment un système inductif dont la limite est $H_1$.
L'analogue de la propriété~\ref{tr8} avec $H_1\cap K_2$ au lieu de $B\cap C$ est immédiat, car les groupes de $\mathcal{HK}$ forment une suite croissante.
On conclut grâce à la proposition~\ref{pro:C amalgame} que $F\mathcal{C}=H_1 *_{H_1 \cap K_2} K_2=C$.

Finalement, pour la deuxième assertion de~\ref{tr8}, considérons $g\in G_\sigma$ dans $\mathcal {BC}$.
Par le lemme~\ref{lem:finite}, $\sigma$ est contenu dans un nombre fini de demi-espaces admissibles $L_1,\ldots, L_k$. De plus, puisque $\sigma$ est dans le poset de $\mathcal{BC}$, les $L_i$ sont définies par des inégalités de la forme $\alpha_1\geq m_2\alpha_2+m_3\alpha_3$ ou $\alpha_2\geq\alpha_3$.
Alors chaque cellule $\sigma'$ (resp.\ $\sigma''$) ayant une face dans la demi-droite $\alpha_2=\alpha_3$ (resp.\ la demi-droite $\alpha_2=2\alpha_3$), et suffisamment loin de $[1,1,1]$, est contenue dans tous les $L_i$.
Les groupes de $\mathcal {BC}$ forment un système inductif de groupes, donc il existe un $G_\rho$ dans $\mathcal {BC}$ contenant $G_\sigma$ et $G_\sigma'$ simultanément.
Par le corollaire~\ref{cor:stabjump}, $G_\sigma\subset G_\sigma'$ comme sous-groupes de $\Tt$, et donc aussi pour leurs copies dans $G_\rho$.
Ainsi il existe $g'\in G_{\sigma'}$ équivalent à $g$, comme attendu.
De même, il existe $g''\in G_{\sigma''}$ équivalent à $g$.
\end{proof}

\section{Courbure négative ou nulle}

Toute cette section est en dimension $n = 3$, sur un corps $\K$ arbitraire.
L'objectif est de montrer la

\begin{proposition} \label{pro:courbure negative}
Soit $\alpha \in \W$, et $g \in \Tt$.
Alors $\nu_{g,[\alpha]}$ admet un voisinage $\CAT(0)$ dans $\Xl_3$.
\end{proposition}

Nous commençons par établir un critère abstrait.

\subsection{Limite d'espaces \texorpdfstring{$\CAT(0)$}{CAT(0)}} \label{sec:limit}

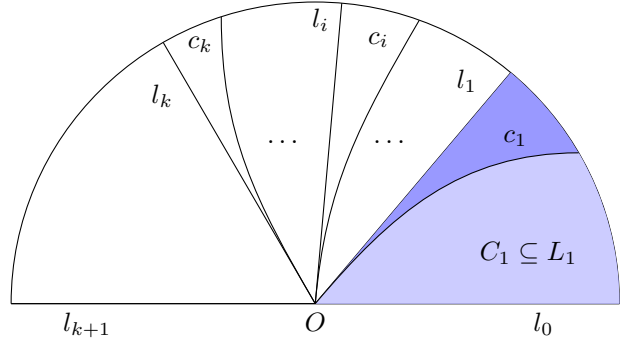
\begin{figure}[ht]
\[
\begin{tikzpicture}[scale=4]
\draw (1,0) arc (0:180:1);
\draw (0,0) to node[below, near end]{$l_{k+1}$} (-1,0);
\draw (0,0) to node[below, pos=0]{$O$} node[below, near end]{$l_0$} (1,0);
\draw (-1,0) to (1,0);
\draw (0,0) to node[auto, very near end]{$l_k$} ++(120:1);
\draw (0,0) to node[auto, very near end]{$l_i$} ++(85:1);
\draw (0,0) to node[auto, very near end]{$l_1$} ++(50:1);
\fill[blue!40] (0,0) to ++(50:1) arc (50:0:1) to cycle;
\fill[blue!20] (0,0) to[out=50,in=180] ++(30:1) arc(30:0:1) to cycle;
\draw (0,0) to[out=120,in=270] node[auto, pos=.92]{$c_k$} ++(108:1);
\draw (0,0) to[out=85,in=240] node[auto, very near end]{$c_i$} ++(70:1);
\draw (0,0) to[out=50,in=180] node[auto, very near end]{$c_1$} ++(30:1);
\coordinate [label=$C_1 \subseteq L_1$] (C1) at (0.7,0.1);
\coordinate [label=$\dots$] (dots) at (-0.1,0.5);
\coordinate [label=$\dots$] (dots) at (0.25,0.5);
\end{tikzpicture}
\]
\caption{Notations pour la section~\ref{sec:limit}.} \label{fig:limit}
\end{figure}

Soit $r>0$ et $\DD = \{(x,y) \in \R^2 \mid x^2 + y^2 \le r^2, y \ge 0\}$ le demi-disque fermé supérieur de rayon $r$ dans $(\R^2,|\cdot,\cdot|)$.
On note $O$ le centre de $\DD$.
Soit $l_0,l_1,\ldots,l_k, l_{k+1}$ une suite de rayons deux à deux distincts et ordonnés dans le sens direct, avec $l_0$ d'extrémité $(r,0)$ et $l_{k+1}$ d'extrémité $(-r,0)$.
Pour $j=1,\ldots,k$, soit $c_j$ l'image d'une courbe lisse plongée $[0,1]\to \DD$.
On suppose que $c_j$ est distincte du rayon $l_{j-1}$, et est contenue dans le secteur fermé délimité par les rayons $l_{j-1}$ et $l_{j}$, avec pour extrémités le centre $O$ du demi-disque et un point dans l'arc de cercle du secteur.
Soit $C_j$ l'adhérence de la composante de $\DD \setminus c_j$ contenant~$l_{0}$, et supposons chacun des $C_j$ convexe.
Soit $L_j$ le secteur fermé délimité par les rayons $l_j$ et~$l_{0}$: on a donc $L_{j-1} \subsetneq C_j \subseteq L_j$.
Par convention, $C_0=L_0=c_0=l_0, c_{k+1}=l_{k+1},C_{k+1}=L_{k+1}=\DD$, et $C_\infty=L_\infty=c_\infty=l_\infty$ est le diamètre $l_0\cup l_{k+1}$ de  $\DD$ (voir figure~\ref{fig:limit}).

Soit $(\DD_\lambda, \phi_\lambda)_{\lambda \in \Lambda}$ une famille de demi-disques indicée par un ensemble quelconque $\Lambda$, où pour chaque $\lambda$, $\phi_\lambda\colon \DD \to \DD_\lambda$ est une isométrie.
Supposons qu'il existe une relation d'équivalence $\sim$ sur l'union disjointe $\bigsqcup_{\lambda\in \Lambda} \DD_\lambda$, qui en restriction à chaque $\DD_\lambda\sqcup \DD_{\lambda'}$ (où $\lambda\neq \lambda'$) est donnée par $\phi_\lambda(x)\sim \phi_{\lambda'}(x)$ si et seulement si $x\in C_{j(\lambda,\lambda')}$, pour une fonction $j$ qui à chaque paire non ordonnée $\lambda, \lambda'$ d'éléments distincts de $\Lambda$ associe un élément de $\mathcal I=\{0,1,\ldots ,k, k+1,\infty\}$.

\begin{remark} \label{lem:espace linearise}
Nous affirmons que la relation $\sim_0$ sur $\bigsqcup_{\lambda\in \Lambda} \DD_\lambda$ qui en restriction à chaque $\DD_\lambda\sqcup \DD_{\lambda'}$ est donnée par $\phi_\lambda(x)\sim_0 \phi_{\lambda'}(x)$ si et seulement si $x\in L_{j(\lambda,\lambda')}$, pour la même fonction $j$ que précédemment, est une relation d'équivalence.

Pour voir cela, remarquons qu'il existe une application symétrique $\cap\colon \mathcal I\times \mathcal I\to \mathcal I$ telle que pour chaque $i,j\in \mathcal I$ on a $C_i\cap C_j=C_{i\cap j}$.
Précisément, si $i,j<\infty$ alors $i\cap j=\min \{i,j\}$, si $i<\infty$ alors $i\cap \infty=0$, et finalement $\infty\cap \infty =\infty$.
De plus, $L_i\cap L_j=L_{i\cap j}$ pour cette même application $\cap$.

La transitivité de la relation $\sim$ revient à dire que pour tout triplet d'indices distincts $\lambda, \lambda', \lambda''$ on a
\[
C_{j(\lambda,\lambda')\cap j(\lambda',\lambda'')}=C_{j(\lambda,\lambda')}\cap C_{j(\lambda',\lambda'')}\subset C_{j(\lambda,\lambda'')},
\]
donc
\[
\big(j(\lambda,\lambda')\cap j(\lambda',\lambda'')\big)\cap j(\lambda,\lambda'')=j(\lambda,\lambda')\cap j(\lambda',\lambda'').
\]
Ceci donne une égalité similaire pour les $L_i$ :
\[
L_{j(\lambda,\lambda')}\cap L_{j(\lambda',\lambda'')}=L_{j(\lambda,\lambda')\cap j(\lambda',\lambda'')}\subset L_{j(\lambda,\lambda'')},
\]
ce qui montre que $\sim_0$ est transitive, et donc est une relation d'équivalence.
\end{remark}

Notons $\V_0$, $\V$ les espaces quotients correspondant à $\sim_0$ et $\sim$ avec leurs pseudo-distances respectives $d_0$, $d$, définies par des longueurs de chaînes dans $\bigsqcup_{\lambda\in \Lambda} \DD_\lambda$ comme dans la définition~\ref{def:chaîne}.
De la même manière que dans la preuve de la proposition~\ref{pro:metrique}, on démontre que ces pseudo-distances sont des distances, et les espaces quotients sont des espaces de longueur. Remarquons que dans chaque cas les points $\phi_\lambda(O)$ sont identifiés à un point que nous appellerons le \emph{centre} de $\V_0$ ou~$\V$, et que nous noterons encore $O$. Par exemple, si $\Lambda=\{\lambda,\lambda'\}$ et $j(\lambda,\lambda')=\infty$, alors $\V_0=\V$ est un disque fermé.

Le résultat principal de cette section est le critère suivant.

\begin{proposition}
\label{pro:V est CAT0}
Les espaces $\V_0$ et $\V$ sont complets.
Si $\V_0$ est $\CAT(0)$, alors $\V$ est également $\CAT(0)$.
\end{proposition}

\begin{proof}
Notre stratégie est de montrer que $\V$ est une limite d'espaces $\CAT(0)$.
Pour tout entier $m\geq 1$ et $j=0,1,\ldots, k+1$ choisissons $c^m_j$ une courbe linéaire par morceaux dont
les extrémités sont les mêmes que celles de $c_j$, avec tous les sommets sur $c_j$, et dont chaque point est à distance $\leq \frac{1}{m}$ de $c_j$.
Notons ensuite $C^m_j$ l'adhérence de la composante de $\DD \setminus c^m_j$ contenant $l_0$.
Comme précédemment on convient que
$C^m_\infty=c^m_\infty=c_\infty$.
Notons $(\V_m,d_m)$ l'espace de longueur obtenu via la relation d'équivalence $\sim_m$ sur $\bigsqcup_{\lambda\in \Lambda} \DD_\lambda$ qui sur chaque $\DD_\lambda\sqcup \DD_{\lambda'}$ est définie par $\phi_\lambda(x)\sim \phi_{\lambda'}(x)$ pour $x\in C^m_{j(\lambda,\lambda')}$ (la remarque~\ref{lem:espace linearise} s'applique à $\sim_m$ de la même manière qu'à $\sim_0$).

Il existe une application bilipschitzienne de $\DD$ vers $\DD$ qui envoie chaque $l_j$ sur $c_j$ (observer que c'est ici que l'on a besoin de $c_j\neq l_{j-1}$).
De même, pour chaque $m\geq 1$ il existe une application bilipschitzienne de $\DD$ vers $\DD$ qui envoie chaque $l_j$ sur $c^m_j$.
Ainsi tous les espaces $\V,\V_0,\V_m$ sont deux à deux bilipschitziens.
L'espace $\V_0$ est complet par \cite[I.7.13]{BH}.
Précisément, il faut remplacer chaque secteur par un triangle isocèle, et donc le demi-disque par un polygone, afin de se ramener au cas d'un complexe polygonal avec un nombre fini de classes d'isométrie de pièces.
Ce remplacement est une application bilipschitzienne.
Les espaces $\V$, $\V_m$ sont alors également complets, car bilipschitziens à $\V_0$.

Montrons maintenant que chaque $\V_m$ est $\CAT(0)$.
Comme $\V_m$ est complet et simplement connexe (il se rétracte radialement sur $O$), par \cite[II.4.1(2)]{BH} il suffit de montrer que $\V_m$ est localement $\CAT(0)$.
Noter qu'il existe une application naturelle $\pr\colon \V_m\to \DD$ qui est l'inverse de chaque $\DD\to \DD_\lambda\subset \V_m$.
Soit $x\in \V_m$, et considérons différentes possibilités pour $\pr(x)$.
Si $\pr(x)$ n'appartient à aucune des courbes $c_j^m$, alors un voisinage suffisamment petit de $x$ est un disque plat.
Si $\pr(x)$ appartient à une courbe $c_j^m$ mais est distinct du centre $O$, alors un voisinage suffisamment petit de $x$ est obtenu en recollant des disques (ou demi-disques) plats le long d'un sous-ensemble convexe commun, et un tel voisinage est $\CAT(0)$ par \cite[II.11.3]{BH}. Finalement, considérons le cas $\pr(x)=O$, c'est-à-dire que $x$ est le centre de~$\V_m$.
Par \cite[I.7.16]{BH} il existe un voisinage de $x$ qui est un cône sur le link de $x$ dans $\V_m$ : ce link est un graphe métrique qu'on va noter $\Gamma_m$.
De même, $\V_0$ est un cône sur un graphe métrique $\Gamma_0$, et il existe une application naturelle $g_m\colon \Gamma_m\to \Gamma_0$ contractant les distances qui est une équivalence d'homotopie.
Comme $\V_0$ est supposé être $\CAT(0)$, par \cite[II.3.14]{BH} le graphe $\Gamma_0$ ne contient aucun cycle de longueur $<2\pi$.
Par conséquent, le graphe $\Gamma_m$ n'en contient pas non plus, et par \cite[II.3.14]{BH}, le voisinage de $x$ dans $\V_m$ est $\CAT(0)$.

Considérons maintenant le sous-ensemble $R_m \subset \V_m\times \V$ constitué des couples $([x]_{\sim_m},[x]_\sim)$ pour $x\in \bigsqcup_{\lambda\in \Lambda} \DD_\lambda$.
Au vu du lemme~\ref{lem:convergence} ci-dessous, $R_m$ est une $M/m$-relation surjective au sens de \cite[I.5.33]{BH}, pour une certaine constante $M >0$.
On en déduit que la suite $\V_m$ converge au sens de Gromov--Hausdorff vers $\V$, et donc $\V$ est $\CAT(0)$ par \cite[II.3.10]{BH}.
\end{proof}

\begin{lemma} \label{lem:convergence}
Il existe une constante $M >0$ tel que pour tout $m\ge 1$, et tous $x,y \in \bigsqcup_{\lambda\in \Lambda} \DD_\lambda$, on a
\[
\lvert d_m([x]_{\sim_m},[y]_{\sim_m})-d([x]_\sim,[y]_\sim) \rvert \leq \frac{M}{m}.
\]
\end{lemma}

\begin{proof}
Soit $l=d([x]_\sim,[y]_\sim)$.
Remarquons d'abord que comme $C_j^m\subset C_j$, on a
\[d_m([x]_{\sim_m},[y]_{\sim_m})\geq l.\]
Par définition de la (pseudo-)distance $d$, il existe une chaîne : un entier $N$ et pour chaque $i=0,1,\ldots, N-1$ un indice $\lambda_i\in \Lambda$ et des points $x'_i,x_{i+1}\in \DD_{\lambda_i}$ avec $x_i\sim x'_i$ et $x=x'_0, y=x_N$, tel que $\sum_{i=0}^{N-1}|x'_i,x_{i+1}|< l+\frac{1}{m}$.
Observons que si $\pr(x_i)=O$ ou $\pr(x'_i)=O$ pour un certain~$i$, alors en faisant peut-être décroître $\sum |x'_i,x_{i+1}|$ on peut se ramener à une sous-chaîne avec $N=2$ et $\pr(x_1)=\pr(x'_1)=O$.
En choisissant une chaîne analogue dans $\V_m$ on obtient alors l'estimé attendu avec $M=1$.
Si au contraire on a $\pr(x_i)\neq O$ et $\pr(x'_i)\neq O$ pour tout $i$, alors pour $i=1,\dots, N-1$ nous notons $j(i)=j(\lambda_{i-1},\lambda_{i})$ et nous opérons les réductions suivantes.

En passant à une sous-chaîne nous pouvons supposer que pour tout $i=1,\ldots, N-1$ on a $\lambda_i\neq \lambda_{i-1}$ et pour tout $i=2,\ldots, N-1$ on a $x'_{i-1}\notin \phi_{\lambda_{i-1}}(C_{j(i)})$, et $x_{i}\notin \phi_{\lambda_{i-1}}(C_{j(i-1)})$.
En particulier, si $N>2$ alors $j(i)\neq k+1$. En outre, si $j(i)\neq \infty$ alors $j(i+1)=\infty$, et vice-versa, ainsi sans perte de généralité nous pouvons supposer que pour chaque $i$ pair nous avons $j(i)=\infty$.
De plus, pour tout $i=1,\ldots, N-1$ l'inégalité triangulaire nous permet de supposer que $\pr(x_i)=\pr(x'_i)$ est dans $c_{j(i)}$.

De plus, nous affirmons que nous pouvons supposer la chaîne \emph{tendue}, au sens que pour chaque $i=2,\ldots, N-2$, l'union des segments $[x'_{i-1},x_i]$ et $[x_{i}',x_{i+1}]$ dans $\big( \DD_{\lambda_{i-1}} \sqcup \DD_{\lambda_{i}}\big) / \sim$  (voir \cite[I.7.20]{BH}).
Cela revient à dire que l'angle au point $x_i \in \DD_{\lambda_{i-1}}$ entre le segment de droite $[x'_{i-1},x_i]$ et la courbe $\phi_{\lambda_{i-1}}(c_{j(i)})$  est égal à l'angle au point $x_i' \in \DD_{\lambda_{i}}$ entre $[x_{i}',x_{i+1}]$ et la courbe $\phi_{\lambda_{i}}(c_{j(i)})$.
(Ce serait naturel de demander la même propriété pour $i=1,N-1$, mais on n'en pas besoin et cela raccourcit un peu la preuve.)

L'affirmation découle des deux observations suivantes.
Tout d'abord, si la chaîne n'est pas tendue, et que la définition précédente est mise en défaut pour un certain entier $i=t$, alors comme $C_{j(t)}$ est convexe et ne contient pas $\pr(x'_{t-1}),\pr(x_{t+1})$, la somme $\sum |x'_i,x_{i+1}|$ n'est pas minimale parmi les suites avec un $N$ et un $(\lambda_i)$ donnés, puisque nous pouvons la diminuer en bougeant $x_t$ le long de $\phi_{\lambda_{t-1}}(c_{j(t)})$ (et simultanément en bougeant $x'_{t}$).
D'autre part, comme $\DD_{\lambda_i}$ et $C_{j(i)}$ sont compacts, le minimum des $\sum |x'_i,x_{i+1}|$ est atteint.
En résumé, en remplaçant la chaîne par celle minimisant $\sum |x'_i,x_{i+1}|$ parmi les chaînes avec $N$ et $(\lambda_i)$ donnés, on produit une chaîne tendue, ce qui justifie notre affirmation.

Notons $\theta_j$ l'angle dans $\DD$ entre $c_j$ et $l_{k+1}$, en particulier
$\theta_j \ge \theta_k >0$ pour tout $j =1,\dots,k$.
Comme pour les $i$ pairs nous avons $j(i)=\infty$, l'angle dans la définition d'une chaîne tendue décroît à chaque $i$ successif par au moins $\theta_k$.
Précisément, pour $i$ impair, si $c_{j(i)} = l_{j(i)}$ alors en travaillant dans le triangle euclidien de sommets $O$, $x_i$, $x_{i+1}$ on voit que l'angle décroît exactement de $\theta_{j(i)}$, et, par convexité de $C_{j(i)}$, dans le cas général la perte d'angle est encore plus grande (voir figure~\ref{fig:chaine}).
Le cas $i$ pair est similaire.
Ainsi $N-4<\frac{\pi}{\theta_k}$.
Pour $i=0,\ldots, N-1$, il existe des $z_i\in \phi_{\lambda_{i-1}}(c^m_{j(i)})$ avec $|x_i,z_i|<\frac{1}{m}$.
En utilisant cette chaîne de $z_i$ au lieu des $x_i$, on obtient $d_m([x]_{\sim_m},[y]_{\sim_m})\leq l+\frac{1}{m}+2N\frac{1}{m}$.
Alors en prenant $M\geq 1+2(\frac{\pi}{\theta_k}+4)$ on obtient l'estimé attendu.
\end{proof}

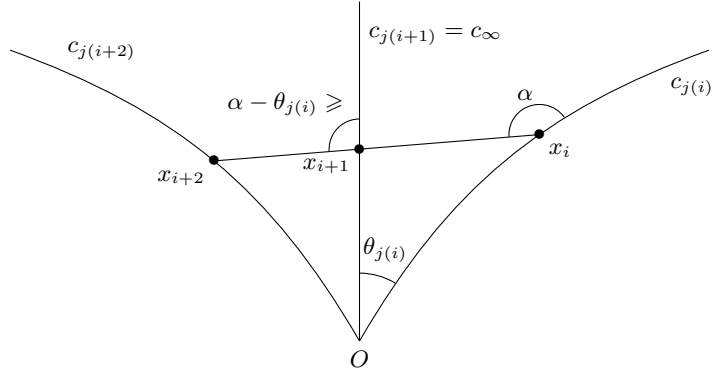
\begin{figure}[ht]
\[
\begin{tikzpicture}[scale=3,font=\small]

\coordinate (O) at (0:0);
\coordinate (A) at (40:2);
\coordinate (B) at (90:1.5);
\coordinate (C) at (140:2);

\draw (O) to[out=60,in=200] coordinate[pos=.1](theta1) coordinate[pos=.6](i-1) coordinate[pos=.61](alpha1) node[auto,swap,pos=.9]{$c_{j(i)}$} (A);
\draw (O) to node[auto,swap,pos=.9]{$c_{j(i+1)} = c_\infty$} (B);
\draw (O) to[out=120,in=-20] coordinate[pos=.1](theta2) coordinate[pos=.5](i+1) node[auto,swap,pos=.9]{$c_{j(i+2)}$} (C);
\coordinate (i) at (intersection of O--B and i-1--i+1);
\draw (i-1)--(i+1);
\draw pic["$\theta_{j(i)}$",draw=black,angle eccentricity=1.4,angle radius=.9cm] {angle=theta1--O--B};
\draw pic["$\alpha$",draw=black,angle eccentricity=1.4,angle radius=.4cm] {angle=alpha1--i-1--i+1};
\draw pic["$\alpha-\theta_{j(i)}\ge \qquad$",draw=black,angle eccentricity=2.2,angle radius=.4cm] {angle=B--i--i+1};
\node[below right] at (i-1) {$x_{i}$};
\node at (i-1) {$\bullet$};
\node[below left] at (i+1) {$x_{i+2}$};
\node at (i+1) {$\bullet$};
\node[below left] at (i) {$x_{i+1}$};
\node at (i) {$\bullet$};
\node[below] at (O) {$O$};
\end{tikzpicture}
\]
\caption{Perte d'angle dans le lemme~\ref{lem:convergence}.} \label{fig:chaine}
\end{figure}

\subsection{Preuve de la proposition}

La preuve de la proposition~\ref{pro:courbure negative} va être différente suivant la forme de $\alpha = (\alpha_1, \alpha_2, \alpha_3)$.
Par le corollaire~\ref{cor:intrin}\ref{cor:intrin2}, on peut supposer $g = \id$ et $\alpha \in \W^+$.

Notons $\Fl = \Stab(\nu_{\id,[\alpha]})$, qui nous servira d'ensemble d'indices tout au long de cette section.
Observons que $\Fl$ est l'ensemble des $f \in \Tame(\A^3)$ tel que $\Ap^+_f$ contienne $\nu_{\id,[\alpha]}$.
Soit
\[\V:=\bigsqcup_{f\in \Fl} \overline{B}(\nu_{f,[\alpha]},\tfrac{\eps}{4})/\sim\]
le voisinage du lemme~\ref{lem:boule}\ref{boule_b}.
Ainsi $\V$ est isométrique à la boule $\overline{B}_\Xl(\nu_{\id,[\alpha]}, \frac{\eps}{4})$.

Observons que si $\mult (\alpha) \le 1$, et $\alpha_1 > \alpha_2 > \alpha_3$ (resp. $\alpha_1 = \alpha_2$ ou $\alpha_2 = \alpha_3$), ce voisinage~$\V$ s'obtient en recollant des disques (resp. des demi-disques) le long d'un même convexe, donné par le lemme~\ref{lem:convex} (resp. correspondant au diamètre des demi-disques).
On conclut dans ce cas que $\V$ est $\CAT(0)$ par \cite[II.11.3]{BH}.
Nous supposons donc dans la suite $\mult (\alpha) \ge 2$.

\subsubsection{Premier cas : \texorpdfstring{$\alpha_1 = \alpha_2$}{a1 = a2}}
\label{sec:1ercas}

Dans cette situation on a $[\alpha] = [p,p,1]$ pour un entier $p \ge 1$, et il y a exactement trois droites admissibles passant par $[\alpha]$.
Ces droites sont principales, donc par la remarque~\ref{rem:droites} ce sont encore des droites pour la distance $|\cdot,\cdot|$, formant 6 secteurs chacun d'angle $\pi/3$.
Par \cite[I.7.16]{BH} le voisinage $\V$ est donc un cône sur un graphe métrique $\Gamma$ dont chaque arête est de longueur $\pi/3$.
Ici, par \emph{cône} on entend un $0$-cône au sens de \cite[I.5.6]{BH}.
En particulier, on n'a pas un plongement isométrique $\Gamma\to \V$, mais on a malgré tout une action induite naturelle de $\Stab(\nu_{\id,[\alpha]})$ sur $\Gamma$.
Pour montrer que $\V$ est $\CAT(0)$, par \cite[II.3.14]{BH} il suffit de montrer que $\Gamma$ est $\CAT(1)$, ce qui revient à dire que tout cycle plongé dans $\Gamma$ est formé d'au moins 6 arêtes.

Pour $p=1$, par le lemme~\ref{lem:GL} le graphe $\Gamma$ est l'immeuble sphérique de Bruhat--Tits de $\GL_3(\K)$, autrement dit le graphe d'incidence des points et des droites du plan projectif $\mathbb{P}^2(\K)$, donc cette propriété est immédiate.

Supposons maintenant $p\geq 2$.
Noter que pour chaque $f \in \Fl$ la boule $\overline{B}(\nu_{f,[\alpha]},\frac{\eps}{4})\subset \Ap^+_f$ est un demi-disque, ou autrement dit un cône sur un chemin de 3 arêtes $\Gamma^+_f \subset \Gamma$. L'application $\wgt$ du corollaire~\ref{cor:intrin}\ref{cor:intrin1} induit une application (pour laquelle on conserve la même notation) $\wgt\colon \Gamma \to I_3$, ou $I_3$ est un chemin de 3 arêtes --- le link de $[p,p,1]$ dans $\PW^+$.
Cette projection $\wgt$ est un isomorphisme en restriction à chaque $\Gamma^+_f$.
Notons $s$ l'extrémité de $I_3$ correspondant à la direction vers $[1,1,1]$ et $q$ l'autre l'extrémité, correspondant à la direction vers $[1,1,0]$.
Notons $s_f=\wgt^{-1}(s)\cap \Gamma^+_f$, $q_f=\wgt^{-1}(q)\cap \Gamma^+_f$ et $\Gamma_f=\Gamma^+_f\cup \Gamma^+_{f\sigma}$ pour $\sigma=(1,2)\in S_3$.
Ainsi $\Gamma_f$ est un cycle de 6 arêtes correspondant aux directions dans $\Ap_f$.
On dira que $s_f \in \wgt^{-1}(s)$ est le \emph{sommet de base} de $\Gamma_f$ (ou de $\Gamma^+_f$).

\begin{lemma}~
\label{lem:branch}
\begin{enumerate}
\item \label{lem:branch1}
La fibre $\wgt^{-1}(q) \subset \Gamma$ est un singleton.
\item \label{lem:branch2}
Pour tous $f,g\in \Fl$, l'intersection $\Gamma_f\cap \Gamma_g$ est soit connexe de longueur $\leq \pi$ (c'est-à-dire formée d'au plus 3 arêtes), soit égale à $\{q_f,s_f\}=\{q_g,s_g\}$.
En particulier tout cycle plongé dans $\Gamma_f \cup \Gamma_g$ est constitué d'au moins 6 arêtes.
\item \label{lem:branch3}
Chaque sommet $u\in \Gamma \setminus \wgt^{-1}\{q,s\}$ est incident à exactement une arête $e$ dont la projection $\wgt(e)$ sépare $\wgt(u)$ de $q$ dans $I_3$. De plus, si $u\in \Gamma^+_f$, alors $e$ est aussi dans $\Gamma^+_f$.
\end{enumerate}
\end{lemma}

\begin{proof}
\begin{enumerate}[wide]
\item
Pour $p'\geq p$, le poids $(p',p',1)$ appartient à chaque demi-espace admissible contenant $(p,p,1)$.
Ainsi, par le corollaire~\ref{cor:stabjump}, pour tout $f\in \Fl$ on a  $q_\id=f(q_\id)=q_f$.
Donc $q_f=q_g$ pour chaque $f,g\in \Fl$, comme attendu.

\item
Supposons que $\Gamma_f\cap \Gamma_g$ n'est ni $\{q_f,s_f\}$ ni $\{q_f\}$.
Alors il existe $\omega\in \{\id,\sigma\}$ avec $\Gamma_{f}\cap\Gamma^+_{g\omega}\neq \{q_f,s_f\},\{q_f\}$.
Ainsi on peut appliquer le corollaire~\ref{cor:intersection appartements} à $f'=\omega^{-1}g^{-1}f$, d'où l'existence d'une permutation $\sigma'\in S_n$ tel que $\Ap_{f'}\cap \Ap_{\id}=\Fix(f'\sigma')\cap\Ap_{\id}$.
Par la proposition~\ref{pro:lieu fixe}, l'ensemble $\Fix(f'\sigma')\cap\Ap_{\id}$ est une intersection d'hyperplans admissibles.
Donc $\Gamma_{f'}\cap \Gamma_{\id}=\Fix(f'\sigma')\cap\Gamma_{\id}$ est une intersection de chemins de longueur $\pi$.
En faisant agir $g\omega$ on obtient la même propriété pour $\Gamma_{f}\cap \Gamma_{g\omega}=\Gamma_{f}\cap \Gamma_{g}$.
Par~\ref{lem:branch1} l'intersection $\Gamma_{f}\cap \Gamma_{g}$ contient $q_{f}$ et on a supposé qu'elle était distincte de $\{q_f,s_f\}$, elle est donc connexe de longueur $\leq \pi$.

\item
Soit $f$ tel que $u\in \Gamma^+_f$, et considérons l'arête incidente à $\wgt(u)$ qui sépare ce sommet de $q$ dans $I_3$.
Sa préimage $e$ dans $\Gamma^+_f$ satisfait la condition annoncée.
Supposons que $e'$ est une autre arête  dans $\Gamma^+_g$ satisfaisant cette condition. Alors $\Gamma_f\cap \Gamma_g$ contient $q_f$ et $u$ mais ne contient pas $e$.
Cela contredit~\ref{lem:branch2}.
\qedhere
\end{enumerate}
\end{proof}

\begin{lemma} \label{lem:bruhat}
Supposons que $s_f=s_{f'}$ pour certains $f,f' \in \Fl$.
Alors il existe $f'' \in \Fl$ tel que $\Gamma^+_f,\Gamma^+_{f'}\subset \Gamma_{f''}$.
\end{lemma}

\begin{proof}
Par le corollaire~\ref{cor:intrin}\ref{cor:intrin2} on peut supposer $f=\id$.
Par la proposition~\ref{pro:stabilisateur}, on peut écrire $f'=hg$ avec $h\in M_\alpha, g\in L_\alpha$.
Observons que $s_\id$ est fixé par tout élément de $L_\alpha$, et également par hypothèse par $f' = hg$.
Ainsi on obtient que $h$ fixe $s_\id$, ce qui revient à dire que $h$ fixe $\nu_{\id,[p-1,p-1,1]}$.
Donc, par la proposition~\ref{pro:stabilisateur}, et au vu de la description de $M_\alpha$ donnée dans l'exemple~\ref{exple:stab}\ref{exple:stab4}, $h$ est un automorphisme triangulaire de la forme
$(x_1 + P(x_3), x_2 + Q(x_3), x_3 + d)$ avec $p-1 \ge \deg
P,\deg Q$.
Toujours par la proposition~\ref{pro:stabilisateur}, $h$ fixe toute valuation de $\Ap^+_\id$ dont le poids $(\alpha_1', \alpha_2', \alpha_3')$ satisfait l'inégalité $\alpha_2' \ge (p-1) \alpha_3'$.
En particulier $h$ fixe un voisinage de $\nu_{\id, [p,p,1]}$ dans $\Ap^+_\id$, et donc $h$ fixe les trois arêtes de $\Gamma^+_{\id}$.

Écrivons $g = b_1 \sigma b_2$ une décomposition de Bruhat de~$g$, où $b_1,b_2$ sont triangulaires supérieures, et $\sigma=(1,2)$ ou $\id$.
Alors $f'' = hb_1$ convient, en effet :
\begin{align*}
\Gamma^+_{\id}&=h\Gamma^+_{\id}=h\Gamma^+_{b_1}=\Gamma^+_{hb_1}, \text{ et}\\
\Gamma^+_{hg} &= hb_1 \sigma \Gamma^+_{b_2} = hb_1\sigma \Gamma^+_\id \subset hb_1 \Gamma_\id = \Gamma_{hb_1}. \qedhere
\end{align*}
\end{proof}

Nous pouvons maintenant terminer la preuve du premier cas :

\begin{proposition}
\label{pro:pp1}
Si $\alpha = (p,p,1)$ avec $p \ge 2$ un entier, alors $\Gamma$ est $\mathrm{\CAT(1)}$.
\end{proposition}

\begin{proof}
Soit $\gamma$ un cycle plongé dans $\Gamma$. Orientons chaque arête $e$ de $\gamma$ de manière que sa projection $\wgt(e)\subset I_3$ soit orientée vers $s$.
Par le lemme~\ref{lem:branch}\ref{lem:branch3}, le cycle $\gamma$ se décompose en un nombre pair de chemins orientés dont chacun se termine sur un sommet de base.
De plus, chaque tel chemin est contenu dans un $\Gamma^+_f$.
Si $\gamma$ est constitué d'au plus $5$ arêtes, alors il contient au plus 2 sommets de base.
Par le lemme~\ref{lem:bruhat}, il existe $f,g \in \Fl$ tels que ces deux sommets soient $s_f$, $s_g$, et $\gamma\subset \Gamma_f\cup \Gamma_g$.
Cela contredit le lemme~\ref{lem:branch}\ref{lem:branch2}.
\end{proof}

\subsubsection{Deuxième cas : \texorpdfstring{$\alpha_1 > \alpha_2=\alpha_3$}{a1 > a2 = a3}}
\label{sec:1>2=3}

Dans cette situation l'hypothèse $\mult(\alpha) \ge 2$ implique $[\alpha] = [p,1,1]$ pour un entier $p \ge 2$.
Par la remarque~\ref{rem:mp1_equations}, $[\alpha]$ est contenu dans $p + 2 \ge 4$ droites admissibles, dont seulement 3 sont principales.

Soit $\DD\subset \PW^+$ le demi-disque de centre $[\alpha]$ et rayon $\eps/4$ pour la distance $|\cdot,\cdot|$. Pour $f\in \Fl$, soient $\phi_f$ l'isométrie identifiant $\DD$ avec le demi-disque $\DD_f:=\overline{B}(\nu_{f,[\alpha]},\frac{\eps}{4})$ dans $\Ap_f^+$.
Ainsi le voisinage $\V$ discuté plus haut s'écrit $\V =\bigsqcup_{f\in \Fl} \DD_f/\sim$.
Noter que les droites admissibles non principales passant par $[\alpha]$ ne sont plus des droites pour la distance $|\cdot,\cdot|$.
En particulier~$\V$ n'est plus un cône, et pour nous ramener à cette situation nous allons utiliser le cadre décrit dans la section~\ref{sec:limit}.

Notons $c_1, \dots, c_{p+1}$ la trace sur le demi-disque $\DD$ des droites admissibles, sauf la droite principale $\alpha_2=\alpha_3$ (qui correspondrait au diamètre de $\DD$).
Noter que chacune de ces courbes $c_i$ admet $[\alpha]$ pour extrémité.
Dans la distance euclidienne $|\cdot,\cdot|$ ces courbes ne sont plus en général des segments de droites (précisément, ce sont encore des droites si et seulement si elles proviennent d'une droite principale), et on note $l_i$ le rayon de $\DD$ tangent à $c_i$ au point $[\alpha]$.
On note $C_{i} \subset \DD$ le convexe délimité par $c_{i}$, trace sur~$\DD$ du demi-espace admissible associé à $c_i$, et $L_{i} \subset \DD$ le secteur délimité par le rayon tangent $l_{i}$ et contenant $C_{i}$.
On numérote les $c_i$ de façon à ce que les $C_i$ forment une famille croissante de convexes.
En particulier, $c_1$ et $c_{p+1}$ sont les deux seuls segments de droites principales parmi les $c_i$.
Quitte à diminuer $\eps$, pour chaque $i=1,\ldots, p+1$ on peut supposer $c_i\cap l_{i-1}=[\alpha]$, où par convention $l_0$ est le rayon au bord de $\DD$ contenu dans chaque $C_i$.
Il sera aussi utile de noter $l_{p+2}$ l'autre rayon bordant $\DD$.
Observer que tout ceci correspond aux conventions de la section~\ref{sec:limit} et de la figure~\ref{fig:limit} en prenant $k = p+1$.

Soit $f, g \in \Fl$. Par le corollaire~\ref{cor:intersection chambres} et la proposition~\ref{pro:lieu fixe}, $\wgt(\DD_f\cap \DD_g)$ est ou bien un des convexes $C_i$, ou bien le demi-disque $\DD$ entier, ou bien son diamètre.
Ainsi on est ramené au cadre de la section~\ref{sec:limit}, où on a introduit le «quotient linéarisé» $\V_0$ de $\V$, qui est un cône sur un graphe métrique $\Gamma_0$.
En tant que graphe, $\Gamma_0$ est isomorphe au link de $\nu_{\id,[\alpha]}$ dans $\Xl$ vu comme un complexe polygonal déterminé par les droites admissibles.
Soit $e$ une arête de $\Gamma_0$ correspondant à une région délimitée par deux courbes de $\Xl$ envoyées par $\wgt$ sur $c_i,c_{i+1}$.
Alors la longueur de $e$ dans $\Gamma_0$ est égale à l'angle entre $c_i,c_{i+1}$ pour la distance $|\cdot,\cdot|$.
Par la proposition~\ref{pro:V est CAT0} (et à nouveau par \cite[II.3.14]{BH}) pour montrer que $\V$ est CAT(0) il suffit de montrer :

\begin{proposition}
\label{pro:p11}
Si $\alpha = (p,1,1)$ avec $p \ge 2$ un entier, alors $\Gamma_0$ est $\mathrm{\CAT(1)}$.
\end{proposition}

Soit $\Gamma^+_f \subset \Gamma_0$ le chemin correspondant au demi-disque $\DD_f\subset \Ap^+_f$.
Appelons \emph{sommet de base} de $\Gamma^+_f$, noté $s_f$, l'extrémité de $\Gamma^+_f$ correspondant à la direction vers $[1,1,1]$.
Notons $\Gamma_f=\Gamma^+_f\cup \Gamma^+_{f\sigma}$ pour $\sigma=(2,3)\in S_3$.

\begin{remark}
\label{rem:2pi/3}
Soit $s_f$ un point base.
Toute arête de $\Gamma_0$ issue de $s_f$ est de longueur $\pi/3$, car correspond après projection à l'angle entre les deux segments de droites principales $c_{p+1} = l_{p+1}$ et $l_{p+2}$.
\end{remark}

\begin{proof}[Preuve de la proposition~\ref{pro:p11}]
Notons que le lemme~\ref{lem:bruhat} reste vrai avec $\Gamma_0$ au lieu de $\Gamma$ : dans la preuve il suffit de remplacer $\sigma=(1,2)$ par $(2,3)$ et $h$ par un automorphisme de la forme $(x_1 + P(x_2,x_3), x_2+d, x_3 + d')$ où $p > \deg P$.
De même, l'analogue du lemme~\ref{lem:branch} reste valable.

La preuve est maintenant identique à la preuve de la proposition~\ref{pro:pp1}, excepté que c'est maintenant la remarque ~\ref{rem:2pi/3} qui permet d'affirmer que si un cycle $\gamma \subset \Gamma_0$ est de longueur $<2\pi$, alors il contient au plus 2 sommets de base.
\end{proof}

\subsubsection{Troisième cas : \texorpdfstring{$\alpha_1 > \alpha_2>\alpha_3$}{a1 > a2 > a3}}

Comme précédemment, soit $\Gamma_0$ le link de $[\alpha]$ dans $\Xl$ avec la structure polygonale déterminée par les droites admissibles.
On a une projection $\wgt\colon \Gamma_0\to S^1$ vers le cercle combinatoire correspondant au link de $[\alpha]$ dans $\PW^+$.
Nous munissons les arêtes de $S^1$ (et de $\Gamma_0$ via $\wgt^{-1}$) des longueurs venants des angles entre les droites admissibles (devenues courbes) relativement à la distance $|\cdot,\cdot|$.
Soit $q\in S^{1}$ (resp.\ $s\in S^1$) le point correspondant à la direction vers $[1,0,0]$ (resp.\ à la direction opposée, c'est-à-dire vers $[0,\alpha_2,\alpha_3]$).
Soit $I\subset S^1$ le chemin de $q$ à $s$, de longueur $\pi$, correspondant au demi-espace $\frac{\alpha'_2}{\alpha'_3}\geq \frac{\alpha_2}{\alpha_3}$.
Enfin, notons $e^q\subset S^1$ (resp. $e^s\subset S^1$) l'arête qui contient le chemin de longueur $\frac{\pi}{3}$ dans $I$ commençant en~$q$ (resp. en~$s$): voir figure~\ref{fig:1>2>3}.

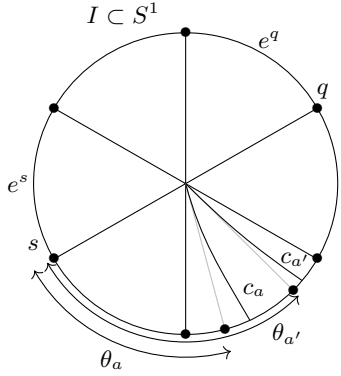
\begin{figure}[ht]
\[
\begin{tikzpicture}[scale = 2,font=\small]
\coordinate (0) at (0,0);
\foreach \angle in {30,-30,90,-90,150,-150,105,-75,135,-45}
{
\coordinate (\angle) at (\angle:1);
}
\draw (0) circle (1);
\node[above left,xshift=-2] at (-150) {$s$};
\node[above,xshift=2] at (30) {$q$};
\draw (-150)--(30) (-90)--(90) (-30)--(150);
\draw[white!75!black] (0)--(-75);
\draw (0) to[out=-75,in=120] node[auto,xshift=-3,pos=.9]{$c_a$} (-65:1);
\draw[white!75!black] (0)--(-45);
\draw (0) to[out=-45,in=145] node[auto,xshift=-9,pos=.95]{$c_{a'}$} (-40:1);
\draw pic[<->,draw=black,angle eccentricity=1.1,angle radius=2.1cm] {angle=-150--0---45};
\draw pic["$\theta_{a}$",<->,draw=black,angle eccentricity=1.1,angle radius=2.3cm] {angle=-150--0---75};
\node at (-56:1.2) {$\theta_{a'}$};
\node at (180:1.1) {$e^s$};
\node at (60:1.1) {$e^q$};
\node at (110:1.2) {$I \subset S^1$};
\foreach \angle in {30,-30,90,-90,150,-150,-45,-75}
{
\node at (\angle) {$\bullet$};
}
\end{tikzpicture}
\]
\caption{Les notations du cas $\alpha_1 > \alpha_2 > \alpha_3$ (dans la situation où $s$ et $q$ sont des sommets, et $e^s, e^q$ sont de longueur exactement $\pi/3$).}
\label{fig:1>2>3}
\end{figure}

Pour tout $f \in \Fl$, notons $\Gamma_f \subset \Gamma_0$ le cycle de longueur $2\pi$ induit par $\Ap_f^+$, et notons $ s_f,q_f,e^q_f\subset \Gamma_f$ les préimages respectives par $\wgt$ de $s,q,e^q$.
Grâce au corollaire~\ref{cor:intersection chambres} et à la proposition~\ref{pro:lieu fixe}, pour tout couple $f, g\in \Fl$, ou bien $\Gamma_f = \Gamma_g$, ou bien  $\Gamma_f\cap\Gamma_g$ est un chemin de longueur $\leq \pi$ contenant $e^q_f=e^q_g$.
En particulier, $\Gamma_f \cup \Gamma_g$ ne contient pas de cycle plongé de longueur $< 2\pi$.
Noter que $q_f$ et $s_f$ ne sont pas nécessairement des sommets (précisément, ce sont des sommets si et seulement si $\alpha_2$ est un multiple entier de $\alpha_3$), donc on appelle les $s_f$ les \emph{points de base}.
Similairement à la remarque~\ref{rem:2pi/3}, le chemin de longueur $\frac{2\pi}{3}$ dans $\Gamma_f$ centré en $s_f$ ne contient aucun sommet distinct de $s_f$ dans son intérieur.

Notons que par la proposition~\ref{pro:V est CAT0}, la proposition suivante entraîne que $\V$ est CAT(0), ce qui terminera la preuve de la proposition~\ref{pro:courbure negative}.
Ici, on représente $\V$ comme une union de demi-disques en coupant chaque disque $\overline{B}(\nu_{f,[\alpha]},\frac{\eps}{4})\subset \Ap^+_f$ le long du diamètre de direction $\frac{\alpha'_2}{\alpha'_3}=\frac{\alpha_2}{\alpha_3}$.

\begin{proposition}
\label{pro:123}
Si $\alpha_1 > \alpha_2 > \alpha_3$, alors $\Gamma_0$ est $\mathrm{\CAT(1)}$.
\end{proposition}

\begin{proof}
Supposons que $\gamma \subset \Gamma_0$ soit un cycle plongé de longueur $<2\pi$.
Orientons comme dans la preuve de la proposition~\ref{pro:pp1} chaque arête de $\gamma$ ne contenant ni un point de base ni un $q_f$ de manière que sa projection dans $S^1 \setminus q$ soit orientée vers $s$.
Nous divisons chaque arête de $\gamma$ contenant un point de base $s_f$ (resp.\ un $q_f$) en deux parties orientées vers $s_f$ (resp.\ dans le sens opposé à $q_f$).
Comme dans les cas précédents, $\gamma$ se décompose alors en un nombre pair de chemins orientés dont chacun se termine sur un point de base.
Chaque tel chemin est contenu dans un $\Gamma_f$ et contribue par au moins $\frac{\pi}{3}$ à la longueur de $\gamma$.
S'il n'y en a que 2, $\gamma\subset \Gamma_f\cup\Gamma_g$ pour certains $f,g \in \Fl$ et on a vu plus haut que c'était impossible.
Ainsi on peut supposer que $\gamma$ se décompose en 4 chemins se terminant sur 2 points de base $s_f \neq s_g$, pour certains $f,g \in \Fl$.
Observons également que $q\notin \wgt(\gamma)$, car sinon au moins 2 de ces chemins seraient de longueur~$\pi$.

Si $s$ est un sommet, vu que $\wgt^{-1}(e^q)$ est une seule arête, similairement au lemme~\ref{lem:branch}\ref{lem:branch3} on démontre l'énoncé suivant : il y a exactement une arête $e^s_f$ (resp.\ $e^s_g$) incidente à $s_f$ (resp.\ $s_g$) dont la projection sépare $s$ de $q$ dans $I$. En particulier $e^s_f\subset \Gamma_f, e^s_g\subset \Gamma_g$.
Ainsi, si $\gamma$ passe par~$e^s_f$, la projection $\gamma\to S^1$ est localement injective autour de $s_f$.
Puisque $q\notin \pi(\gamma)$ cela entraîne que $e^s_g$ est aussi contenue dans $\gamma$.
Par conséquent $\gamma\subset \Gamma_f\cup\Gamma_g$, contradiction.
Si $s$ n'est pas un sommet, les arêtes entières contenant $s_f,s_g$ sont évidemment empruntées par $\gamma$ et on obtient une contradiction de la même manière.

Ainsi on peut supposer que $s$ est un sommet et que ni $e^s_f$, ni $e^s_g$ ne sont contenues dans $\gamma$, autrement dit $\wgt(\gamma)\subset S^1 \setminus\mathrm{int}\ I$.
Puisque $s$ est un sommet, $\alpha$ satisfait une équation admissible de la forme $\alpha_2=p\alpha_3$, où $s$ correspond à la demi-droite principale $l$ vers $[0,p,1]$.
Par la remarque~\ref{rem:mp1}, $[\alpha]=[m,p,1]$ pour des entiers $m > p > 1$.
Écrivons $m=pq+r$, avec $0\leq r<p$.
Il existe donc $f_0, \dots, f_3 \in \Fl$ tels que
\[
\gamma = \gamma_0 \cup \gamma_1^{-1} \cup \gamma_2 \cup \gamma^{-1}_3
\]
où chaque $\gamma_i$ est un chemin dans $\Gamma_{f_i}$, orienté comme précédemment.
En particulier
\[
\wgt(\gamma_0\cap\gamma_1)=\wgt(\gamma_2\cap\gamma_3)=s.
\]
Par la remarque~\ref{rem:mp1_equations}, les points $\wgt(\gamma_1\cap \gamma_2),\wgt(\gamma_3\cap\gamma_0)$ correspondent aux directions vers
$[m-pa, 0,1]$ et $[m-pa', 0,1]$, pour certains entiers $0\leq a, a' \leq q$.
Soient $c_a,c_{a'} \subset \PW$ les demi-droites issues de $[\alpha]$ correspondant à ces directions.
Soit $\theta_a$ (resp. $\theta_{a'}$) l'angle entre $l$ et $c_a$ (resp. $l$ et $c_{a'}$) relativement à la distance $|\cdot,\cdot|$ (pour laquelle $c_a,c_{a'}$ deviennent des courbes).
Quitte à faire une permutation cyclique des $\gamma_i$ on peut supposer $a \ge a'$, ce qui équivaut à $\theta_a \le \theta_{a'}$.
La longueur de $\gamma$ étant $2 (\theta_a + \theta_{a'})$, pour obtenir une contradiction il suffit de montrer $\theta_a + \theta_{a'} \ge \pi$.
Nous allons obtenir cette inégalité sous la forme $ (\theta_a-\pi/3) + (\theta_{a'}-\pi/3) \ge \pi/3$ (voir figure~\ref{fig:1>2>3}).

En utilisant le corollaire~\ref{cor:intrin}\ref{cor:intrin2} on peut supposer $f_0 = \id$.
Par le lemme~\ref{lem:normal}\ref{lem:normal1}, on peut supposer aussi $f_1,f_2,f_3\in M_\alpha$.
D'après le lemme~\ref{lem:secteur commun}, on peut écrire
\begin{align*}
f_1=f_0^{-1}f_1 &\sim_\nu h_0,&
f_1^{-1}f_2 &\sim_\nu h_1,\\
f_2^{-1}f_3 &\sim_\nu h_2, &
f_3^{-1}=f_3^{-1}f_0 &\sim_\nu h_3,
\end{align*}
où les $h_i$ sont des automorphismes triangulaires de la forme suivante :
\begin{align*}
h_0 &= (x_1, x_2 + d' x_3^p, x_3),&
h_1 &= (x_1 + P_1(x_2, x_3), x_2, x_3), \\
h_2 &= (x_1, x_2 + d x_3^p, x_3), &
h_3 &= (x_1 + P_3(x_2, x_3), x_2, x_3),
\end{align*}
avec $P_1, P_3$ homogènes de degré $m$ relativement aux variables $x_2, x_3$ de poids respectifs $p, 1$.
Par le lemme~\ref{lem:normal}\ref{lem:normal2}, le groupe $N_\alpha$  est normal dans $M_\alpha$, donc $h_0h_1h_2h_3\sim_\nu\id$. Ainsi
\[
h_0h_1h_2h_3=(x_1 + P(x_2, x_3), x_2 + Q(x_3), x_3)
\]
avec $-\nu(P)<m, \deg Q=-\nu(Q)<p$. On a $Q(x_3)= (d+d')x_3^p$, donc $d+d'=0$.
Puisque $P_1,P_3$ sont homogènes de degré $m$, $P$ l'est aussi, donc $P=0$ et $P_1(x_2 + d x_3^p, x_3) = -P_3(x_2, x_3)$.
On écrit
\begin{align*}
P_3(x_2, x_3) = x_2^a (x_2 + d x_3^p)^b R_3(x_2, x_3) \\
P_1(x_2, x_3) = x_2^{a'} (x_2 - d x_3^p)^{b'} R_1(x_2, x_3)
\end{align*}
où $R_3$ est premier avec  $x_2$ et $x_2 + d x_3^p$, et $R_1$ est premier avec  $x_2$ et $x_2 - d x_3^p$.
Noter que par le lemme~\ref{lem:secteur commun}\ref{lem:secteur1}, les exposants $a$ et $a'$ correspondent bien aux entiers définissant les demi-droites $c_a$ et $c_{a'}$ introduites plus haut.
L'égalité $P_1(x_2 + d x_3^p, x_3) = -P_3(x_2, x_3)$ équivaut à
\[
(x_2 + d x_3^p)^{a'} x_2^{b'} R_1(x_2+ d x_3^p, x_3) = -x_2^a (x_2 + d x_3^p)^b R_3(x_2, x_3).
\]
Finalement $a' = b$, et $pa + pa' = p(a + b) \le m$.
On conclut grâce au lemme~\ref{lem:angle} que $ (\theta_a-\pi/3) + (\theta_{a'}-\pi/3) \ge \pi/3$, d'où la contradiction attendue.
Précisément, dans le cas d'égalité $p(a+a') = m$, les rayons $c_2, c_3$ du lemme~\ref{lem:angle} correspondent respectivement aux demi-droites $c_a, c_{a'}$, et on obtient $\theta_a = \pi/3 + \theta_{12}$, $\theta_{a'} = \pi/3 + \theta_{13}$, et $(\theta_a-\pi/3) + (\theta_{a'}-\pi/3) = \pi/3$.
Le cas d'une inégalité stricte $p(a+a') < m$ correspond à des angles $\theta_a$, $\theta_{a'}$ plus grands, et donc est encore plus favorable.
\end{proof}

Dans la preuve de la proposition~\ref{pro:V est CAT0} une bonne part des complications vient du fait que $d_\Xl$ n'est pas la distance de la structure polygonale de $\Xl$, ce qui nous oblige à d'abord linéariser la situation avant de pouvoir conclure en étudiant le link de chaque point.
Il semblerait \textit{a priori} naturel de vouloir garder la distance du simplexe, pour laquelle toutes les droites admissibles, principales ou non, sont bien des droites.
Nous concluons cette section avec un exemple qui montre le problème avec cette approche.

\begin{figure}[ht]
\begin{tikzpicture}[scale = 3.5,font=\scriptsize]
\coordinate [label=above left:$\trip{0}{1}{0}$] (010) at (0,0);
\coordinate [label=above right:$\trip{1}{0}{0}$] (100) at (1,0);
\coordinate [label=below:$\trip{0}{0}{1}$] (001) at (.5, -.86);
\coordinate (110) at ($ (010)!.5!(100) $) {};
\coordinate (101) at ($ (001)!.5!(100) $) {};
\coordinate (011) at ($ (001)!.5!(010) $) {};
\coordinate (111) at (intersection of 011--100 and 101--010);
\coordinate [label=below right:$\trip{m}{0}{1}$] (m01) at ($ (100)!1/11!(001) $) {};
\coordinate [label=left:$\trip{0}{p}{1}$] (0p1) at ($ (010)!1/3!(001) $) {};
\coordinate [label=above:$\trip{q}{1}{0}$] (q10) at ($ (100)!1/6!(010) $) {};
\coordinate [label={[below right, yshift=-1, xshift=-4]:$[\alpha]$}] (alpha) at (intersection of 0p1--100 and 010--m01);
\draw [white!50!black] (010)--(100)--(001)--(010);
\draw (0p1)--(100) (m01)--(010) (q10)--(001);
\draw [white!80!black] (011)--(100) (101)--(010) (110)--(001);
\draw pic[thick,draw=black,angle eccentricity=2.2,angle radius=.6cm] {angle=010--alpha--001};
\end{tikzpicture}
\caption{Un arc de longueur $2\pi/3$ pour la métrique $|\cdot,\cdot|$, mais de longueur $\pi/3 + \eps$ pour la métrique du simplexe.}
\label{fig:arc}
\end{figure}
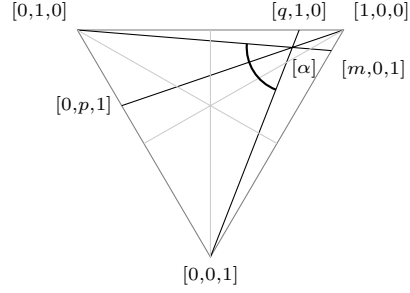

\begin{example}[figure~\ref{fig:arc}] \label{exple:angles}
Soient $p,q \ge 1$ deux entiers, et notons $m = pq$, $\alpha = (m,p,1)$, $\nu = \nu_{\id, [\alpha]}$.
Les automorphismes
\[
f = (x_1 + x_2^q, x_2, x_3) \quad \text{ et } \quad g = (x_1 + x_3^m, x_2,x_3)
\]
appartiennent au groupe $M_\alpha \subset \Stab(\nu)$, et commutent.
En parcourant quatre arcs de cercles centrés en $\nu$ dans les appartements successifs $\Ap_\id$, $\Ap_f$, $\Ap_{fg}$ et $\Ap_g$, où les poids varient entre les demi-droites principales dirigées vers $[0,0,1]$ et $[0,1,0]$, on obtient un cycle.
Pour la distance du graphe $\Gamma_0$ correspondant au link du voisinage linéarisé $\V_0$, chacun de ces arcs est de longueur $2\pi/3$, et donc le lacet est de longueur totale $8\pi/3 > 2\pi$.
En revanche, si l'on utilise la distance du simplexe, chacun de ces arcs est de longueur $\pi/3 + \eps$ avec $\eps$ tendant vers 0 quand $q$ tend vers l'infini, et ainsi on produit des lacets de longueur $4\pi/3 + 4\eps < 2\pi$.
On voit que munir chaque appartement de la distance du simplexe conduirait à une distance sur $\Xl$ qui ne serait plus de courbure négative ou nulle.
\end{example}

\section{Preuve des résultats principaux}
\label{sec:princ}

Nous obtenons d'abord le théorème~\ref{thm:main} annoncé dans l'introduction:

\begin{proof}[Preuve du théorème~\ref{thm:main}]
On sait déjà que l'espace métrique $\Xl_3$ est connexe par le corollaire~\ref{cor:connexe}, et complet par le lemme~\ref{lem:complet}.
De plus $\Xl_3$ est
simplement connexe par la proposition~\ref{pro:X3 1-connexe}, et à courbure négative ou nulle par la proposition~\ref{pro:courbure negative}.
Le résultat découle alors du théorème de Cartan--Hadamard (\cite[II.4.1(2)]{BH}).
\end{proof}

Pour obtenir la linéarisabilité des sous-groupes finis, nous utilisons le critère suivant, qui repose sur un argument de moyennisation.
Ici on profite de la structure vectorielle de $\K^n$ pour définir un endomorphisme par moyenne d'automorphismes (en général une telle moyenne n'a pas de raison d'être inversible).

\begin{lemma}[{\cite[Lemma 5.1]{BFL}}] \label{lem:abstract linearization}
Soit $\K$ un corps de caractéristique zéro, et $G$ un sous-groupe du groupe des bijections de $\K^n$ admettant une structure de produit semi-direct $G = M \rtimes L$ avec $L\subset\GL_n(\K)$.
Supposons que pour toute suite finie $m_1, \dots, m_r$ de $M$, la moyenne $\frac1r \sum_{i=1}^r m_i$ est encore dans $M$.
Alors tout sous-groupe fini de $G$ est conjugué par un élément de $M$ à un sous-groupe de $L$.
\end{lemma}

\begin{proof}[Preuve du corollaire~\ref{cor:main}]
En fait, on va montrer que chaque sous-groupe fini $F$
de $\Tame(\K^3)$ est conjugué par un élément de $\Tame(\K^3)$ (est pas seulement de $\Aut (\K^3)$) à un sous-groupe de $\GL_3(\K)$.

Le groupe $F\subset \Tame(\A^3)$ agit par isométries sur $\Xl_3$, qui est complet CAT(0) par
le théorème~\ref{thm:main}. Par \cite[II.2.8(1)]{BH},
$F$ admet (au moins) un point fixe global $\nu\in \Xl_3$, que l'on peut construire comme le «circumcenter» d'une orbite quelconque.
Par le corollaire~\ref{cor:intrin}\ref{cor:intrin2}, en conjuguant par un élément de $\Tt$ nous pouvons supposer que $\nu = \nu_{\id,[\alpha]} \in \Ap_\id^+$.
Par la proposition~\ref{pro:stabilisateur}, $F\subset \Stab(\nu) = M_\alpha \rtimes L_\alpha$. On constate immédiatement que le groupe d'automorphismes
triangulaires $M_\alpha$ est stable par moyenne, on peut donc conclure grâce au lemme~\ref{lem:abstract linearization}.
\end{proof}

\section{Appendice: autres propriétés de $\Xl$}

\subsection{Fidélité de l'action}

\begin{proposition} \label{pro:fidele}
L'action de $\TA$ sur $\Xl$ est fidèle.
\end{proposition}

\begin{proof}
Soit $f$ un élément non trivial de $\TA$.
On veut montrer que l'application induite par $f$ sur $\Xl$ n'est pas l'identité.
Supposons que $f$ agisse trivialement sur l'appartement standard $\Ap_\id$ (sinon il n'y a rien à montrer).
Alors par le corollaire~\ref{cor:stab chambre},
\[
f = (c_1 x_1 + t_1, c_2 x_2 + t_2, \dots).
\]
Comme $f$ est non trivial, en conjuguant par un élément du groupe symétrique, nous pouvons de plus supposer que l'on n'est pas dans le cas $c_1 = c_2 = 1$ et $t_1 = t_2 = 0$.

Soit $r \ge 2$ un entier premier avec la caractéristique de $\K$.
Considérons $g = (x_1 - x_2^r, x_2, \dots, x_n)$ et $\alpha \in \W$, et comparons $\nu = \nu_{g,\alpha}$ et $f \cdot \nu$ en évaluant ces valuations sur le polynôme $P = x_1 + x_2^r$.
On a
\begin{align*}
-\nu(x_1+x_2^r) &= \alpha_1;\\
-(f\cdot \nu)(x_1+ x_2^r)
&= -\nu(c_1x_1 + t_1 + (c_2 x_2 + t_2)^r) \\
&= -\nu_{\id, \alpha}(c_1x_1 -c_1x_2^r + t_1 + c_2^r x_2^r + rc_2^{r-1}t_2x_2^{r-1} + \dots + t_2^r) \\
&=
\begin{cases}
\alpha_1 & \text{si } c_2^r = c_1, t_2 = 0,\\
\max \{ \alpha_1, (r-1)\alpha_2 \} & \text{si } c_2^r = c_1, t_2 \neq 0 \\
\max \{ \alpha_1, r\alpha_2 \} & \text{sinon}.
\end{cases}
\end{align*}
Ainsi, sauf dans le premier cas, les deux valuations sont distinctes dès que $\alpha_2 > \alpha_1$, et dans ce cas leurs classes d'homothétie sont distinctes car par ailleurs $-\nu(x_2) = -(f\cdot \nu)(x_2) = \alpha_2$.

Considérons pour finir le cas où $c_2^r = c_1$ et $t_2 = 0$.
Si $c_1^r \neq c_2$, ou si $t_1 \neq 0$, on peut reproduire l'argument précédent après avoir conjugué par la transposition qui échange $x_1$ et $x_2$.
Sinon, on a $t_1 = t_2 = 0$, et $c_1^r = c_2$ et $c_2^r = c_1$ pour tout entier $r$ premier avec la caractéristique de $\K$.
Mais ceci implique $c_1 = c_2 = 1$ : si $\car \K \neq 2$, les équations $c_1^2 = c_2$, $c_2^2 = c_1$ plus une équation supplémentaire $c_1^r = c_2$ avec $r \equiv 1 \mod 3$ suffisent déjà, et si $\car \K = 2$ cela découle des deux équations $c_1^3 = c_2$, $c_1^5 = c_2$.
Dans tous les cas on obtient une contradiction, ce qui conclut la preuve.
\end{proof}

\subsection{Une application vers un complexe simplicial}

La preuve de la proposition~\ref{pro:X3 1-connexe} est une version algébrique d'une première approche topologique qui nous a finalement semblé plus délicate à mettre au point.
L'idée est de construire en toute dimension $n$ une application naturelle, candidate à être une équivalence d'homotopie, depuis $\Xl_n$ vers un complexe simplicial $\CComp_n$ étudié dans \cite{Lamy, LP}.
Nous indiquons rapidement la construction, qui mène à des questions intéressantes et qui met également mieux en lumière le rôle clé joué par la théorie des réductions de Shestakov--Umirbaev et Kuroda (encapsulée plus haut dans le théorème~\ref{thm:triangle}, et ici dans la simple connexité de $\CComp_3$).

Rappelons d'abord la construction du complexe simplicial qui était noté $\mathcal C_n$ dans \cite{Lamy}, mais que nous noterons ici $\CComp_n$ afin d'éviter toute confusion avec le complexe de groupes utilisé plus haut.
Pour tout $1 \le r \le n$, nous appelons \emph{$r$-uplet de composantes} un morphisme
\begin{align*}
f\colon\A^n &\to \A^r \\
x = (x_1, \ldots, x_n) &\mapsto \left( f_1(x), \ldots, f_r(x) \right)
\end{align*}
qui peut s'étendre en un automorphisme modéré $f =
(f_1,\ldots, f_n)$ de $\A^n$.
On définit $n$ types distincts de sommets, en considérant des $r$-uplets de composantes
 modulo composition par un automorphisme affine au but, $r = 1,
\dots, n$:
\begin{equation*}
[f_1, \dots,f_r] := A_r (f_1, \ldots, f_r) = \{ a \circ (f_1, \ldots, f_r) ; a \in A_r\}
\end{equation*}
où $A_r = \GL_r(\K) \ltimes \K^r$ est le groupe affine en dimension $r$.
Nous dirons que $[f_1, \ldots,f_r]$ est un sommet de \emph{type} $r$.
Ensuite, pour tout automorphisme modéré $(f_1,\ldots, f_n) \in \Tame(\A^n)$ nous
attachons un $(n-1)$-simplexe sur les sommets $[f_n],[f_{n-1},f_n],\ldots,[f_1, \ldots,f_n]$.
Cette définition ne dépend pas d'un choix de représentants, et produit
un complexe simplicial $\CComp = \CComp_n$ de dimension $(n-1)$ sur lequel le groupe modéré
agit
par isométries, via
\[
g\cdot [f_1, \ldots,f_r] := [f_1 \circ g^{-1},\ldots, f_r \circ g^{-1}].
\]

\begin{example}
Le complexe $\CComp_2$ est l'arbre de Bass--Serre du scindement $\Td=A_2\ast_{A_2\cap E_2}E_2$, où $A_2$ est le groupe affine et $E_2= \{(ax_1+P(x_2),bx_2+d); a,b\neq 0\}$.
En caractéristique nulle, le complexe $\CComp_3$ est le développement universel du triangle de groupes $A,B,C$  et de leurs intersections (voir théorème~\ref{thm:triangle}).
\end{example}

On paramétrise le simplexe $[x_n],[x_{n-1},x_n],\ldots,[x_1, \dots,x_n]$ de la manière suivante.
Soit $\Delta\subset \W \subset \R^n$ le $(n-1)$-simplexe de sommets
\[
v_1=(2,\ldots, 2,1), v_2=(2,\ldots,2, 1,1),\ldots, v_n=(1,\dots, 1,1).
\]
On définit $\iota\colon \Delta \to \CComp$ comme l'application affine envoyant chaque $v_r$ sur le  sommet de type $r$ $[x_{n-r+1}, \ldots,x_n]$.
Notons que $\Delta\subset \W^+$ s'identifie par projectivisation à un simplexe de $\PW^+$ que l'on note encore $\Delta$.

On définit maintenant une application $\PW^+\to \Delta$.
Soit $\alpha\in \W^+$. Nous définissons un nouveau poids $\alpha'\in \Delta$ en posant, pour tout $i = 1, \dots, n$,
\[
\alpha_i' = \min (2, \alpha_i / \alpha_{\min}),
\]
qui ne dépend que de $[\alpha]\in \PW^+$.
On définit $\pi\colon \Xl \to \CComp$ par $\pi(\nu_{f,[\alpha]})= f\cdot\iota(\alpha')$.
L'application $\pi$ envoie la chambre $\Ap^+_{\id}$ sur le simplexe $[x_n],[x_{n-1},x_n], \ldots,[x_1, \dots,x_n]$.
La figure~\ref{fig:r} illustre l'effet de $\pi$ dans le cas $n = 3$.

\begin{figure}[ht]
\[
\begin{tikzpicture}[scale = 9,font=\small]
\coordinate (010) at (0,0);
\coordinate [label=above:$\trip{1}{0}{0}$] (100) at (1,0);
\coordinate (001) at (.5, -.86);
\coordinate [label=above:$\trip{1}{1}{0}$](110) at ($ (010)!.5!(100) $) {};
\coordinate  (101) at ($ (001)!.5!(100) $) {};
\coordinate  (011) at ($ (001)!.5!(010) $) {};
\coordinate  (210) at ($ (100)!1/3!(010) $) {};
\coordinate  (m10) at ($ (100)!1/6!(010) $) {};
\coordinate (120) at ($ (010)!1/3!(100) $) {};
\coordinate (201) at ($ (100)!1/3!(001) $) {};
\coordinate  (102) at ($ (001)!1/3!(100) $) {};
\coordinate  (021) at ($ (010)!1/3!(001) $) {};
\coordinate  (012) at ($ (001)!1/3!(010) $) {};
\coordinate  (221) at (intersection of 001--110 and 100--021) {};
\coordinate  (211) at (intersection of 100--011 and 010--201) {};
\coordinate (432) at ($ (221)!.5!(211) $) {};
\node  (111) [typethree] at (intersection of 011--100 and 101--010) {};
\node  (221) [typeone] at (intersection of 001--110 and 100--021) {};
\node  (211) [typetwo] at (intersection of 100--011 and 010--201) {};
\draw (111)--(211)--(221)--(111);
\draw[mid arrow] (110)--(221);
\draw[mid arrow] (210)--(221);
\draw[mid arrow] (m10)--(221);
\draw[mid arrow] (100)--(211);
\draw[mid arrow] (100)--(221);
\draw[mid arrow] (100)--(432);
\draw[white!75!black] (110)--(100);
\node  (111) [label=left:$\trip{1}{1}{1}$, typethree] at (intersection of 011--100 and 101--010) {};
\node  (221) [label=left:$\trip{2}{2}{1}$, typeone] at (intersection of 001--110 and 100--021) {};
\node  (211) [label=below right:$\trip{2}{1}{1}$, typetwo] at (intersection of 100--011 and 010--201) {};
\end{tikzpicture}
\]
\caption{L'application $\pi$ en restriction à une chambre de $\Xl_3$.}
\label{fig:r}
\end{figure}
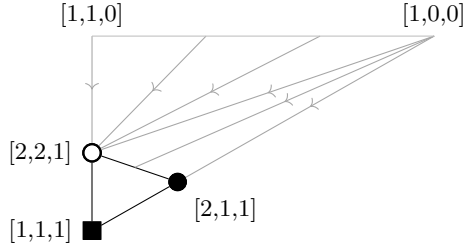

\begin{lemma}
Si $\nu_{g, [\alpha]} = \nu_{f, [\beta]}$, alors
$f\cdot \iota(\alpha') = g\cdot\iota(\beta')$.
En particulier $\pi$ est une application bien définie de $\Xl$
vers $\CComp$.
\end{lemma}

\begin{proof}
Par la proposition~\ref{pro:alpha intrinseque}, on peut supposer $g = \id$ et
$\beta=\alpha$. On a donc $\nu_{\id, [\alpha]} = \nu_{f,[\alpha]}$ où, d'après la proposition
\ref{pro:stabilisateur}, $f \in M_\alpha \rtimes L_\alpha$. On normalise $\alpha$ en demandant $\alpha_{\min} = 1$.

Il nous faut montrer que $f\cdot\iota(\alpha') = \iota(\alpha')$.
Pour cela il suffit de montrer que $f$ fixe chaque sommet d'un simplexe contenant $\iota(\alpha')$.
Considérons $(i_1, \ldots, i_k)$ la collection des indices $1\le i \le n$ tels que
$\alpha_{i-1}' > \alpha_i'$, où par convention, $\alpha'_{0} = 2$.
Remarquons que $\alpha'$ est une combinaison convexe des $v_{i_j}$, ainsi $\iota(\alpha')$ est contenu dans le simplexe de sommets $v_{i_j}$.

Pour chaque $j$ les coordonnées $f_{i_j}, \dots, f_n$ de $f$ correspondent à un
automorphisme en les variables $x_{i_j}, \dots, x_n$, puisque par
définition de $i_j$ on a $\alpha'_{i_j -1} > \alpha'_{i_j}$, et donc aussi $\alpha_{i_j -1} > \alpha_{i_j}$.
De plus cet automorphisme est affine, car $2 > \alpha'_{i_j} = \alpha_{i_j}$.
Ceci donne l'égalité centrale dans
\[
f^{-1}\cdot\iota(v_{i_j}) = [f_{i_j}, \dots, f_n] = [x_{i_j}, \dots, x_n] =\iota(v_{i_j}). \qedhere
\]
\end{proof}

\begin{remark}\label{rem:centre}
Pour tout $f = (f_1, \dots, f_n) \in \TA$, l'application $\pi$ induit une bijection locale entre un voisinage de la  valuation $\nu_{f^{-1}, [1, \dots, 1]}$ dans $\Xl$ et un voisinage du sommet de type $n$ correspondant $[f_1, \dots, f_n] \in \CComp$.

Suivant \cite{Lamy}, rappelons que les sommets de type $i$ à distance 1 de $[f_1, \dots, f_n]$ sont en bijection avec les sous-espaces vectoriels de dimension $i$ de
\[\text{vect}(f_1, \dots, f_n) \subset \K[x_1, \dots, x_n],\] et de même les simplexes contenant $[f_1, \dots, f_n]$ correspondent aux drapeaux dans cet espace.
On obtient ainsi une preuve alternative du fait que le link dans $\Xl_n$ d'une valuation de poids $[1, \dots, 1]$ est isomorphe à l'immeuble sphérique standard de $\GL_n(\K)$ (lemme~\ref{lem:GL}).
\end{remark}

\begin{question}
Les fibres de $\pi$ sont-elles contractiles? Si oui, cela entraîne-t-il que $\pi$ est une équivalence d'homotopie?
Une réponse positive donnerait une autre preuve de la proposition~\ref{pro:X3 1-connexe} grâce à la simple connexité de $\CComp_3$ \cite[Proposition 5.7]{Lamy}.
De plus, une fois le théorème~\ref{thm:main} établi, cela donnerait une autre preuve du fait que $\CComp_3$ est contractile (voir \cite[Theorem~A]{LP}).
\end{question}

\subsection{Comparaison avec la notion classique d'immeuble}
\label{sec:appart}

Nous justifions ici la remarque mentionnée dans l'introduction concernant le fait que le système d'appartements $\{ \Ap_f \mid f \in \TA$\} ne fait pas de $\Xl$ la réalisation de Davis d'un immeuble \cite[Definitions 4.1, 12.65]{AB}.

\begin{lemma} \label{lem:not a building}
Soit $f \in \TA$, et $[\alpha] \in \PW$ distinct de $[1, \ldots, 1]$ et de multiplicité au moins~1.
Alors tout voisinage de $\nu_{f, [\alpha]}$ contient deux points qui n'appartiennent pas à un même appartement.
\end{lemma}

\begin{proof}
Par le corollaire~\ref{cor:intrin}\ref{cor:intrin2}, sans perte de généralité on peut supposer $f = \id$.
Notons $\delta > 0$ le minimum des distances $|\sigma_1([\alpha]),\sigma_2([\alpha])|$ entre deux points distincts de l'orbite de $[\alpha]$ pour l'action de $S_n$ sur~$\PW$.
Noter que $\delta$ est bien défini car l'hypothèse $[\alpha] \neq [1, \dots, 1]$ assure que cette orbite contient au moins 2 points.
Un réel $r > 0$ étant fixé, considérons un point $\alpha'$ du segment ouvert reliant $\alpha$ à $[1, \ldots, 1]$ tel que $|\alpha, \alpha'| \le r$.
Alors $\alpha, \alpha'$ ont les mêmes symétries, c'est-à-dire que pour tout $\sigma \in S_n$, $\sigma(\alpha) = \alpha$ si et seulement si $\sigma(\alpha') = \alpha'$.
De plus il existe un demi-espace admissible contenant $\alpha$ mais pas $\alpha'$, et donc par la remarque~\ref{rem:lieu fixe} il existe  $g \in \Stab(\nu_{\id,[\alpha]})$ ne fixant pas~$\nu_{\id,[\alpha']}$.
Supposons qu'il existe un appartement $\Ap_h$ contenant $\nu_{g,[\alpha']}$ et $\nu_{\id,[\alpha']}$.
Alors on devrait avoir des permutations $\sigma_1, \sigma_2 \in S_n$ telles que
\[
\nu_{h, [\sigma_1(\alpha')]} = \nu_{g,[\alpha']} \neq \nu_{\id, [\alpha']} = \nu_{h, [\sigma_2(\alpha')]},
\]
donc $\sigma_1(\alpha') \neq \sigma_2(\alpha')$.
Ceci mène à une contradiction dès que $r< \delta/4$, car en utilisant la remarque~\ref{rem:plonge} on obtient :
\begin{align*}
2r
&\ge  d_\Xl (\nu_{h, [\sigma_1(\alpha')]}, \nu_{\id, [\alpha]}) + d_\Xl ( \nu_{\id, [\alpha]}, \nu_{h, [\sigma_2(\alpha')]})
\ge d_\Xl (\nu_{h, [\sigma_1(\alpha')]}, \nu_{h, [\sigma_2(\alpha')]}) \\
&= | [\sigma_1(\alpha')],  [\sigma_2(\alpha')] |
\ge | [\sigma_1(\alpha)],  [\sigma_2(\alpha)] | - | [\sigma_1(\alpha')],  [\sigma_1(\alpha)] | - | [\sigma_2(\alpha')],  [\sigma_2(\alpha)] | \\
&\ge \delta - 2r. \qedhere
\end{align*}
\end{proof}

Par contraste, par le lemme~\ref{lem:GL}, le link d'un point de poids $[1,\dots,1]$ est un immeuble sphérique.

De même, pour $n=3$ on pourrait démontrer que le link de chaque valuation de poids $[p,p,1]$ est un «hexagone généralisé», c'est-à-dire un graphe de diamètre 3 et de systole 6.
En particulier ce link est un immeuble sphérique pour le système d'appartements
constitué de tous les hexagones (dont certains ne sont pas la trace d'un appartement $\Ap_f$).
Notre preuve de ce fait étant un peu longue et calculatoire nous l’omettons, mais nous souhaitions le mentionner comme motivation pour la remarque suivante.

\begin{remark}
\label{rem:nonimmeuble}
On peut se demander si $\Xl$ pourrait devenir la réalisation de Davis d'un immeuble après avoir agrandi le système d'appartements.
Pour $n=2$ c'est vrai, car $\Xl_2$ et un arbre complet sans feuilles (voir la section~\ref{sec:arbre}).
La réponse est par contre toujours négative pour $n\geq 3$.
En effet, dans la réalisation de Davis d'un immeuble euclidien de dimension $\geq 2$, par deux points passe une copie de $\R^2$ isométriquement plongée.
En particulier, pour tous points $\nu,\nu',\nu''$ et $\delta >0$ avec $|\nu,\nu'|=|\nu,\nu''|=\delta,|\nu',\nu''|=2\delta$, il existe un point $\nu^\perp$ avec $|\nu,\nu^\perp|=\delta, |\nu',\nu^\perp|=|\nu'',\nu^\perp|=\sqrt{2}\delta$.

Cependant, prenons $\nu=\nu_{\id, [\alpha]}\in \Xl$ avec $[\alpha]\in \PW$ de multiplicité $1$ dans un hyperplan admissible $c$ qui ne passe par aucun sommet de $\PW$ (et qui par conséquent ne contient aucune droite dans la métrique $|\cdot,\cdot|$).
Soit $\mathbf{n}$ le vecteur normal à $c$ en $[\alpha]$ orienté vers l'extérieur du demi-espace admissible $L$ délimité par $c$, dans la métrique $|\cdot,\cdot|$.
Pour $\beta=(\beta_i)$ avec $\beta_i=\log \alpha_i$ posons $\beta'=\beta+\delta\mathbf{n}$ et $\alpha'_i=\exp \beta_i'$, pour $\delta = \frac{\eps}4$ venant du lemme~\ref{lem:boule}\ref{boule_b}.
Soit $\nu'=\nu_{\id, [\alpha']}$ et $\nu''=\nu_{g,[\alpha']}$ pour un $g$ qui fixe $\nu$ mais pas $\nu'$ (voir remarque~\ref{rem:lieu fixe}).
Puisque $L$ est convexe, $|\nu',\nu''|=2\delta$.
Pourtant, l'intersection de la sphère de rayon $\sqrt{2}\delta$ centrée en $\nu'$ (resp.\ en $\nu''$) avec la sphère de rayon $\delta$ centrée en $\nu$ est contenue dans $\Ap_\id \setminus \Ap_g$ (resp.\ $\Ap_g \setminus \Ap_\id$).
Ainsi un $\nu^\perp$ comme plus haut n'existe pas.
\end{remark}

\bibliographystyle{myalpha}
\bibliography{biblio}

\end{document}